%% file: NEWMAIN.tex
\documentclass[11pt, dvipsnames]{amsart}

\input{NEWPREAMBLE}
\title{\textsc{A model structure for cartesian \texorpdfstring{$2$}{2}-fibrations}}

\author[C. Bardomiano]{Cesar Bardomiano Martinez}
\address{Department of Mathematics, Johns Hopkins University, Baltimore MD, 21218, USA}
\email{cbardom1@jhu.edu}

\author[J. K. Nickel]{Jana K. Nickel}
\address{{Department of Mathematics, University of Hamburg, Bundesstrasse 55, 20146 Hamburg, Germany}}
\email{jana@nickel-math.com}

\author[M. Sarazola]{Maru Sarazola}
\address{{School of Mathematics, University of Minnesota, Minneapolis MN, 55415, USA}}
\email{maru@umn.edu}

\author[D. Teixeira]{Daniel Teixeira}
\address{{Department of Mathematics \& Statistics,
Dalhousie University, Halifax, Nova Scotia}}
\email{daniel.teixeira@dal.ca}

\author[S. Toro]{Santiago Toro Oquendo}
\address{{Univ. Artois, \textsc{UR 2462}, Laboratoire de Mathématiques de Lens (LML), F-62300 Lens, France}}
\email{santiago.torooquendo@univ-artois.fr}

\author[P. Verdugo]{Paula Verdugo}
\address{{Max Planck Institute for Mathematics, Bonn NRW, 53111, Germany}}
\email{verdugo@mpim-bonn.mpg.de}

\begin{document}

%----------------------------------------------------------------------------------------
% Abstract
%----------------------------------------------------------------------------------------

\begin{abstract} 
  Cartesian 2-fibrations provide a way to understand indexed categories, but their classical ``straightening'' construction requires several layers of weak coherence data. This paper develops a homotopical framework that replaces much of this bookkeeping with a fully strict model. By using marked 2-categories to record the  cartesian morphisms and 2-cells, we construct a model structure whose fibrant objects are precisely the cartesian 2-fibrations over a fixed 2-category $\cC$. We then show that the marked Grothendieck construction identifies these 2-fibrations, up to weak equivalence, with strict 2-functors from $\cC$ into $\twocat$. 
  As an additional contribution, we construct localizations of 2-categories that simultaneously invert selected morphisms and 2-cells.
\end{abstract}

\maketitle

%========================================================================================
% DOCUMENT SECTIONS
%========================================================================================

\setcounter{tocdepth}{1}
\tableofcontents

\section{Introduction}
The theory of cartesian fibrations was first developed by Grothendieck \cite{grothendieckdescent} to describe notions of descent in algebraic geometry. Since then, fibrations have established themselves as a particularly useful notion, not only within the foundations of category theory \cite{street1, street2, benabou} but also in applications ranging from \v{C}ech cohomology \cite{graydescent}, non-abelian cohomology \cite{giraud1,giraud2} and six functor formalisms \cite{Cisinski-Deglise2019}, to categorical logic \cite{catlogic1,lawverelogic,catlogic2} and type theory \cite{typetheory}. 

Our main connection to the world of fibrations is given by the Grothendieck construction \cite{grothendieck}, which plays a central role in category theory: it relates fibrations and indexed categories, providing a powerful framework for understanding how categories ``vary'' over a fixed base. To obtain all fibrations over a fixed category $\mathscr{C}$, however, we are forced to work with \emph{pseudo-functors} $\mathscr{C}^\op\to\underline{\Cat}$, where $\underline{\Cat}$ is considered as a 2-category.
We then get an equivalence of categories   
\[\textstyle\int_\mathscr{C}\colon \mathrm{Ps}[\mathscr{C}^\op,\underline{\Cat}]\to\mathrm{Fib}_{/\mathscr{C}}.\]
 Unfortunately, this comes at a cost, as the passage from strict to pseudo-functors increases the complexity: pseudo-functors only preserve identities and composites up to isomorphism, so one must now keep track of  this 2-dimensional coherence data, resulting in a significant increase in the required bookkeeping.

 Going up one dimension, we can fix a 2-category $\cC$  and get a notion of \emph{2-fibration} over $\cC$ 
making full use of the available morphisms and 2-cells 
\cite{grothcat1,Buckley2014,grothcat3}. 
Although there exists a corresponding Grothendieck construction for 2-functors, if one wishes to recover all 2-fibrations, the situation is much more subtle than the 1-dimensional case. Indeed, what we obtain is a (weak) triequivalence of tricategories
\begingroup\renewcommand{\theequation}{\thesection.\arabic{equation}}%
\begin{equation}\label{eq.intro} \textstyle\int_\cC\colon \mathrm{Ps}[\cC^\coop,\mathrm{Gray}]\to\mathrm{2Fib}_{/\cC}\end{equation} as proved in \cite[Theorem 3.3.12, Remark 3.3.13]{Buckley2014}.
\endgroup

This correspondence introduces several layers of complexity:
\begin{enumerate}[leftmargin=*]
\item\label{one} We no longer obtain an equivalence of categories, but rather a triequivalence of tricategories.
    \item\label{two} Instead of functors whose codomain is the 3-category $\underline{\twocat}$ of strict 2-categories,  2-functors, 2-natural transformations and modifications as one might initially expect, the codomain required for the correspondence is the \emph{semi-strict} 3-category $\mathrm{Gray}$ of 2-categories,  2-functors,  \emph{pseudo}-natural transformations and modifications \cite{graycats}.  Here, composition of cells is strictly associative and unital, but the interchange law holds only up to coherent isomorphism.
    \item\label{three} As expected, it is not enough to consider strict 3-functors $\cC^\coop\to\mathrm{Gray}$ and strict natural transformations, and one must instead use a certain class of pseudo 3-functors and pseudo-natural transformations.
\end{enumerate}

As category theory has evolved to encompass higher-dimensional structures, nowadays many of these constructions are homotopically flavored, with model-categorical techniques playing a central role. Indeed, the advancements from the past several decades have taught us that there is much to gain by using homotopical tools for categorical purposes. For instance, the cofibrant replacements in the canonical model structure on $\twocat$ provide a convenient method for the strictification of pseudo-functors \cite{Lack2004}. It is then natural to wonder: is it possible to use homotopical methods to ``strictify'' the weak data  involved in the correspondence described above? This paper provides a positive answer.

The first step in this direction is to construct a homotopical structure on the category of strict 2-functors over $\cC$ that models 2-fibrations. Upon inspection, this endeavor faces an immediate technical obstacle: we cannot encode the data of a 2-fibration $P$ as a lifting property of the 2-functor $P$ itself. A similar obstacle was encountered by the third author and Moser when they carried out an analogous program for the 1-dimensional Grothendieck construction \cite{mosersarazola2023model}. There, the solution was to consider \emph{marked categories}, which allow us to effectively ``mark'' the cartesian morphisms. In this 2-dimensional context, we use  \emph{marked 2-categories},  where one records markings on both morphisms and 2-cells.

The notion of 2-fibration $P\colon\cP\to\cC$ is then captured by that of a \emph{cart-marked 2-fibration}: a 2-functor between marked 2-categories $P\colon \markedscat{\cP}\to (\cC,\operatorname{mor}\cC,\operatorname{2cell}\cC)$ where the sets $E^1_\cP$ and $E^2_\cP$ of marked morphisms and 2-cells consist precisely of the $P$-cartesian morphisms and 2-cells. Letting $\twocat^+$ denote the category of marked 2-categories and 2-functors that preserve the markings, we prove the following result.

\begin{theoremA}[\cref{Thm:the ms}]\label{theoremA}
There is a combinatorial model structure on the category $\twocat^{+}_{/ \cC}$ of marked $2$-categories over $(\cC,\operatorname{mor}\cC,\operatorname{2cell}\cC)$, whose
	\begin{itemize}
    \item 
		fibrant objects are the cart-marked $2$-fibrations,
		\item 
		cofibrations are the underlying cofibrations in $\twocat$, 
		\item 
		weak equivalences between fibrant objects are the biequivalences, and
		\item fibrations between fibrant objects are the equifibrations, i.e.\ the fibrations in $\twocat$.
	\end{itemize}
\end{theoremA}

The Grothendieck construction admits a marked version 
\[
  {\textstyle\int}_\cC^+\colon 2\Fun[\cC^{\coop},\ttwocat]
  \longrightarrow \twocat^{+}_{/ \cC},
\]
which sends a 2-functor to its Grothendieck construction, marked by the cartesian morphisms and $2$-cells. Interestingly, this is not a right Quillen functor, since it is not a right adjoint, as we prove in \cref{subsec:construction_weak_left_adjoint}. %Nevertheless, $\int_\cC^+$ preserves and reflects all weak equivalences, and gives the desired comparison at the level of homotopy categories. 
% \begin{theoremA}[\cref{mainthm}]
% The marked Grothendieck construction
% \[{\textstyle\int}_\cC^+\colon 2\Fun[\cC^{\coop},\ttwocat] \to\twocat^{+}_{/ \cC}\] 
% induces an equivalence between the homotopy categories obtained from the enriched projective model structure and the model structure of \cref{theoremA}.
% \end{theoremA}
Nevertheless, $\int_\cC^+$ preserves and reflects all weak equivalences, and gives the desired homotopical comparison. 
\begin{theoremA}[\cref{mainthm,eq.infty}]
The marked Grothendieck construction
\[{\textstyle\int}_\cC^+\colon 2\Fun[\cC^{\coop},\ttwocat] \to\twocat^{+}_{/ \cC}\] 
induces an equivalence between the homotopy categories obtained from the enriched projective model structure and the model structure of \cref{theoremA}, and moreover, an equivalence of underlying $(\infty,1)$-categories.
\end{theoremA}

Introducing these homotopical tools significantly reduces the complexity of the categorical structures involved. Indeed, comparing to the subtleties outlined above for the categorical statement  \eqref{eq.intro}, we see that (\ref{one}) is now an equivalence of homotopy (1-)categories instead of a weak 3-dimensional equivalence; (\ref{two}) now has the expected codomain $\ttwocat$ as opposed to $\mathrm{Gray}$; and (\ref{three}) needs only consider strict 2-functors and strict 2-natural transformations instead of their pseudo counterparts. In other words, our techniques allow us to remain in a fully strict setting; the required higher-dimensional information is then encoded homotopically.

In the 1-dimensional case, the homotopy inverse functor uses a localization process which inverts the marked morphisms. A similar construction is required for the 2-dimensional case as well, where all marked morphisms and 2-cells are appropriately inverted. As the localizations for 2-categories that we could find in the literature only deal with inverted morphisms \cite{local1,pronk1996etendues,local3}, we establish in \cref{subsec.localization} a process that localizes a 2-category at sets of morphisms and 2-cells that may be of independent interest. 
\begin{theoremA}
      For any small 2-category $\cC$ and sets of morphisms and 2-cells $\cW_1,\cW_2$ in $\cC$, there is a localization 2-functor $\cC\to\cC[\cW_1^{-1},\cW_2^{-1}]$ universal among 2-functors sending all morphisms in $\cW_1$ to equivalences and all 2-cells in $\cW_2$ to isomorphisms.
\end{theoremA}
We note that such a localization, along with its universal property, appears in \cite[Definition 3.1]{vazquez2021three}; however, we could not find a study of when the localization exists, or how to construct it, prior to our work.  

We have not addressed in this paper the expected interactions between our results and the $(\infty,2)$-categorical setting. We expect the equivalence above to be appropriately related to the $(\infty,2)$-categorical straightening-unstraightening equivalence, and that this would identify the homotopy theory of cartesian $2$-fibrations constructed here with a reflective localization of the
homotopy theory of $\infty$-bicategories fibred over $\infty$-bicategories \cite{garcia2CartesianFibrationsModel2021, abellan2CartesianFibrationsII2023, gagnaEquivalenceAllModels2022}. We leave the precise formulation and proof of such a comparison to subsequent work.

\subsection*{Acknowledgements} 
	This collaboration was made possible thanks to the Matem\'aticas en el Cono Sur II workshop in 2024, and the authors are grateful to the Universidad de la República and to the workshop organizers for the hospitality and excellent working environment provided during the workshop week in Montevideo. During the realization of this work, the third author was partially supported by an NSF grant DMS-2506116.  The fourth author was partially supported by the Simons Collaboration for Global Categorical Symmetries and by the Nova Scotia Graduate Scholarship program. The fifth author gratefully acknowledges the support of the Laboratoire Angevin de Recherche en Mathématiques (UMR 6093) for partially funding a visit to Universidad de la República, Uruguay. 

\subsection*{Declaration of generative AI and AI-assisted technologies} A final draft was reviewed by Claude Opus 5, and it suggested \cref{eq.infty} together with a partially incorrect proof. We include in this work this result and a modified version of the proof. This is the only result where AI was used.

\section{Preliminaries}\label{sec:preliminaries}
	In this section, we introduce our notation for several basic notions in ($2$-)category theory. We also review some aspects of cartesian fibrations in the setting of $1$- and $2$-categories, and recall some key features of the canonical model structure on the category of $2$-categories and 2-functors, as well as the corresponding projective model structure. 

    \subsection{Notation and terminology}
		We use $\cat$ to denote the $1$-category of small categories and functors. Similarly, we use $\twocat$ to denote the category of small $2$-categories and strict 2-functors, and $\ttwocat$ for the $2$-category of small $2$-categories, 2-functors, and $2$-natural transformations. 

		We denote the initial category by $\emptyset$, and use $[n]$ for the category determined by the poset $\{0<1<\dots<n\}$; in particular, $[0]$ is the terminal category. We use the same notation for the corresponding $2$-categories having only trivial $2$-cells.

		To help with readability, we sometimes use $f.g$ and $\alpha.\beta$ for the composition of morphisms and the vertical composition of $2$-cells. We write $f\cdot \alpha$ or $\alpha\cdot f$ to denote the whiskering of a morphism $f$ and a $2$-cell $\alpha$. We use the same notation for the horizontal composite of $2$-cells.

		Throughout this paper, we will use several standard 1- and $2$-categories encoding the notions of isomorphism, equivalence, and adjoint equivalence, which we now recall.

		\begin{defn}\label{walkingiso}
			The \emph{walking isomorphism}  is the category $\mathbb{I}$ generated by the data 
			\[\begin{tikzcd}
				0\rar[shift left,"i"] & 1\lar[shift left,"i^{-1}"]
			\end{tikzcd}\]
			subject to the relations $i^{-1}  i = \id_0$  and $i i^{-1} = \id_1$.
			\end{defn}

			\begin{defn}\label{defn:walkingequiv}
				The \emph{walking equivalence} is the $2$-category $E$ generated by the data 
			\[\begin{tikzcd}
				0\rar[shift left,"i"] & 1\lar[shift left,"i^{-1}"]
			\end{tikzcd}\]
				together with the data of invertible 2-cells $\eta\colon\id_0\xRightarrow{\cong} i^{-1} i$ and $\epsilon\colon i  i^{-1} \xRightarrow{\cong} \id_1$. 
			\end{defn}
		
		\begin{defn}\label{defn:walkingadjointequiv}
			The \emph{walking adjoint equivalence} $\walkingadjoint$ is the $2$-category obtained from $E$ by imposing the triangle relations, meaning
			\[ (\epsilon \cdot i) .  (i \cdot \eta) = \id_{i} \quad \text{and} \quad (i^{-1} \cdot \epsilon) . (\eta \cdot i^{-1}) = \id_{i^{-1}}. \]    \end{defn}

		Finally, it will also be convenient to recall the following.

		\begin{defn}\label{par:suspension}
			Given a category $\cC$, its \textit{suspension} $\Sigma\cC$ is the $2$-category with two objects $x,y$, and with 
			hom-categories given as follows:
			\[ \Sigma\cC(x,y)=\cC, \quad \Sigma \cC(x,x) = \Sigma \cC(y,y) = [0], \quad \text{ and } \quad \Sigma \cC(y,x)= \emptyset.\]
		\end{defn}

	\subsection{\texorpdfstring{$2$}{2}-cartesian fibrations} Although we assume the reader has a working familiarity with cartesian maps and fibrations in the 1-categorical context, we recall these notions here for their convenience.

	\begin{defn}
		Let $P\colon\cP \to \cC$ be a functor, and $h\colon Px\to Py$ a morphism in $\cC$. We say a morphism $\widehat{h}\colon x\to y$ in $\cP$ is a $P$-\emph{lift} of $h$ if $P\widehat{h}=h$.
	\end{defn}

	\begin{defn}\label{defn:cartesian1cat}
		Let $P\colon\cP \to \cC$ be a functor. A morphism $f\colon x \to y$ in $\cP$ is $P$-\emph{cartesian} if, for every morphism $g\colon z \to y$ in $\cP$ and every morphism $h\colon Pz \to Px$ in $\cC$ such that $ Pg = (Pf) h$, there exists a unique $P$-lift $\widehat{h}\colon z \to x$ of $h$ such that $g = \widehat{h} f$.
	\end{defn}

	\begin{defn}
		A functor $P\colon\cP \to \cC$ is a \emph{cartesian fibration} (or \emph{Grothendieck fibration}) if for every object $y \in \cP$ and every morphism $f\colon x' \to Py$ in $\cC$ there is a  $P$-cartesian lift $\widehat{f}\colon x \to y$ of $f$. 
	\end{defn}

    \begin{prop}[{\cite[\nopp Proposition 3.4]{vistoli2007notesgrothendiecktopologiesfibered}}]\label{prop:cartesian_isomorphism_1cats}
		Let $P\colon\cP \to \cC$ be a functor. A morphism $f$ in $\cP$ is an isomorphism if and only if $f$ is cartesian and $Pf$ is an isomorphism in $\cC$.
	\end{prop}

    The notions of cartesian maps and fibrations have $2$-categorical analogues, which are the main protagonists of this paper. We recall some of the central definitions and results below, all of which appear in \cite{Buckley2014}.

		\begin{defn}\label{Def:cartesian_1cells}
			Let  $P\colon  \cP \to \cC$ be a 2-functor. A morphism $f\colon x\to y$ in $\cP$ is \emph{cartesian} (or \emph{$P$-cartesian}) if the following are satisfied:

			\begin{enumerate}[label=\stlabel{Def:cartesian_1cells},ref=\arabic*]
			\item \label{Def:cartesian_1cells.1}
            For all morphisms $g\colon z \to y$  in $\cP$ and $h\colon Pz \to Px$  in $\cC$ together with an invertible 2-cell $\alpha\colon (Pf).h \xRightarrow{\cong} Pg$
			\[\begin{tikzcd}[row sep = small]
				z &&&& Pz \\
				&& {} && {} \\
				x && y && Px && Py
				\arrow["g", from=1-1, to=3-3]
				\arrow["h"', from=1-5, to=3-5]
				\arrow[""{name=0, anchor=center, inner sep=0}, "Pg", from=1-5, to=3-7]
				\arrow["P"{pos=0.3}, shorten >=19pt, maps to, from=2-3, to=2-5]
				\arrow["f"', from=3-1, to=3-3]
				\arrow["Pf"', from=3-5, to=3-7]
				\arrow["\alpha","\cong"', shorten <=8pt, shorten >=8pt, Rightarrow, from=3-5, to=0]
			\end{tikzcd}\]
			 there exists a morphism $\widehat{h}\colon z \to x$ and invertible 2-cells $\widehat{\alpha}\colon f\widehat{h}\xRightarrow{\cong} g$ and $\widehat{\beta}\colon P\widehat{h}\xRightarrow{\cong} h$ such that $\alpha.(Pf\cdot\widehat{\beta}) = P\widehat{\alpha}$.

			\item \label{Def:cartesian_1cells.2}  
            Suppose we have 2-cells $\sigma\colon g \Rightarrow g'$ in $\cP$ and $h,h'\colon Pz \to Px$ in $\cC$ and isomorphisms $\alpha \colon (Pf).h \xRightarrow{\cong} Pg$ and $\alpha' \colon (Pf).h' \xRightarrow{\cong} Pg'$ such that $(h, \alpha)$ and $(h', \alpha')$ have lifts $(\widehat{h},\widehat{\alpha},\widehat{\beta})$ and $(\widehat{h'},\widehat{\alpha'},\widehat{\beta'})$ as in the item above. Then, for any $\delta \colon h \Rightarrow h'$ in $\cC$ with $\alpha'.(Pf \cdot \delta) = (P\sigma).\alpha$
			as depicted \[
			\begin{tikzcd}
				z &&&&& Pz \\
				{} && {} &&& {} && {} \\
				x &&& y && Px &&& Py
				\arrow["{\widehat{h}}"', curve={height=12pt}, from=1-1, to=3-1]
				\arrow["{\widehat{h'}}", curve={height=-12pt}, dashed, from=1-1, to=3-1]
				\arrow[""{name=0, anchor=center, inner sep=0}, "{g'}"{pos=0.6}, curve={height=-12pt}, from=1-1, to=3-4]
				\arrow[""{name=1, anchor=center, inner sep=0}, "g"'{pos=0.6}, curve={height=12pt}, from=1-1, to=3-4]
				\arrow[""{name=2, anchor=center, inner sep=0}, "h"', curve={height=12pt}, from=1-6, to=3-6]
				\arrow[""{name=3, anchor=center, inner sep=0}, "{h'}", curve={height=-12pt}, dashed, from=1-6, to=3-6]
				\arrow[""{name=4, anchor=center, inner sep=0}, "Pg"'{pos=0.6}, curve={height=12pt}, from=1-6, to=3-9]
				\arrow[""{name=5, anchor=center, inner sep=0}, "{Pg'}"{pos=0.6}, curve={height=-12pt}, from=1-6, to=3-9]
				\arrow["{\widehat{\alpha'}}"', shift left=5, shorten <=14pt, shorten >=27pt, Rightarrow, dashed, from=2-1, to=2-3]
				\arrow["{\alpha'}"', shift left=5, shorten <=14pt, shorten >=27pt, Rightarrow, dashed, from=2-6, to=2-8]
				\arrow[""{name=6, anchor=center, inner sep=0}, "f"', from=3-1, to=3-4]
				\arrow[""{name=7, anchor=center, inner sep=0}, "Pf"', from=3-6, to=3-9]
				\arrow["\sigma", shift right=5, shorten <=3pt, shorten >=3pt, Rightarrow, from=1, to=0]
				\arrow["\delta", shorten <=5pt, shorten >=5pt, Rightarrow,  from=2, to=3]
				\arrow["P\sigma", shift right=5, shorten <=3pt, shorten >=3pt, Rightarrow, from=4, to=5]
				\arrow["{\widehat{\alpha}}"{pos=0.3}, shift left=4, shorten <=4pt, shorten >=20pt, Rightarrow, from=3-1, to=6]
				\arrow["\alpha"{pos=0.3}, shift left=4, shorten <=4pt, shorten >=20pt, Rightarrow, from=3-6, to=7]
			\end{tikzcd}
			\]
			there exists a unique $\widehat{\delta} \colon \widehat{h} \Rightarrow \widehat{h'}$ such that $\widehat{\alpha'}.(f \cdot \widehat{\delta}) = \sigma.\widehat{\alpha}$ and $\delta.\widehat{\beta} = \widehat{\beta'}.P\widehat{\delta}$.
			\end{enumerate}

            In particular, as remarked by Buckley in \cite[Definition 3.1.1]{Buckley2014}, the uniqueness of lifted 2-cells implies that lifted morphisms are unique up to a coherent invertible 2-cell.
		\end{defn}

		\begin{defn}
		Let $P\colon \cP\to\cC$ be a 2-functor, and let $f,g\colon x\to y$ be morphisms in $\cP$. A $2$-cell $\alpha\colon f \Rightarrow g$  is \emph{cartesian} if it is cartesian as a morphism (in the 1-categorical sense of \cref{defn:cartesian1cat}) for the functor between hom-categories $P_{x,y}\colon \cP(x,y)\to \cC(Px,Py)$. 
		\end{defn}

		\begin{defn}\label{Def:2-fibration}
			A 2-functor $P\colon\cP\to\cC$ is a \emph{$2$-fibration} if it satisfies the following conditions:
			\begin{itemize}
				\item \label{Def:2-fibration.1}
				For every object $y\in\cP$ and every morphism $f\colon x'\to Py$ in $\cC$, there is a cartesian $P$-lift $\widehat{f}\colon x\to y$ of $f$.
				\item \label{Def:2-fibration.2} The 2-functor
				$P$ is locally a cartesian fibration.
				\item \label{Def:2-fibration.3}
				The horizontal composite of any two cartesian 2-cells is cartesian.
			\end{itemize}
		\end{defn}

        \begin{rmk}
            Some authors call these 2-functors \textit{cartesian} 2-fibrations, to distinguish them from \textit{discrete} 2-fibrations. Since we shall be concerned only with the former, we omit the adjective \textit{cartesian}, following the convention of \cite{Buckley2014}.
        \end{rmk}
        
		We now recall some elementary properties of cartesian $1$- and $2$-cells and $2$-fibrations.

		\begin{prop} \label{prop:propertiescartesian}
			For any 2-functor $P\colon\cP\to\cC$, the following hold:
            \begin{enumerate}[label=\stlabel{prop:propertiescartesian},ref=\arabic*]
            \item \label{prop:propertiescartesian.0} If there is an invertible 2-cell $f\cong g$, then $f$ is cartesian if and only if $g$ is cartesian {\normalfont \cite[Proposition 3.1.8]{Buckley2014}}.
				\item \label{prop:propertiescartesian.1}
				The composite of cartesian morphisms is cartesian {\normalfont \cite[Proposition 3.1.9]{Buckley2014}}.
                \item \label{prop:propertiescartesian.1.5} If $g$ and $gf$ are cartesian then
$f$ is cartesian {\normalfont \cite[Proposition 3.1.10]{Buckley2014}}.
				\item \label{prop:propertiescartesian.2}
				A morphism $f$ in $\cP$ is an equivalence if and only if $f$ is cartesian and $Pf$ is an equivalence {\normalfont \cite[Propositions 3.1.11 \& 3.1.12]{Buckley2014}}.
				\item \label{prop:propertiescartesian.3}
				If $g\colon x'\to x$ and $g'\colon x''\to x$ are two cartesian lifts of the same morphism,  then there is an equivalence $h\colon x''\to x'$ and an invertible $2$-cell $\tau\colon gh\Rightarrow g'$, both unique up to an invertible $2$-cell. Moreover, one can choose adjoint  equivalence data $(h,h^{-1},\eta,\epsilon)$ such that $Ph,Ph^{-1},P\eta,P\epsilon$ and $P\tau$ are all equal to the corresponding identity morphism and $2$-cells
				{\normalfont\cite[Proposition 3.1.15]{Buckley2014}}.
			\end{enumerate}
		\end{prop}

		\begin{prop}[{\cite[Proposition 3.2.1]{Buckley2014}}]\label{prop:buckley3.2.1}
			Let $P\colon \cP\to\cC$ be a 2-functor which is locally a cartesian fibration. Then, every lift $(\widehat{h},\widehat{\alpha},\widehat{\beta})$ of $(h,\alpha)$ as in {\normalfont{\enumref{Def:cartesian_1cells}{1}}} along a cartesian morphism can be chosen so that $\widehat{\beta} = \id_h$. In other words, lifts along cartesian morphisms can be chosen so that $P\widehat{h}= h$.
		\end{prop}

	\subsection{The canonical model structure on \texorpdfstring{$\twocat$}{twocat}} In this final preliminary section we introduce some classes of 2-functors that play a central role in what follows, and recall the canonical model structure on $\twocat$, due to Lack \cite{Lack2002,Lack2004}, as well as the projective model structure it induces.

		\begin{defn}
			A 2-functor $P\colon\cP\to\cC$ is a \emph{biequivalence} if the induced functors \linebreak
            $P_{x,x'}\colon\cP(x,x')\to\cC(Px,Px')$ are equivalences of categories and, moreover, for every object $x$ in $\cC$ there is an object $y$ in $\cP$ and an equivalence $x\xrightarrow{\sim} Py$ in $\cC$.
		\end{defn}

		\begin{defn}\label{defn:lackfib}
			A 2-functor $P\colon\cP\to\cC$ is an \emph{equifibration} if the following hold:
			\begin{itemize}
				\item \label{defn:lackfib.1}
				For every object $y$ of $\cP$ and every equivalence $f\colon x\xrightarrow{\sim} Py$ in $\cC$, there is an equivalence $g\colon x'\xrightarrow{\sim} y$ in $\cP$ with $Px' = x$ and $Pg = f$.
				\item \label{defn:lackfib.2}
				For every morphism $g\colon x\to y$ in $\cP$ and every invertible 2-cell $\beta\colon f\Rightarrow Pg$ in $\cC$, there is an invertible 2-cell $\alpha\colon g'\Rightarrow g$ in $\cP$ with $Pg' = f$ and $P\alpha = \beta$.
			\end{itemize}
		\end{defn}

As expected, the biequivalences and equifibrations determine a model structure on $\twocat$.
        
		\begin{thm}[{\cite[Theorem 4]{Lack2004}}] \label{cofibrant-model:2cat}
			There exists a cofibrantly generated model structure on $\twocat$ whose weak equivalences are the biequivalences, and fibrations are the equifibrations.
		\end{thm}

Given a 2-functor $P\colon\cP \to \cC$ we will write $UP\colon U\cP \to U\cC$ for the underlying functor between the underlying categories. Moving forward, we will adhere to the practice of denoting by $U$ a forgetful functor which, as the name suggests, forgets certain structure. It is usually clear from the context what the functor forgets. 

For our purposes, it will be convenient to describe the cofibrations of this model structure, as well as to identify a set of generating cofibrations. 

		\begin{prop}[{\cite[4.14]{Lack2002}}]\label{prop:cofibrationsLack}
			A 2-functor $P\colon\cP\to\cC$ is a cofibration if and only if its underlying functor $UP\colon U\cP\to U\cC$ is injective on objects and faithful, and there are functors $I\colon U\cC \to \widehat{\cC}$ and $R\colon\widehat{\cC}\to U\cC$ satisfying $RI = \id$ so that the $1$-category $\widehat{\cC}$ is obtained from the image of $P$ by freely adjoining objects and then arrows between specified objects.
		\end{prop}

		\begin{notation}\label{Def:Lackgeneratingcofibs}
			The following 2-functors provide a set of generating cofibrations for $\twocat$:
			\begin{itemize}
				\item \label{Def:Lackgeneratingcofibs.1}
				the canonical 2-functor $\emptyset \to [0]$, 
				
				\item \label{Def:Lackgeneratingcofibs.2}
				the suspension of $\emptyset \to [0]$; that is, the canonical 2-functor $[0]\sqcup[0]\to [1]$,
				
				\item \label{Def:Lackgeneratingcofibs.3} the
				suspension of $[0]\sqcup[0]\to [1]$,  
				\item \label{Def:Lackgeneratingcofibs.4}
				the suspension of $\{[0] \rightrightarrows [1]\} \to [1]$.
			\end{itemize}
		\end{notation}

	Studying the lifting properties with respect to each of the generating cofibrations above, one can directly obtain the following description of the trivial fibrations in $\twocat$.

        \begin{prop}\label{char:lacktrivfibs}
            A 2-functor in $\twocat$ is a trivial fibration if and only if it is surjective on objects, full on morphisms, and locally fully faithful on hom-categories.
        \end{prop}

	The following result relating the behaviour of trivial fibrations and cartesian morphisms will be of use later in the paper. 

        \begin{lem}\label{lem:trivialLackfibs_and_cartesians}
			Let $P\colon\cP \to \cC$ and $Q\colon\cQ \to \cC$ be 2-functors, and let $F\colon \cP \to \cQ$ be a trivial fibration in $\twocat$ such that $P=QF$. Then, a morphism $f$ in $\cP$ is $P$-cartesian if and only if $Ff$ is $Q$-cartesian.
		\end{lem}
		\begin{proof}
			We prove both implications using the characterization of cartesian morphisms given in \cite[Proposition 3.1.2]{Buckley2014}. Let $f\colon x\to y$ be a morphism in $\cP$. Since $F$ is a trivial fibration, it is surjective on objects, and so for any object $z'$ in $\cQ$ there is an object $z$ in $\cP$ such that $Fz = z'$. Moreover, since $P=QF$, we have $Pz = QFz = Qz'$. We can now consider the following diagram of categories 
            \[\begin{tikzcd}
{\cP(z,x)}\rar["F_{z,x}","\simeq"']\dar["f_*"'] & {\cQ(z',Fx)}\rar["Q_{z',Fx}"]\dar["(Ff)_*"] & {\cC(Pz,Px)}\dar["(Pf)_*"]\\
			{\cP(z,y)}\rar["F_{z,y}","\simeq"'] & {\cQ(z', Fy)}\rar["Q_{z',Fy}"] & {\cC(Pz,Py)}
            \end{tikzcd}\]
			in which the square on the left is a bipullback, since $F$ is locally an equivalence of categories. The pasting lemma for bipullbacks \cite[Lemma 1.2.24]{maillard:tel-03475256} and Buckley's characterization of cartesian $1$-morphisms \cite[Proposition 3.1.2]{Buckley2014} then imply that $f$ is $P$-cartesian if and only if $Ff$ is $Q$-cartesian.
		\end{proof}

In analogy with the 1-categorical case (see \cite[\nopp Proposition 3.36]{vistoli2007notesgrothendiecktopologiesfibered}), we have the result below, which allows us to show that a 2-functor is a biequivalence by studying one fiber at a time. The proof follows standard methods, but we nevertheless include it as it is a result of independent interest in the theory of 2-fibrations.

\begin{prop}\label{bieq.on.fibers}
    Let $P\colon\cP\to \cC$ and $Q\colon \cQ\to\cC$ be 2-fibrations, and let $H\colon P\to Q$ be a cartesian 2-functor over $\cC$. If for each object $c\in\cC$, we have that the 2-functor induced on fibers $H_c\colon \cP_c\to \cQ_c$ is a biequivalence, then $H$ is a biequivalence. 
\end{prop}
\begin{proof}
Recall that the fiber 2-category $\cP_c$ is the sub-2-category of $\cP$ with objects $x\in\cP$ such that $Px=c$, morphisms $u\colon x\to x'$ in $\cP$ such that $Pu=\id_c$, and 2-cells $\alpha\colon u\Rightarrow v$ in $\cP$ such that $P\alpha=\id$. We begin by showing that $H$ is essentially surjective on objects. Let $y \in \cQ$ be an object and set $c' \coloneqq Qy$. Then $y$ lies in the fiber $\cQ_{c'}$. Since $H_{c'}$ is a biequivalence, there exists an object $x \in \cP_{c'}$ and an equivalence $H_{c'}(x)\simeq y$ in $\cQ_{c'}$. The inclusion $\cQ_{c'} \hookrightarrow \cQ$ carries the equivalence $H_{c'}(x) \simeq y$ to an equivalence in $\cQ$. This shows that $H$ is essentially surjective on objects. It remains to show that for any pair of objects $x,x' \in \cP$, the induced functor on hom-categories
\[
H_{x,x'}\colon \cP(x,x') \to \cQ(Hx,Hx')
\]
is an equivalence of categories. 
Since $P$ and $Q$ are $2$-fibrations, the local functors $P_{x,x'}\colon \cP(x,x') \to \cC(Px,Px')$ and $Q_{Hx,Hx'}\colon \cQ(Hx,Hx') \to \cC(Px,Px')$ are Grothendieck fibrations of categories. Moreover, the functor $H_{x,x'}$ preserves cartesian 2-cells so it is a cartesian functor over $\cC(Px,Px')$. By \cite[\nopp Proposition 3.36]{vistoli2007notesgrothendiecktopologiesfibered}, it suffices to show that for each morphism $f\colon Px\to Px'$ in $\cC$, the induced functor on fibers
\[
(H_{x,x'})_f\colon (P_{x,x'})_f \to (Q_{Hx,Hx'})_f
\] 
is an equivalence of categories. 

Since $P$ is a $2$-fibration, the morphism $f$ admits a cartesian $P$-lift $\widehat{f}\colon f^{\ast}x' \to x'$ in $\cP$. Moreover, $H$ preserves cartesian morphisms and $P=QH$ so that $H\widehat{f}\colon Hf^{\ast}x'\to Hx'$ is a cartesian $Q$-lift of $f$. Thus, both $\widehat{f}$ and $H\widehat{f}$ induce equivalences of categories, respectively
\[
    \widehat{f} \circ - \colon \cP_{Px}(x, f^{\ast}x') \to (P_{x,x'})_f \quad\text{ and }\quad
    H\widehat{f} \circ - \colon \cQ_{Px}(Hx,Hf^{\ast}x') \to (Q_{Hx,Hx'})_f.
\]
Notice that for any morphism $u$ in $P_{Px}(x, f^{\ast}x')$ we have $H(\widehat{f}.u) = H\widehat{f}. Hu$ so that the equivalences of categories above fit into a commutative diagram 
\[\begin{tikzcd}[cramped]
	{\cP_{Px}(x, f^{\ast}x')} && {(P_{x,x'})_f} \\
	\\
	{\cQ_{Px}(Hx,Hf^{\ast}x')} && {(Q_{Hx,Hx'})_f}
	\arrow["{\widehat{f}\circ -}","\simeq"', from=1-1, to=1-3]
	\arrow["{(H_{Px})_{x,f^{\ast}x'}}"',"\mathrel{\rotatebox[origin=c]{-90}{$\simeq$}}", from=1-1, to=3-1]
	\arrow["{(H_{x,x'})_f}", from=1-3, to=3-3]
	\arrow["{H\widehat{f}\circ - }","\simeq"', from=3-1, to=3-3]
\end{tikzcd}\]
where the left vertical arrow is also an equivalence of categories by hypothesis. We conclude that $(H_{x,x'})_f$ is an equivalence of categories.
\end{proof}

	We conclude this background section by recalling the projective model structure on the category $2\Fun[\cC^{\coop}, \ttwocat]$ of strict 2-functors $\cC^{\coop}\to\ttwocat$ and strict 2-natural transformations.

    \begin{thm}
    For any 2-category $\cC$, there exists a model structure on the category \linebreak $2\Fun[\cC^{\coop}, \ttwocat]$, called the projective model structure, whose weak equivalences and fibrations are defined levelwise. In particular, since all objects in $\twocat$ are fibrant, all objects in $2\Fun[\cC^{\coop}, \ttwocat]$ are fibrant. 
    \end{thm}
    \begin{proof}
        This is directly implied by \cite[Theorem 5.4]{Moser_2019}, where, in that statement's notation, $(\cV,\otimes,I)=(\Cat,\times, [0])$ which is a locally presentable base by \cite[Example 1.4]{MR1624638}; $\cA=\ttwocat$ which is locally $\Cat$-presentable\textemdash indeed, we know that $\twocat$ is locally presentable as a closed category, and therefore by \cite[Proposition 4.8]{KellyLack_loc_presentable} it is so as $\twocat$-enriched category, and thus also as a $\cat$-enriched category\textemdash as well as enriched over $\Cat$; and $D=\cC^\coop$. The canonical model structure on $\twocat$ is induced by the canonical model structure on $\Cat$ in the sense of \cite{muro2014dwyer} and hence is combinatorial, and in particular accessible. Condition (ii) in \cite[Theorem 5.4]{Moser_2019} is satisfied by \cite[Remark 5.5]{Moser_2019}, as all objects are cofibrant in the canonical model structure on $\Cat$. 
    \end{proof}

\section{Model structure for cartesian \texorpdfstring{$2$}{2}-fibrations}
\label{sec:model-structure}
	The goal of this section is to construct a model structure whose fibrant objects capture the $2$-fibrations from \cref{Def:2-fibration}. As mentioned in the introduction, the $2$-fibrations over a $2$-category $\cC$ cannot be encoded by a lifting condition on the category $\twocat_{/ \cC}$, due to technical obstructions. Indeed, while it is possible to use a lifting property to express what it means for a given morphism or 2-cell to be $P$-cartesian, we cannot encode the data of a 2-fibration as a lifting property of the 2-functor $P$ itself, as this would require us to have access to the set of all $P$-cartesian
morphisms and 2-cells. The solution we employ is to use marked $2$-categories instead, whose definition we now introduce.

	\begin{defn}\label{Def:marked-cat}
		A \emph{marked $2$-category} is a triple $\markedscat{\cP}$ where $\cP$ is a $2$-category and $E^1_{\cP}$ and $E^2_{\cP}$ are sets of morphisms and $2$-cells in $\cP$, respectively, such that $E^1_{\cP}$ is closed under composition and contains all identities, and $E^2_{\cP}$ is closed under vertical composition and contains all invertible $2$-cells.
		
		A \emph{marked 2-functor} $P\colon \markedscat{\cP}\to \markedscat{\cC}$ is a 2-functor $P\colon\cP \to \cC$ that preserves markings, in the sense that $P(E^1_{\cP})\subseteq E^1_{\cC}$ and $P(E^2_{\cP})\subseteq E^2_{\cC}$. 
        
        We denote by $\twocat^{+}$ the category of marked $2$-categories and marked 2-functors.
	\end{defn}

	\begin{rmk}\label{rmk:forgetful-marked}
		The canonical forgetful functor $U\colon \twocat^{+}\to \twocat$ admits both a left and a right adjoint. The left adjoint $(-)^{\flat}\colon \twocat \to \twocat^{+}$ equips a $2$-category with its minimal marking; in which only the identity morphisms and invertible $2$-cells are marked. The right adjoint $(-)^{\sharp}\colon \twocat \to \twocat^{+}$ equips a $2$-category with its maximal marking; in which every morphism and every $2$-cell is marked. The suspension functor $\Sigma\colon \cat \to \twocat$ from \cref{par:suspension} extends to the category of marked categories, giving $\Sigma\colon \cat^{+}\to \twocat^{+}$; only identity morphisms in $\Sigma\cC$ are marked, and a $2$-cell of $\Sigma\cC$ is marked precisely when the corresponding morphism of $\cC$ is.
	\end{rmk}

    With this in hand, we can introduce the following.

	\begin{defn}\label{Def:2Cat-over-C-sharp}
		Let $\cC$ be a $2$-category. Then $\twocat^{+}_{/\cC^\sharp}$ is the category whose
		\begin{itemize}
			\item \label{Def:2Cat-over-C-sharp.1} objects are marked 2-functors $\markedscat{\cP}\to \Csharp$, and 
			\item \label{Def:2Cat-over-C-sharp.2} morphisms $\left( P\colon\markedscat{\cP}\to \Csharp\right) \to \left( Q\colon \markedscat{\cQ}\to \Csharp\right)$ are the marked 2-functors $F\colon \markedscat{\cP}\to \markedscat{\cQ}$ such that $QF=P$. 
		\end{itemize}
	\end{defn}

	Throughout this section, we let $\cC$ be a fixed 2-category. The main objective of this section is then to construct a model structure on $\twocat^{+}_{/\cC^\sharp}$ whose fibrant objects detect the 2-fibrations. More precisely, the fibrant objects will be the marked 2-functors $P\colon \markedscat{\cP}\to\Csharp$ such that $P\colon\cP\to\cC$ is a $2$-fibration and the markings $E^1_\cP$ and $E^2_\cP$ consist of the $P$-cartesian morphisms and $2$-cells in $\cP$. Our strategy will be to use \cite[Theorem 2.8]{guetta2023fibrantlyinduced}, and the bulk of this section seeks to establish all the necessary requirements. In \cref{sec:anodyne} we find a set of generating anodyne extensions; these will be some of the trivial cofibrations, and their main purpose is to detect the desired fibrant objects. In \cref{sec:cofibandtrivfib} we study the cofibrations and trivial fibrations of this model structure, and we introduce the weak equivalences in \cref{sec:weakequivs}. Lastly, \cref{sec:pathobject} deals with the construction of path objects, and \cref{sec:theMS} brings all these results together to construct the desired model structure. 

	\subsection{Anodyne extensions and naive fibrations}\label{sec:anodyne}
		Let us start by describing the set of generating anodyne extensions. To simplify notation, we follow the convention that identities and composites are not explicitly depicted in the sets of markings, nor are invertible $2$-cells, as all these are always marked. For instance, the marked $2$-category $([2], \set{0\to 1, 1 \to 2})$ has $[2]$ as its underlying $2$-category; this only has identity $2$-cells and hence they are all marked and omitted from the notation. The marked morphisms are $0 \to 1$ and $1 \to 2$ as explicitly indicated, as well as the three identity morphisms and the composite $0 \to 2$.

		\begin{defn}\label{Def:An}
			The class $\an^{+}$ of \emph{marked anodyne extensions} is the smallest weakly saturated class of marked 2-functors over $\cC^{\sharp}$ containing the following: 
			\begin{enumerate}[label=\stlabel{Def:An}, ref=\arabic*]
				\item \label{Def:An.1} the inclusion  $j_1\colon [0]\xrightarrow{1} [1]^{\sharp}$;
				
				\item \label{Def:An.2} the inclusion $j_2\colon  \left(\Lambda^2[2],\set{1 \rightarrow 2}\right) \to  \left([2]^{\cong},\set{1\rightarrow 2}\right)$ given by
				\[
				\begin{tikzcd}
					& 1 &&& 1 \\
					0 & 2 && 0 & 2
					\arrow[from=2-1, to=2-2, "h"']
					\arrow[""{name=0, anchor=center, inner sep=0}, "{+}"',"g", from=1-2, to=2-2]
					\arrow[from=2-4, to=2-5, "h"']
					\arrow["{+}"',"g", from=1-5, to=2-5]
					\arrow[""{name=1, anchor=center, inner sep=0}, from=2-4, to=1-5,"f"]
					\arrow["\cong"{description}, draw=none, from=1, to=2-5]
					\arrow[shorten <=24pt, shorten >=24pt, from=0, to=1]
				\end{tikzcd}
				\]
				 where $[2]^{\cong}$ denotes the $2$-category with the same objects as $[2]$, with  morphisms $f\colon 0\to 1$, $g\colon 1\to 2$, $gf\colon 0\to 2$ and $h\colon 0\to 2$, and with an invertible $2$-cell $gf\xRightarrow{\cong} h$;
                 
				\item \label{Def:An.3} the inclusion $j_3\colon E^{\flat} \to E^{\sharp}$ where $E$ is the walking equivalence;
				
				\item \label{Def:An.4} 
                the inclusion $j_4$ depicted below
                \[\begin{tikzcd}
					0 & 1 && 0 & 1
					\arrow[""{name=0, anchor=center, inner sep=0}, "+", curve={height=-12pt}, from=1-1, to=1-2]
					\arrow[""{name=1, anchor=center, inner sep=0}, ""', curve={height=12pt}, from=1-1, to=1-2]
					\arrow[""{name=2, anchor=center, inner sep=0}, "+", curve={height=-12pt}, from=1-4, to=1-5]
					\arrow[""{name=3, anchor=center, inner sep=0}, "+"', curve={height=12pt}, from=1-4, to=1-5]
					\arrow[shorten <=10pt, shorten >=10pt, from=1-2, to=1-4]
					\arrow["\cong"'{xshift=-2pt}, shorten <=3pt, shorten >=3pt, Rightarrow, from=0, to=1]
					\arrow["\cong"'{xshift=-2pt}, shorten <={3pt}, shorten >=3pt, Rightarrow, from=2, to=3]
				\end{tikzcd};\] 
				
				\item \label{Def:An.5} the inclusion $j_5=\Sigma j_1\colon \Sigma [0] \to \Sigma [1]^{\sharp}$ given by
				\[
				\begin{tikzcd}
					x & y && x & y
					\arrow[""{name=0, anchor=center, inner sep=0}, "f", bend right={-32pt}, from=1-4, to=1-5]
					\arrow[""{name=1, anchor=center, inner sep=0}, "{g}"', bend right={32pt}, from=1-4, to=1-5]
					\arrow["g", from=1-1, to=1-2]
					\arrow[shorten <=10pt, shorten >=20pt,  from=1-2, to=1-4]
					\arrow["{+}", shorten <=3pt, shorten >=3pt, Rightarrow,from=0, to=1]
				\end{tikzcd};
				\]
				
				\item \label{Def:An.6} the inclusion $j_6$ displayed below 
                \[
				\begin{tikzcd}
					& g &&& g \\
					h & f && h & f
					\arrow[Rightarrow, from=2-1, to=2-2]
					\arrow[Rightarrow, ""{name=0, anchor=center, inner sep=0}, "{+}", from=1-2, to=2-2]
					\arrow[Rightarrow, from=2-4, to=2-5]
					\arrow[Rightarrow, "{+}", from=1-5, to=2-5]
					\arrow[Rightarrow, ""{name=1, anchor=center, inner sep=0}, from=2-4, to=1-5]
					\arrow["="{description}, draw=none, from=1, to=2-5]
					\arrow[shorten <=24pt, shorten >=24pt, from=0, to=1]
				\end{tikzcd};
				\]
				
				\item \label{Def:An.8}  the 2-functor $j_7$ depicted below with $\beta.\gamma_1 = \beta.\gamma_2 = \alpha$, which collapses $\gamma_1$ and $\gamma_2$ to a single 2-cell
                \[
				\begin{tikzcd}
					& g &&& g \\
					h & f && h & f
                    \arrow[Rightarrow, from=2-1, to=1-2, shift left,"\gamma_1"]
                    \arrow[Rightarrow, from=2-1, to=1-2, shift right, "\gamma_2"']
					\arrow[Rightarrow, from=2-1, to=2-2, "\alpha"']
					\arrow[Rightarrow, ""{name=0, anchor=center, inner sep=0}, "{+}"',"\beta", from=1-2, to=2-2]
					\arrow[Rightarrow, from=2-4, to=2-5,"\alpha"']
					\arrow[Rightarrow, "{+}"',"\beta", from=1-5, to=2-5]
					\arrow[Rightarrow, ""{name=1, anchor=center, inner sep=0}, from=2-4, to=1-5,"\gamma"]
					\arrow["="{description}, draw=none, from=1, to=2-5]
                    \arrow["="{description}, draw=none, from=1, to=2-5]
					\arrow[shorten <=24pt, shorten >=24pt, from=0, to=1]
				\end{tikzcd};
				\]

				\item \label{Def:An.10} the inclusion $j_{8}$ shown in the diagram below
                \[\begin{tikzcd}[ampersand replacement=\&, column sep=12pt, row sep=1.8ex]
                	0 \&\&\&\&\&\& 0 \\
                	{} \&\&\& {} \&\&\& {} \&\& {} \& {} \\
                	{} \&\&\& {} \&\& {} \& {} \&\&\& {} \\
                	1 \&\&\&\& 2 \&\& 1 \&\&\&\& 2
                	\arrow["{h'}", curve={height=-12pt}, from=1-1, to=4-1]
                	\arrow["h"', curve={height=12pt}, from=1-1, to=4-1]
                	\arrow[""{name=0, anchor=center, inner sep=0}, "g"'{pos=0.6}, curve={height=12pt}, from=1-1, to=4-5]
                	\arrow[""{name=1, anchor=center, inner sep=0}, "{g'}"{pos=0.6}, curve={height=-18pt}, from=1-1, to=4-5]
                	\arrow[""{name=2, anchor=center, inner sep=0}, "h"', curve={height=12pt}, from=1-7, to=4-7]
                	\arrow[""{name=3, anchor=center, inner sep=0}, "{h'}", curve={height=-12pt}, from=1-7, to=4-7]
                	\arrow[""{name=4, anchor=center, inner sep=0}, "{g'}"{pos=0.6}, curve={height=-18pt}, from=1-7, to=4-11]
                	\arrow[""{name=5, anchor=center, inner sep=0}, "g"'{pos=0.6}, curve={height=12pt}, from=1-7, to=4-11]
                	\arrow["{\alpha'}"{pos=0.42},"\cong"'{pos=0.42}, shift left=1, between={0.3}{0.6}, Rightarrow, from=2-1, to=2-4]
                	\arrow["{\alpha'}"{pos=0.42},"\cong"'{pos=0.42}, shift left=1, between={0.3}{0.6}, Rightarrow, from=2-7, to=2-10]
                	\arrow["\alpha"{pos=0.3},"\cong"'{pos=0.3}, shift right=5, between={0.1}{0.5}, Rightarrow, from=3-1, to=3-4]
                	\arrow[shift left=5, between={0.4}{1}, from=3-4, to=3-6]
                	\arrow["\alpha"{pos=0.3},"\cong"'{pos=0.3}, shift right=5, between={0.1}{0.5}, Rightarrow, from=3-7, to=3-10]
                	\arrow["f"', from=4-1, to=4-5]
                	\arrow["f"', from=4-7, to=4-11]
                	\arrow["\sigma"', shift right=5, between={0.2}{0.8}, Rightarrow, from=0, to=1]
                	\arrow["\delta", between={0.2}{0.8}, Rightarrow, from=2, to=3]
                	\arrow["\sigma"', shift right=5, between={0.2}{0.8}, Rightarrow, from=5, to=4]
                \end{tikzcd}\]
				where $\alpha\colon fh \xRightarrow{\cong} g$ and $\alpha'\colon fh'\xRightarrow{\cong} g'$ are invertible $2$-cells and $\alpha'. (f \cdot \delta) = \sigma. \alpha$;

				\item \label{Def:An.11} the 2-functor $j_{9}$ depicted below, collapsing $\delta$ and $\delta'$ to a single 2-cell $\delta$
                \[\begin{tikzcd}[ampersand replacement=\&, column sep=12pt, row sep=1.8ex]
                	0 \&\&\&\&\&\& 0 \\
                	{} \&\&\& {} \&\&\& {} \&\& {} \& {} \\
                	{} \&\&\& {} \&\& {} \& {} \&\&\& {} \\
                	1 \&\&\&\& 2 \&\& 1 \&\&\&\& 2
                	\arrow[""{name=0, anchor=center, inner sep=0}, "{h'}", curve={height=-12pt}, from=1-1, to=4-1]
                	\arrow[""{name=1, anchor=center, inner sep=0}, "h"', curve={height=12pt}, from=1-1, to=4-1]
                	\arrow[""{name=2, anchor=center, inner sep=0}, "g"'{pos=0.6}, curve={height=12pt}, from=1-1, to=4-5]
                	\arrow[""{name=3, anchor=center, inner sep=0}, "{g'}"{pos=0.6}, curve={height=-18pt}, from=1-1, to=4-5]
                	\arrow[""{name=4, anchor=center, inner sep=0}, "h"', curve={height=12pt}, from=1-7, to=4-7]
                	\arrow[""{name=5, anchor=center, inner sep=0}, "{h'}", curve={height=-12pt}, from=1-7, to=4-7]
                	\arrow[""{name=6, anchor=center, inner sep=0}, "{g'}"{pos=0.6}, curve={height=-18pt}, from=1-7, to=4-11]
                	\arrow[""{name=7, anchor=center, inner sep=0}, "g"'{pos=0.6}, curve={height=12pt}, from=1-7, to=4-11]
                	\arrow["{\alpha'}"{pos=0.42},"\cong"'{pos=0.42}, shift left=1, between={0.3}{0.6}, Rightarrow, from=2-1, to=2-4]
                	\arrow["{\alpha'}"{pos=0.42}, "\cong"'{pos=0.42}, shift left=1, between={0.3}{0.6}, Rightarrow, from=2-7, to=2-10]
                	\arrow["\alpha"{pos=0.3},"\cong"'{pos=0.3}, shift right=5, between={0.1}{0.5}, Rightarrow, from=3-1, to=3-4]
                	\arrow[shift left=5, between={0.4}{1}, from=3-4, to=3-6]
                	\arrow["\alpha"{pos=0.3},"\cong"'{pos=0.3}, shift right=5, between={0.1}{0.5}, Rightarrow, from=3-7, to=3-10]
                	\arrow["f"', from=4-1, to=4-5]
                	\arrow["f"', from=4-7, to=4-11]
                	\arrow["{\delta'}", shift left=4, between={0.2}{0.8}, Rightarrow, from=1, to=0]
                	\arrow["\delta", shift right=3, between={0.2}{0.8}, Rightarrow, from=1, to=0]
                	\arrow["\sigma"', shift right=3, between={0.2}{0.8}, Rightarrow, from=2, to=3]
                	\arrow["\delta", between={0.2}{0.8}, Rightarrow, from=4, to=5]
                	\arrow["\sigma"', shift right=5, between={0.2}{0.8}, Rightarrow, from=7, to=6]
                \end{tikzcd}\]
				where $\alpha\colon fh \xRightarrow{\cong} g$ and $\alpha'\colon fh'\xRightarrow{\cong} g'$ are invertible $2$-cells and $\alpha'.(f \cdot \delta) = \sigma.\alpha$ and;
				
				\item \label{Def:An.12} the 2-functor $j_{10}$ illustrated below
				\[
				\begin{tikzcd}[column sep=15pt,
    row sep=2ex]
					0 && 1 && 2 &&& 0 && 1 && 2
					\arrow[""{name=0, anchor=center, inner sep=0}, "gf", bend right={-70pt}, from=1-1, to=1-5]
					\arrow[""{name=1, anchor=center, inner sep=0}, "f", bend right={-30pt}, from=1-1, to=1-3]
					\arrow[""{name=2, anchor=center, inner sep=0}, "{f'}"', bend right={30pt}, from=1-1, to=1-3]
					\arrow[""{name=3, anchor=center, inner sep=0}, "g", bend right={-30pt}, from=1-3, to=1-5]
					\arrow[""{name=4, anchor=center, inner sep=0}, "{g'}"', bend right={30pt}, from=1-3, to=1-5]
					\arrow[""{name=5, anchor=center, inner sep=0}, "{g'f'}"', bend right={70pt}, from=1-1, to=1-5]
					\arrow[""{name=6, anchor=center, inner sep=0}, "f", bend right={-30pt}, from=1-8, to=1-10]
					\arrow[""{name=7, anchor=center, inner sep=0}, "g", bend right={-30pt}, from=1-10, to=1-12]
					\arrow[""{name=8, anchor=center, inner sep=0}, "gf", bend right={-70pt}, from=1-8, to=1-12]
					\arrow[""{name=9, anchor=center, inner sep=0}, "{f'}"', bend right={30pt}, from=1-8, to=1-10]
					\arrow[""{name=10, anchor=center, inner sep=0}, "{g'}"', bend right={30pt}, from=1-10, to=1-12]
					\arrow[""{name=11, anchor=center, inner sep=0}, "{g'f'}"', bend right={70pt}, from=1-8, to=1-12]
					\arrow[shorten <=16pt, shorten >=16pt, from=1-5, to=1-8]
					\arrow["\alpha"',"+", shorten <=3pt, shorten >=3pt, Rightarrow, from=6, to=9]
					\arrow["\beta"',"+", shorten <=3pt, shorten >=3pt, Rightarrow, from=7, to=10]
					\arrow["\gamma"{pos=0.8},"+"'{pos=0.8},bend right={-25pt}, shorten <=4pt, shorten >=4pt, Rightarrow, from=8, to=11,crossing over]
					\arrow["\gamma"{pos=0.8},bend right={-25pt}, shorten <=4pt, shorten >=4pt, Rightarrow, from=0, to=5,crossing over]
					\arrow["\alpha"',"+", shorten <=3pt, shorten >=3pt, Rightarrow, from=1, to=2]
					\arrow["\beta"',"+", shorten <=3pt, shorten >=3pt, Rightarrow, from=3, to=4]
				\end{tikzcd}
				\]
				where  $\gamma$ is the horizontal composite of $\alpha$ and $\beta$.
			\end{enumerate}
		\end{defn}

		\begin{rmk}
			Suppose $P\colon\cP \to \cC$ is a 2-functor. To give readers a glimpse of the intuition behind each of the maps $j_i$ in the definition above, one can observe that the right lifting property of $P$ with respect to the anodyne extension \enumref{Def:An}{1} will encode the existence of $P$-lifts of morphisms of the form $x \to Py$. The maps \enumref{Def:An}{2}, \enumref{Def:An}{10} and \enumref{Def:An}{11} will be used to ensure that these $P$-lifts are cartesian morphisms; see \cref{Def:cartesian_1cells}. Similarly, $P$ having the right lifting property with respect to the maps \enumref{Def:An}{5}, \enumref{Def:An}{6} and \enumref{Def:An}{8} will encode the property of $P$ being locally a cartesian fibration, and the right lifting property with respect to the map \enumref{Def:An}{12} will encode the third requirement for $P$ to be a $2$-fibration, namely that cartesian $2$-cells are closed under horizontal composition; see \cref{Def:2-fibration}. The remaining anodyne extensions \enumref{Def:An}{3} and \enumref{Def:An}{4} will be used to encode relevant properties of cartesian morphisms and $2$-cells. We will see all this in detail in \cref{Prop:naivefibrant_iff_cms}.
		\end{rmk}

		\begin{defn}\label{Def:Naivefibrations}
			The class $\nfib^{+}$ of \emph{marked naive fibrations} consists of the marked 2-functors in  $\twocat^{+}_{/\cC^\sharp}$ with the right lifting property with respect to the marked anodyne extensions.
            
            An object $P\colon \markedscat{\cP}\to\cC^\sharp$ in $\twocat^{+}_{/\cC^\sharp}$ is \emph{naive fibrant} if the unique marked 2-functor from $P$ to the terminal object $\id_{{\cC^{\sharp}}}\colon\cC^{\sharp}\to\cC^{\sharp}$ is a marked naive fibration.
		\end{defn}

        \begin{rmk}
            It is straightforward to verify that an object $P\colon \markedscat{\cP}\to\cC^\sharp$ is naive fibrant if and only if $P$ itself has the right lifting property with respect to the marked anodyne extensions.
        \end{rmk}

        The class of marked anodyne extensions of \cref{Def:An} was specifically designed so that the fibrant objects coincide with the marked $2$-fibrations over $\cC^\sharp$ whose marking is determined by the cartesian cells. 

		\begin{defn}\label{Def:Cart-marked-2-fibrations}
			An object $P\colon\markedscat{\cP} \to \cC^\sharp$ in $\twocat^{+}_{/\cC^\sharp}$ is a \emph{cart-marked $2$-fibration} if $P$ is a $2$-fibration whose markings $E^1_{\cP}$ and $E^2_{\cP}$ consist of the $P$-cartesian morphisms and $2$-cells, respectively.  
		\end{defn}

		\begin{thm}\label{Prop:naivefibrant_iff_cms}
			A marked 2-functor $P\colon \markedscat{\cP} \to \cC^{\sharp}$ is naive fibrant if and only if it is a cart-marked $2$-fibration.
		\end{thm}

		\begin{proof}
			We first show that if $P$ satisfies the lifting condition with respect to the marked anodyne extensions, then it is a $2$-fibration. To prove that $P$ is locally a cartesian fibration, let $x,x'$ be objects in $\cP$ and consider the functor $P_{xx'}\colon \cP(x,x') \to \cC(Px,Px')$. Notice that if $f\colon x\to x'$ is any morphism of $\cP$ and $\alpha\colon g \Rightarrow Pf$ is a $2$-cell in $\cC$, then the right lifting property of $P$ with respect to the map \enumref{Def:An}{5} guarantees the existence of a marked $P$-lift $\widehat{\alpha}\colon \widehat{g} \Rightarrow f$. We claim this marked $P$-lift is cartesian as a morphism in $\cP(x,x')$. For this, let $\beta\colon h \Rightarrow f$ and $\gamma\colon Ph \Rightarrow g$ be $2$-cells in $\cC$ and $\cP$, respectively, with $\alpha.\gamma = P\beta$. Then, since $\widehat{\alpha}$ is marked, the lifting property of $P$ with respect to the map \enumref{Def:An}{6} ensures the existence of a $P$-lift $\widehat{\gamma}\colon h\Rightarrow\widehat{g}$ of $\gamma$ such that $\beta = \widehat{\alpha}.\widehat{\gamma}$. Moreover, the lifting property with respect to \enumref{Def:An}{8} guarantees the uniqueness of the $2$-cell $\widehat{\gamma}$. We conclude that $P$ is locally a cartesian fibration. 
					
			Next, we prove that any morphism $f\colon y \to Px$ in $\cC$ admits a cartesian $P$-lift. The lifting condition with respect to the map \enumref{Def:An}{1} guarantees the existence of a marked $P$-lift $\widehat{f}\colon y' \to x$ of $f$; we must show that  $\widehat{f}$ is cartesian. Note that the lifting property of $P$ with respect to the map \enumref{Def:An}{2} implies that  $\widehat{f}$ lifts morphisms up to isomorphism (see Definition \enumrefalt{Def:cartesian_1cells}{1}), with the required invertible $2$-cell $\widehat{\beta}$ being the identity. It remains to show that $\widehat{f}$ also lifts $2$-cells coherently, as described in Definition \enumrefalt{Def:cartesian_1cells}{2}. For this,  suppose we are given all the starting data of Definition \enumrefalt{Def:cartesian_1cells}{2}; recall that as we have already proved that $P$ is locally a cartesian fibration, by \cref{prop:buckley3.2.1} we can assume that the invertible 2-cells $\widehat{\beta}$ and $\widehat{\beta'}$ participating in the given lifts are identities. Then, the lifting property of $P$ with respect to the map \enumref{Def:An}{10} ensures the existence of the lift $\widehat{\delta}$ required in Definition \enumrefalt{Def:cartesian_1cells}{2}, which is moreover unique due to the lifting property with respect to the map \enumref{Def:An}{11}.

        To conclude that the underlying 2-functor of $P$ is a 2-fibration, it remains to show that the horizontal composite of cartesian $2$-cells is cartesian. This will follow from the lifting property of $P$ with respect to the map \enumref{Def:An}{12}, once we establish that the set $E^2_\cP$ of marked 2-cells consists of the cartesian 2-cells. 
		
        We thus turn our attention to the markings $E_{\cP}^1$ and $E_{\cP}^2$, and show that these coincide with the $P$-cartesian morphisms and 2-cells, respectively. For this, it is enough to show that $P$-cartesian morphisms and 2-cells are in $E_{\cP}^1$ and $E_{\cP}^2$, for we have already proved in the preceding paragraphs that every marked morphism and $2$-cell in $\cP$ is cartesian. Let $f\colon x' \to x$ be a cartesian morphism in $\cP$, and consider the morphism $Pf$ in $\cC$. The lifting property of $P$ with respect to the map \enumref{Def:An}{1} implies that $Pf$ admits a marked $P$-lift $f'\colon x'' \to x$, which, like all marked morphisms, is then cartesian. Then $f$ and $f'$ are two cartesian lifts of the same morphism $Pf$, and so by Proposition \enumrefalt{prop:propertiescartesian}{3} there exists an equivalence $h\colon x' \to x''$ and an invertible 2-cell $\tau\colon f'h \Rightarrow f$. 
            This data can be represented in the following diagram
            \[\begin{tikzcd}[row sep=1in]
					x' & x && Px' & Px
					\arrow[""{name=0, anchor=center, inner sep=0},  "f'h \ +", curve={height=-12pt}, from=1-1, to=1-2]
					\arrow[""{name=1, anchor=center, inner sep=0}, "f"', curve={height=12pt}, from=1-1, to=1-2]
					\arrow[""{name=2, anchor=center, inner sep=0}, "P(f'h) \ +", curve={height=-12pt}, from=1-4, to=1-5]
					\arrow[""{name=3, anchor=center, inner sep=0}, "Pf \ +"', curve={height=12pt}, from=1-4, to=1-5]
					\arrow[shorten <=10pt, shorten >=10pt, from=1-2, to=1-4, mapsto, "P"]
					\arrow["\cong"'{xshift=-2pt},"\tau", shorten <=3pt, shorten >=3pt, Rightarrow, from=0, to=1]
					\arrow["\cong"'{xshift=-2pt},"P\tau", shorten <={3pt}, shorten >=3pt, Rightarrow, from=2, to=3]
				\end{tikzcd}\]
                where we have used the fact that the equivalence $h$ is marked by the lifting property of $P$ with respect to \enumref{Def:An}{3}; that marked morphisms are always closed under composition to mark $f'h$; and that $P$ preserves markings to mark $P(f'h)$ and $Pf=Pf'$. Then, the lifting property with respect to \enumref{Def:An}{4} implies that the cartesian morphism $f$ must be marked in $\cP$.
                
                The argument for $2$-cells is similar: starting from a cartesian 2-cell $\alpha\colon f\Rightarrow g$, the lifting property with respect to the map \enumref{Def:An}{5} gives a marked $P$-lift $\beta\colon f'\xRightarrow{+} g$ of $P\alpha$, which like all marked 2-cells, must be cartesian. Then both $\alpha$ and $\beta$ are cartesian $P$-lifts of $P\alpha$, and so by the 1-categorical analogue of Proposition \enumrefalt{prop:propertiescartesian}{3}, there is an invertible 2-cell $\gamma\colon f\xRightarrow{\cong} f'$ such that $\alpha=\beta.\gamma$. Since invertible 2-cells are all marked by definition, and markings are closed under vertical composition, we have that $\beta.\gamma$ is marked; hence so is $\alpha$.

			For the converse, one can verify directly that if $P\colon \markedscat{\cP} \to \cC^{\sharp}$ is a cart-marked 2-fibration, the fact that the markings agree with the cartesian morphisms and 2-cells implies that $P$ has the right lifting property with respect to the maps in \enumref{Def:An}{2}, \enumref{Def:An}{3} (as equivalences are always cartesian \cite[Proposition 3.1.11]{Buckley2014}), \enumref{Def:An}{4} (as morphisms isomorphic to cartesians are also cartesian \cite[Proposition 3.1.8]{Buckley2014}), \enumref{Def:An}{6}, \enumref{Def:An}{8}, \enumref{Def:An}{10}, and \enumref{Def:An}{11}. The remaining lifting properties with respect to the maps \enumref{Def:An}{1}, \enumref{Def:An}{5}, and \enumref{Def:An}{12} are ensured by the first, second, and third parts of \cref{Def:2-fibration}, respectively. 
		\end{proof}

	We conclude this subsection with a characterization of the naive fibrations between naive fibrant objects, showing that their underlying 2-functors agree with the fibrations in the canonical model structure on $\twocat$. 

		\begin{prop}\label{prop:fibbtwfibobj}
			A marked 2-functor $F\colon \markedscat{\cP} \to \markedscat{\cQ}$ between cart-marked 2-fibrations $P\colon\markedscat{\cP}\to \cC^{\sharp}$ and $Q\colon\markedscat{\cQ}\to \cC^{\sharp}$ is a naive fibration if and only if its underlying 2-functor $UF\colon \cP \to \cQ$ is an equifibration.
		\end{prop}

		\begin{proof}
			Suppose first that $F$ is a naive fibration; we will verify the conditions in \cref{defn:lackfib}. For the first condition, let $f\colon x \xrightarrow{\sim} Fy$ be an equivalence in $\cQ$. Since $F$ has the right lifting property with respect to the generating anodyne extension \enumref{Def:An}{1}, there exists some marked morphism $g\colon y'\to y$ such that $Fg=f$. As $P\colon(\cP,E^1_{\cP},E^2_{\cP})\to \cC^{\sharp}$ is cart-marked, we have that $g$ is cartesian, and so by Proposition \enumrefalt{prop:propertiescartesian}{2} the morphism $g$ must be an equivalence, as desired. For the second condition, suppose we have a morphism $g\colon x\to y$ in $\cP$ and an invertible $2$-cell $\beta\colon f\xRightarrow{\cong} Fg$ in $\cQ$. Then, the right lifting property of $F$ with respect to the generating anodyne extension \enumref{Def:An}{5} yields a marked (and hence, cartesian) 2-cell $\alpha\colon h\Rightarrow g$ in $\cP$ such that $F\alpha=\beta$. Using \cref{prop:cartesian_isomorphism_1cats} we conclude that $\alpha$ must be invertible. 

			Conversely, suppose that the underlying 2-functor $UF$ is an equifibration; we must check that $F$ has the right lifting property with respect to each of the generating anodyne extensions of \cref{Def:An}. Some of these are almost immediate:
			
			\begin{itemize}[leftmargin=*]
				\item For the map \enumref{Def:An}{3}, we use the fact that $\cP$ and $\cQ$ are cart-marked and all equivalences are cartesian (see \cite[Proposition 3.1.11]{Buckley2014}). 
                \item Similarly, for the map \enumref{Def:An}{4}, we use the fact that any morphism related by an invertible 2-cell to a cartesian morphism is itself cartesian (see \cite[Proposition 3.1.8]{Buckley2014}). 
				\item For the map \enumref{Def:An}{12}, we use the fact that $P$ is a cart-marked $2$-fibration and as such the horizontal composite of $P$-cartesian $2$-cells is cartesian. 
				\item The right lifting properties  of $F$ with respect to the maps \enumref{Def:An}{5},  \enumref{Def:An}{6} and \enumref{Def:An}{8} hold by \cite[Proposition 4.14]{mosersarazola2023model}, since the cart-marked 2-fibrations $P$ and $Q$ are locally cart-marked Grothendieck fibrations (see \cite[Definition 4.9]{mosersarazola2023model}), the equifibration $F$ is locally an isofibration, and these anodyne extensions on hom-categories correspond to the generating maps (i), (ii) and (iii) of \cite[Definition 4.11]{mosersarazola2023model}.
			\end{itemize}

			To show that $F$ has the right lifting property with respect to the map \enumref{Def:An}{1}, suppose that $f\colon x\to Fy$  is a cartesian morphism in $\cQ$; we need to find a marked (i.e.\ cartesian) morphism $\widehat{g}\colon\widehat{x}\to y$ in $\cP$ such that $F\widehat{g}= f$. For this, let us consider the corresponding morphism $Qf\colon Qx \to QFy = Py$ in $\cC$. Since $P$ is a $2$-fibration, there exists a cartesian lift $g\colon x' \to y$ with $Pg = Qf$. Then $Fg$ and $f$ are two cartesian $Q$-lifts of $Qf$, and so by Proposition \enumrefalt{prop:propertiescartesian}{3} there exists an equivalence $h\colon x \to F(x')$ and an invertible $2$-cell $\tau\colon f \Rightarrow (Fg)h$. Since $F$ is an equifibration, there exists an equivalence ${g'}\colon \widehat{x} \to x'$ such that $Fg'=h$. We then have an invertible $2$-cell $\tau\colon f \Rightarrow F(gg')$, and so using the fact that $F$ is an equifibration once again, we can find a morphism $\widehat{g}\colon \widehat{x} \to y$ in $\cP$ and an invertible $2$-cell $\alpha\colon \widehat{g}  \Rightarrow gg'$ such that $F\widehat{g} = f$ and $F\alpha = \tau$. Thus \cite[Proposition 3.1.8]{Buckley2014} ensures that $\widehat{g}$ is the desired cartesian morphism, as it is related by the invertible 2-cell $\alpha$ to the composite of $g'$ (which is an equivalence, hence cartesian) and $g$ (which is cartesian by construction).

			The lifting problem   with respect to the map \enumref{Def:An}{2} is represented by the following data:
			\[\begin{tikzcd}[ampersand replacement=\&,row sep=scriptsize]
				{x} \&\&\&\& {Fx} \&\&\&\& {Px} \\
				\\
				{y} \&\& {z} \&\& {Fy} \&\& {Fz} \&\& {Py} \&\& {Pz}
				\arrow[""{name=0, anchor=center, inner sep=0}, "g", from=1-1, to=3-3]
				\arrow[""{name=1, anchor=center, inner sep=0}, "h"', from=1-5, to=3-5]
				\arrow[""{name=2, anchor=center, inner sep=0}, "Fg", from=1-5, to=3-7]
				\arrow[""{name=3, anchor=center, inner sep=0}, "Qh"', from=1-9, to=3-9]
				\arrow[""{name=4, anchor=center, inner sep=0}, "Pg", from=1-9, to=3-11]
				\arrow["f"', from=3-1, to=3-3]
				\arrow["Ff"', from=3-5, to=3-7]
				\arrow["Pf"', from=3-9, to=3-11]
				\arrow["F", mapsto, between={0.3}{0.7}, from=0, to=1]
				\arrow["Q", mapsto, between={0.3}{0.7}, from=2, to=3]
				\arrow["\alpha","\cong"', between={0}{0.8}, Rightarrow, from=3-5, to=2]
				\arrow["{Q\alpha}", "\cong"', between={0}{0.8}, Rightarrow, from=3-9, to=4]
			\end{tikzcd}\]
			where $f$ is $P$-cartesian and $\alpha$ is an invertible $2$-cell. The lifting problem is solved provided we can find a morphism $\widehat{h}\colon x \to y$ in $\cP$ together with an invertible $2$-cell $\widehat{\alpha}\colon f\widehat{h}\Rightarrow g$ such that $F \widehat{\alpha}=\alpha$. 
            To achieve this, we can first use the fact that $f$ is $P$-cartesian to get a solution for the composite lifting problem above: that is, a morphism $k\colon x\to y$ and an invertible 2-cell $\gamma\colon fk\Rightarrow g$ in $\cP$ such that $P\gamma=Q\alpha$. We then have that $(Fk,F\gamma)$ and $(h,\alpha)$ are two solutions to the lifting problem below.
            \[\begin{tikzcd}[ampersand replacement=\&,row sep=scriptsize]
            	\& Fx \&\&\&\& Px \\
            	\\
            	{} \& Fy \&\& Fz \&\& Py \&\& Pz
            	\arrow["Fk"'{pos=0.6}, curve={height=12pt}, from=1-2, to=3-2]
            	\arrow["h"{pos=0.6}, curve={height=-12pt}, from=1-2, to=3-2]
            	\arrow[""{name=0, anchor=center, inner sep=0}, "Fg", from=1-2, to=3-4]
            	\arrow[""{name=1, anchor=center, inner sep=0}, "Qh"', from=1-6, to=3-6]
            	\arrow[""{name=2, anchor=center, inner sep=0}, "Pg", from=1-6, to=3-8]
            	\arrow["Ff"', from=3-2, to=3-4]
            	\arrow["Pf"', from=3-6, to=3-8]
            	\arrow["Q", maps to, between={0.3}{0.7}, from=0, to=1]
            	\arrow["\alpha"{pos=0.7}, shift right=4, between={0.4}{1}, Rightarrow, from=3-2, to=0]
            	\arrow["{Q\alpha}","\cong"', between={0}{0.8}, Rightarrow, from=3-6, to=2]
                \arrow["{F\gamma}"{pos=0.65},shift left=3, between={0.5}{0.8}, Rightarrow, from=3-1, to=0, crossing over]
            \end{tikzcd}\]
                   
            Since lifted morphisms along the $Q$-cartesian morphism $Ff$ must be unique up to a coherent invertible 2-cell (see \cref{Def:cartesian_1cells}), there is some invertible 2-cell $\sigma\colon h\Rightarrow Fk$ in $\cQ$ with $\alpha=F\gamma.(Ff \cdot \sigma)$. In turn, as $F$ is an equifibration, there exists some morphism $\widehat{h}\colon x\to y$ and an invertible 2-cell $\widehat{\sigma}\colon \widehat{h}\Rightarrow k$ in $\cP$ such that $F\widehat{h}=h$ and $F\widehat{\sigma}=\sigma$. The morphism $\widehat{h}$ is the one we seek, and we can define the desired 2-cell as $\widehat{\alpha}=\gamma.(f\cdot \widehat{\sigma})$, which by construction is invertible and satisfies $F\widehat{\alpha}=F\gamma.(Ff\cdot \sigma)=\alpha.$

It remains to check the right lifting property of $F$ with respect to the maps \enumref{Def:An}{10} and \enumref{Def:An}{11}. The lifting problem with respect to the first map can be represented as follows:
    \[\begin{tikzcd}[ampersand replacement=\&, column sep=14pt,
    row sep=2.6ex]
    	z \&\&\&\&\&\& Fz \&\&\&\&\&\& Pz \\
    	{} \&\&\& {} \&\&\& {} \&\&\& {} \&\&\& {} \&\&\& {} \\
    	{} \&\&\& {} \&\& {} \& {} \&\&\& {} \&\& {} \& {} \&\&\& {} \\
    	x \&\&\&\& y \&\& Fx \&\&\&\& Fy \&\& Px \&\&\&\& Py
    	\arrow["{h'}", curve={height=-12pt}, from=1-1, to=4-1, dashed]
    	\arrow["h"', curve={height=12pt}, from=1-1, to=4-1]
    	\arrow[""{name=0, anchor=center, inner sep=0}, "g"'{pos=0.6}, curve={height=12pt}, from=1-1, to=4-5]
    	\arrow[""{name=1, anchor=center, inner sep=0}, "{g'}"{pos=0.6}, curve={height=-18pt}, from=1-1, to=4-5]
    	\arrow[""{name=2, anchor=center, inner sep=0}, "Fh"', curve={height=12pt}, from=1-7, to=4-7]
    	\arrow[""{name=3, anchor=center, inner sep=0}, "{Fh'}", curve={height=-12pt}, from=1-7, to=4-7,dashed]
    	\arrow[""{name=4, anchor=center, inner sep=0}, "{Fg'}"{pos=0.6}, curve={height=-18pt}, from=1-7, to=4-11]
    	\arrow[""{name=5, anchor=center, inner sep=0}, "Fg"'{pos=0.6}, curve={height=12pt}, from=1-7, to=4-11]
    	\arrow[""{name=6, anchor=center, inner sep=0}, "Ph"', curve={height=12pt}, from=1-13, to=4-13]
    	\arrow[""{name=7, anchor=center, inner sep=0}, "{Ph'}", curve={height=-12pt}, from=1-13, to=4-13,dashed]
    	\arrow[""{name=8, anchor=center, inner sep=0}, "{Pg'}"{pos=0.6}, curve={height=-18pt}, from=1-13, to=4-17]
    	\arrow[""{name=9, anchor=center, inner sep=0}, "Pg"'{pos=0.6}, curve={height=12pt}, from=1-13, to=4-17]
    	\arrow["{\alpha'}"{pos=0.42},"\cong"'{pos=0.42}, shift left=1, between={0.3}{0.6}, Rightarrow, from=2-1, to=2-4,dashed]
    	\arrow["{F\alpha'}"{pos=0.42}, "\cong"'{pos=0.42}, shift left=1, between={0.3}{0.6}, Rightarrow, from=2-7, to=2-10,dashed]
    	\arrow["{P\alpha'}"{pos=0.42},"\cong"'{pos=0.42}, shift left=1, between={0.3}{0.6}, Rightarrow, from=2-13, to=2-16,dashed]
    	\arrow["\alpha"{pos=0.3}, "\cong"'{pos=0.3}, shift right=5, between={0.1}{0.5}, Rightarrow, from=3-1, to=3-4]
    	\arrow["F"{pos=0.7}, shift left=5, between={0.4}{1}, maps to, from=3-4, to=3-6]
    	\arrow["{F\alpha}"{pos=0.3}, "\cong"'{pos=0.3}, shift right=5, between={0.1}{0.5}, Rightarrow, from=3-7, to=3-10]
    	\arrow["Q"{pos=0.7}, shift left=5, between={0.4}{1}, maps to, from=3-10, to=3-12]
    	\arrow["{P\alpha}"{pos=0.3},"\cong"'{pos=0.3}, shift right=5, between={0.1}{0.5}, Rightarrow, from=3-13, to=3-16]
    	\arrow["f"', from=4-1, to=4-5]
    	\arrow["Ff"', from=4-7, to=4-11]
    	\arrow["Pf"', from=4-13, to=4-17]
    	\arrow["\sigma"', shift right=5, between={0.2}{0.8}, Rightarrow, from=0, to=1]
    	\arrow["\delta", between={0.2}{0.8}, Rightarrow, from=2, to=3]
    	\arrow["{F\sigma}"', shift right=5, between={0.2}{0.8}, Rightarrow, from=5, to=4]
    	\arrow["{Q\delta}", between={0.2}{0.8}, Rightarrow, from=6, to=7]
    	\arrow["{P\sigma}"', shift right=5, between={0.2}{0.8}, Rightarrow, from=9, to=8]
    \end{tikzcd}\]
        where $\alpha\colon fh \xRightarrow{\cong} g$ and $\alpha'\colon fh'\xRightarrow{\cong} g'$ are invertible $2$-cells and $F\alpha'. (Ff \cdot \delta) = F\sigma. F\alpha$. We wish to show that there exists some 2-cell $\widehat{\delta}\colon h\Rightarrow h'$ in $\cP$ such that $\alpha'. (f \cdot \widehat{\delta}) = \sigma. \alpha$ and $F\widehat{\delta}=\delta$. Similarly to the previous paragraph, we can first consider the composite lifting problem; since $f$ is $P$-cartesian, there exists some $\widehat{\delta}\colon h\Rightarrow h'$ such that $\alpha'. (f \cdot \widehat{\delta}) = \sigma. \alpha$ and $P\widehat{\delta}=Q\delta$. Hence, we only need to show that $F\widehat{\delta}=\delta$. This is due to the fact that $F\widehat{\delta}$ and $\delta$ are two $Q$-lifts of the lifting problem  
        \[\begin{tikzcd}[ampersand replacement=\&,
      column sep=14pt,
    row sep=2.6ex]
        	Fz \&\&\&\&\&\& Pz \\
        	{} \&\&\& {} \&\&\& {} \&\&\& {} \\
        	{} \&\&\& {} \&\& {} \& {} \&\&\& {} \\
        	Fx \&\&\&\& Fy \&\& Px \&\&\&\& Py
        	\arrow[""{name=0, anchor=center, inner sep=0}, "Fh"', curve={height=12pt}, from=1-1, to=4-1]
        	\arrow[""{name=1, anchor=center, inner sep=0}, "{Fh'}", curve={height=-12pt}, from=1-1, to=4-1,dashed]
        	\arrow[""{name=2, anchor=center, inner sep=0}, "{Fg'}"{pos=0.6}, curve={height=-18pt}, from=1-1, to=4-5]
        	\arrow[""{name=3, anchor=center, inner sep=0}, "Fg"'{pos=0.6}, curve={height=12pt}, from=1-1, to=4-5]
        	\arrow[""{name=4, anchor=center, inner sep=0}, "Ph"', curve={height=12pt}, from=1-7, to=4-7]
        	\arrow[""{name=5, anchor=center, inner sep=0}, "{Ph'}", curve={height=-12pt}, from=1-7, to=4-7,dashed]
        	\arrow[""{name=6, anchor=center, inner sep=0}, "{Pg'}"{pos=0.6}, curve={height=-18pt}, from=1-7, to=4-11]
        	\arrow[""{name=7, anchor=center, inner sep=0}, "Pg"'{pos=0.6}, curve={height=12pt}, from=1-7, to=4-11]
        	\arrow["{F\alpha'}"{pos=0.42}, "\cong"'{pos=0.42}, shift left=1, between={0.3}{0.6}, Rightarrow, from=2-1, to=2-4,dashed]
        	\arrow["{P\alpha'}"{pos=0.42},"\cong"'{pos=0.42}, shift left=1, between={0.3}{0.6}, Rightarrow, from=2-7, to=2-10,dashed]
        	\arrow["{F\alpha}"{pos=0.3},"\cong"'{pos=0.3}, shift right=5, between={0.1}{0.5}, Rightarrow, from=3-1, to=3-4]
        	\arrow["Q"{pos=0.7}, shift left=5, between={0.4}{1}, maps to, from=3-4, to=3-6]
        	\arrow["{P\alpha}"{pos=0.3},"\cong"'{pos=0.3}, shift right=5, between={0.1}{0.5}, Rightarrow, from=3-7, to=3-10]
        	\arrow["Ff"', from=4-1, to=4-5]
        	\arrow["Pf"', from=4-7, to=4-11]
        	\arrow["\delta", shift left=5, between={0.2}{0.8}, Rightarrow,  from=0, to=1]
        	\arrow["{F\widehat{\delta}}", shift right=5, between={0.2}{0.8}, Rightarrow,  from=0, to=1]
        	\arrow["{F\sigma}"', shift right=5, between={0.2}{0.8}, Rightarrow, from=3, to=2]
        	\arrow["{Q\delta}", between={0.2}{0.8}, Rightarrow, from=4, to=5]
        	\arrow["{P\sigma}"', shift right=5, between={0.2}{0.8}, Rightarrow, from=7, to=6]
        \end{tikzcd}\]
        along the $Q$-cartesian morphism $Ff$, which must be unique by definition. The reasoning for the map \enumref{Def:An}{11} is more direct: if $\delta,\delta'\colon h\Rightarrow h'$ are two 2-cells in $\cP$ fitting into the corresponding diagram along the $P$-cartesian morphism $f$, and such that $F\delta=F\delta'$, then we will have that $P\delta=P\delta'$, and so the uniqueness of the 2-cells in the definition of the $P$-cartesian morphism $f$ ensures that $\delta=\delta'$. 
	\end{proof}

    \subsection{Cofibrations and trivial fibrations}\label{sec:cofibandtrivfib}

		This subsection focuses on the classes of marked cofibrations and marked trivial fibrations, introducing them and studying some useful properties. We start by defining the cofibrations,  and showing they are cofibrantly generated by a set.

		\begin{defn}\label{Def:cofibrations}
			A marked 2-functor in $\twocat^{+}_{/\cC^\sharp}$  is a \emph{marked cofibration} if its underlying 2-functor is a cofibration in the canonical model structure on $\twocat$. 
		\end{defn}

		The canonical model structure on $\twocat$ is cofibrantly generated, and a generating set of cofibrations was listed in \cref{Def:Lackgeneratingcofibs}. If we apply the minimal marking functor $(-)^{\flat}\colon\twocat \to \twocat^{+}$ to this set, we obtain marked cofibrations which provide a generating set for $\twocat^{+}_{/\cC^\sharp}$, as we now show.

		\begin{prop}\label{prop:genmarkedcofibs}
			The set $I$ consisting of the marked 2-functors over $\cC^\sharp$
                \begin{itemize}
                \item $i_1\colon \emptyset\to [0]$,
                \item $i_2\colon [0]\sqcup  [0]\to [1]^\flat$,
                \item $i_3\colon \Sigma([0]\sqcup  [0])^\flat\to \Sigma [1]^\flat$,
                \item $i_4\colon \Sigma(\{0 \rightrightarrows 1\})^\flat \to \Sigma[1]^\flat$,
                \item  $i_5\colon [1]^{\flat}\to [1]^{\sharp}$, and \item $i_6 \colon \Sigma{[1]^{\flat}} \to \Sigma{[1]^{\sharp}}$
            \end{itemize}
           is a set of generating marked cofibrations.
		\end{prop}
		\begin{proof}
            Marked cofibrations are defined simply in terms of their underlying 2-functors, and so it is easy to see that all the maps in $I$ are marked cofibrations. Moreover, we have that $\mathrm{cof}(UI)=\mathrm{cof}(UI\setminus \{Ui_5,U i_6\})$ (since $Ui_5$ and $Ui_6$ are identities), and the latter is precisely the class of all cofibrations in $\twocat$. Hence, the maps $i_1$-$i_4$ suffice to construct all the possible cofibrations with the minimal marking, and then pushouts with respect to $i_5$ and $i_6$ yield all possible choices of markings on these at the level of morphisms and 2-cells.
		\end{proof} 

		\begin{rmk}\label{AnisinCof}
			It is not hard to explicitly check that the underlying 2-functor of each marked anodyne extension $j_1,j_2,\dots,j_{10}$ described in \cref{Def:An} is a  cofibration in $\twocat$. Hence, by \cref{prop:genmarkedcofibs} we have that all these are also marked cofibrations.
		\end{rmk}

		We can now use the set of generating cofibrations to provide an explicit description of the class of trivial fibrations, by inspecting the right lifting properties with respect to each of the generating maps.

        \begin{defn}
           A marked 2-functor in $\twocat^{+}_{/\cC^\sharp}$ is a \emph{marked trivial fibration} if it has the right lifting property with respect to all marked cofibrations. 
        \end{defn}

		\begin{prop} \label{prop:char_trivfib}
			A marked 2-functor $F\colon \markedscat{\cP} \to \markedscat{\cQ}$ from $P$ to $Q$ is a marked trivial fibration if and only if the following hold:
			\begin{enumerate}[label=\stlabel{prop:char_trivfib}, ref= \arabic*]
				\item \label{prop:char_trivfib.1} The underlying 2-functor $UF$ is a trivial fibration in $\twocat$.
				\item \label{prop:char_trivfib.2} A morphism $f$ in $\cP$ is marked if and only if $Ff$ is marked in $\cQ$.
				\item \label{prop:char_trivfib.3} A $2$-cell $\alpha$ in $\cP$ is marked if and only if $F\alpha$ marked in $\cQ$.
			\end{enumerate}
		\end{prop}

		\begin{proof}
			If $F$ is a marked trivial fibration, then it has the right lifting property with respect to the generating marked cofibrations of \cref{prop:genmarkedcofibs}
			\[
				i_j\colon\mathcal A^\flat\to\mathcal B^\flat
				,\quad
				i_5\colon [1]^\flat\to[1]^\sharp
				,\quad\text{and}\quad
				i_6\colon \Sigma[1]^\flat\to\Sigma[1]^\sharp,
			\]
			where $Ui_j\colon \mathcal A\hookrightarrow \mathcal B$ ranges over the generating cofibrations of $\twocat$ (see \cref{Def:Lackgeneratingcofibs}).
			These three lifting properties correspond, respectively, to conditions \enumref{prop:char_trivfib}{1}, \enumref{prop:char_trivfib}{2} and \enumref{prop:char_trivfib}{3}. 
            
            For the converse, we must show that any commutative diagram
			\[
			\begin{tikzcd}
			\mathcal A \arrow[r, "S"] \arrow[d, "G"', hookrightarrow]     & \mathcal P \arrow[d, "F"] \\
			\mathcal B \arrow[r, "T"'] \arrow[ru, dashed,"\widehat T" description] & \cQ            
			\end{tikzcd},
			\]
			with $G$ a marked cofibration, admits a lift $\widehat{T}$.
			The 2-functor $\widehat T$ exists since $UF$ is a trivial fibration in $\twocat$ by assumption, and $UG$ is a cofibration in $\twocat$ by definition, so it remains to check that it preserves the markings.
			Indeed, if $f\in\mathcal B$ is marked then so is $Tf = F(\widehat Tf)$, which by hypothesis implies that $\widehat Tf$ is marked.
			A similar argument applies for marked 2-cells.
		\end{proof}

        In the case when $P\colon \markedscat{\cP}\to\cC^\sharp$ and $Q\colon\markedscat{\cQ}\to\cC^\sharp$ are cart-marked 2-fibrations, the trivial fibrations between them admit a more streamlined description.

		\begin{prop}\label{cor:trivfibbtwfibobj}
			A marked 2-functor $F\colon\markedscat{\cP}\to \markedscat{\cQ}$ between naive fibrant objects $P$ and $Q$ is a marked trivial fibration if and only if $UF$ is a  trivial fibration in $\twocat$.
		\end{prop}

		\begin{proof}
			The fact that a marked trivial fibration (between any  objects) has an underlying 2-functor that is a trivial fibration in $\twocat$ is directly implied by \cref{prop:char_trivfib}. For the converse, we must verify conditions \enumref{prop:char_trivfib}{2} and \enumref{prop:char_trivfib}{3} in \cref{prop:char_trivfib}. Condition \enumref{prop:char_trivfib}{3} holds due to the analogous relation between trivial fibrations and cartesian maps in the 1-categorical case; see \cite[Lemma 4.3]{mosersarazola2023model}. For condition \enumref{prop:char_trivfib}{2}, first note that if $f$ is a marked morphism then $Ff$ is also marked, since $F$ is a marked 2-functor. It remains to show that if $Ff$ is a marked morphism, then $f$ must be marked.

    		Suppose that $Ff\colon Fx\to Fy$ is marked in $\cQ$, i.e.\ it is cartesian, since $\markedscat{\cQ}$ is cart-marked. We need to show that $f\colon x\to y$ is cartesian, that is, it verifies both conditions in \cref{Def:cartesian_1cells}. For the first condition \enumref{Def:cartesian_1cells}{1}, suppose we have the following data
			\[\begin{tikzcd}[ampersand replacement=\&,row sep=scriptsize]
				z \&\&\&\& Fz \\
				\&\& {} \&\& {} \\
				x \&\& y \&\& Fx \&\& Fy
				\arrow["g", from=1-1, to=3-3]
				\arrow["h"', from=1-5, to=3-5]
				\arrow[""{name=0, anchor=center, inner sep=0}, "Fg", from=1-5, to=3-7]
				\arrow["F"{pos=0.3}, between={0}{0.6}, maps to, from=2-3, to=2-5]
				\arrow["f"', from=3-1, to=3-3]
				\arrow["Ff"', from=3-5, to=3-7]
				\arrow["\alpha","\cong"', between={0.3}{0.7}, Rightarrow, from=3-5, to=0]
			\end{tikzcd}\] 
			Since $UF$ is a trivial fibration, it is in particular surjective on morphisms up to an invertible 2-cell, so there exists $\widehat{h}\colon z\to x$ in $\cP$ together with an invertible $2$-cell $\widehat{\beta}\colon F\widehat{h} \Rightarrow h$ in $\cQ$. Additionally, $UF$ is an equifibration, and so there exists an invertible $2$-cell $\alpha'\colon\widehat{h}\Rightarrow h'$ in $\cP$ with $F(h')=h$ and $F\alpha' = \widehat{\beta}$. Hence $\alpha\in\cQ(F(fh'),Fg)$, and as trivial fibrations in $\twocat$ are fully faithful on hom-categories, there exists an invertible $2$-cell $\widehat{\alpha}\colon fh'\Rightarrow g$ such that $F\widehat{\alpha}=\alpha$, as required. 

			For the second condition \enumref{Def:cartesian_1cells}{2}, suppose we have prescribed lifts $(\widehat{h},\widehat{\alpha},\widehat{\beta})$ and $(\widehat{h'},\widehat{\alpha'},\widehat{\beta'})$ of $(h,\alpha)$ and $(h',\alpha')$ respectively. Let us first assume that $\widehat{\beta}$ and $\widehat{\beta'}$ are identities and thus $F\widehat{h}=h$, $F\widehat{h'}=h'$, $F\widehat{\alpha}=\alpha$ and $F\widehat{\alpha'}=\alpha'$. Given a 2-cell $\delta \colon h \Rightarrow h'$ with $\alpha'.(Ff \cdot \delta) = F\sigma.\alpha$ as depicted below
			 
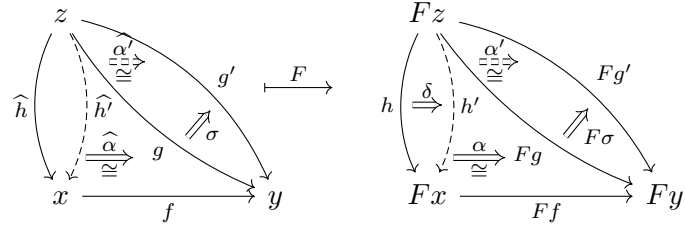
\[\begin{tikzcd}[ampersand replacement=\&,
      column sep=14pt,
    row sep=2.6ex]
    	z \&\&\&\&\&\& Fz \\
    	{} \&\&\& {} \&\&\& {} \&\&\& {} \&\\
    	{} \&\&\& {} \&\& {} \& {} \&\&\& {} \& \\
    	x \&\&\&\& y \&\& Fx \&\&\&\& Fy 
    	\arrow["{\widehat{h'}}", curve={height=-12pt}, from=1-1, to=4-1, dashed]
    	\arrow["\widehat{h}"', curve={height=12pt}, from=1-1, to=4-1]
    	\arrow[""{name=0, anchor=center, inner sep=0}, "g"'{pos=0.6}, curve={height=12pt}, from=1-1, to=4-5]
    	\arrow[""{name=1, anchor=center, inner sep=0}, "{g'}"{pos=0.6}, curve={height=-18pt}, from=1-1, to=4-5]
    	\arrow[""{name=2, anchor=center, inner sep=0}, "h"', curve={height=12pt}, from=1-7, to=4-7]
    	\arrow[""{name=3, anchor=center, inner sep=0}, "{h'}", curve={height=-12pt}, from=1-7, to=4-7,dashed]
    	\arrow[""{name=4, anchor=center, inner sep=0}, "{Fg'}"{pos=0.6}, curve={height=-18pt}, from=1-7, to=4-11]
    	\arrow[""{name=5, anchor=center, inner sep=0}, "Fg"'{pos=0.6}, curve={height=12pt}, from=1-7, to=4-11]
    	\arrow["{\widehat{\alpha'}}"{pos=0.42},"\cong"'{pos=0.42}, shift left=1, between={0.3}{0.6}, Rightarrow, from=2-1, to=2-4,dashed]
    	\arrow["{\alpha'}"{pos=0.42}, "\cong"'{pos=0.42}, shift left=1, between={0.3}{0.6}, Rightarrow, from=2-7, to=2-10,dashed]
    	\arrow["\widehat{\alpha}"{pos=0.3}, "\cong"'{pos=0.3}, shift right=5, between={0.1}{0.5}, Rightarrow, from=3-1, to=3-4]
    	\arrow["F"{pos=0.7}, shift left=5, between={0.4}{1}, maps to, from=3-4, to=3-6]
    	\arrow["{\alpha}"{pos=0.3}, "\cong"'{pos=0.3}, shift right=5, between={0.1}{0.5}, Rightarrow, from=3-7, to=3-10]
    	\arrow["f"', from=4-1, to=4-5]
    	\arrow["Ff"', from=4-7, to=4-11]
    	\arrow["\sigma"', shift right=5, between={0.2}{0.8}, Rightarrow, from=0, to=1]
    	\arrow["\delta", between={0.2}{0.8}, Rightarrow, from=2, to=3]
    	\arrow["{F\sigma}"', shift right=5, between={0.2}{0.8}, Rightarrow, from=5, to=4]
    \end{tikzcd}\]
			we can use the fact that $F$ is fully faithful on hom-categories to get a unique $\widehat{\delta}\colon\widehat{h}\Rightarrow \widehat{h'}$ such that $F\widehat{\delta}=\delta$. Hence, we can see that
			\[F(\widehat{\alpha'}.(f \cdot \widehat{\delta})) = F\widehat{\alpha'}.(Ff \cdot F\widehat{\delta})=\alpha'.(Ff\cdot \delta) = F\sigma.\alpha=F(\sigma.\widehat{\alpha}),\] and by faithfulness we must have $\widehat{\alpha'}.(f \cdot \widehat{\delta}) = \sigma.\widehat{\alpha}$ as desired. In the case where $\beta$ and $\beta'$ may be arbitrary invertible 2-cells, the argument proceeds similarly, first considering instead the lifting problem where $h$ and $h'$ are replaced by $P\widehat{h}$ and $P\widehat{h'}$ and $\delta$ is replaced by $\beta'^{-1}.\delta.\beta$. This is a strict lifting problem and hence falls within the first case we considered, so there is some $\widehat{\delta}\colon \widehat{h}\Rightarrow \widehat{h'}$ which solves it; a straightforward verification then shows that the same $\widehat{\delta}$ also solves the original lifting problem.
		\end{proof}

        In any model structure, one automatically has that if a map $X\to Y$ is a (trivial) fibration and $Y$ is fibrant, then so is $X$. As an interesting consequence of \cref{prop:char_trivfib}, we obtain the following result, which will be useful in the next subsection.

        \begin{prop}\label{prop:mtf_sends_cart_to_cart}
			Let $F\colon \markedscat{\cP} \to \markedscat{\cQ}$ be a marked trivial fibration, where $P\colon \markedscat{\cP}\to \cC^{\sharp}$ is a cart-marked 2-fibration. Then $Q\colon \markedscat{\cQ}\to \cC^{\sharp}$ is also a cart-marked 2-fibration.
		\end{prop}
		\begin{proof}
			Let us first prove that $Q$ is a $2$-fibration, by checking the conditions in \cref{Def:2-fibration}. For the first condition, let $f\colon x'\to Qy$ be a morphism in $\cC$; we wish to find a $Q$-cartesian lift of $f$. Since the underlying 2-functor of $F$ is a trivial fibration in $\twocat$ by \cref{prop:char_trivfib}, we know that $F$ is essentially surjective and so there exists an object $x \in \cP$ and an equivalence $q\colon y \to Fx$ in $\cQ$. Moreover, as $F$ is also an equifibration,  there exists an equivalence $p\colon\widehat{y} \to x$ in $\cP$ such that $Fp =q$. In particular, $F\widehat{y}=y$ and therefore $P\widehat{y}=QF\widehat{y} = Qy$. This means that $f$ is a morphism from $x'$ to $Qy = P\widehat{y}$, and as $P$ is a $2$-fibration, $f$ admits a $P$-cartesian lift $\widehat{f}\colon \widehat{x} \to \widehat{y}$. Notice that $QF\widehat{f} = P\widehat{f} = f$; thus $F\widehat{f}$ is a $Q$-lift of $f$, which is moreover $Q$-cartesian by \cref{lem:trivialLackfibs_and_cartesians}. 

			Next, we show that $Q$ is locally a cartesian fibration. For this, first note that given a marked trivial fibration $F\colon \markedscat{\cP} \to \markedscat{\cQ}$ in $\twocat^{+}_{/\cC^\sharp}$, the hom-functors $F_{a,a'}\colon \cP(x,x')\to\cQ(Fx,Fx')$ are marked trivial fibrations in $\cat^{+}_{/\cC(Px,Px')^\sharp}$ in the sense of \cite[Definition 4.15]{mosersarazola2023model}, since the underlying 2-functor of $F$ is a trivial fibration in $\twocat$ and thus locally an isofibration (see \cite[Corollary 4.18]{mosersarazola2023model}). Since $F_{x,x'}$ is a marked trivial fibration in $\cat^{+}_{/\cC(Px,Px')^\sharp}$ by \cref{prop:char_trivfib} and \cite[Proposition 4.16]{mosersarazola2023model}, and $P_{x,x'}\colon\cP(x,x')\to\cC(Px,Px')$ is a cart-marked fibration, we can use  \cite[Proposition 4.20]{mosersarazola2023model} to conclude that the functor $Q_{Fx,Fx'}\colon\cQ(Fx,Fx')\to\cC(Px,Px')$ is also a cart-marked fibration. Hence, $Q$ is locally a cartesian fibration and the marked $2$-cells in $E^2_{\cQ}$ are precisely the $Q$-cartesian $2$-cells. 
			
			 To show that any marked morphism $q\colon x\to y$ in $E^1_{\cQ}$ is  $Q$-cartesian, 
			 suppose we have some $g\colon z\to y$ in $\cQ$, together with a morphism $h\colon Qz \to Qx$ and an invertible $2$-cell $\alpha\colon (Qq)h \Rightarrow Qg$ in $\cC$. Since $F$ is surjective on objects and full on morphisms by \cref{char:lacktrivfibs}, there exist morphisms $\widehat{q}\colon x'\to y'$ and $\widehat{g}\colon z' \to y'$ in $\cP$ with $F\widehat{q} = q$ and $F\widehat{g}=g$, and \cref{prop:char_trivfib} then implies that $\widehat{q}$ is marked in $\cP$, hence $P$-cartesian. Then, there exists a morphism $\widehat{h}\colon z'\to x'$ and an invertible $2$-cell $\widehat{\alpha}\colon\widehat{q} \widehat{h} \Rightarrow \widehat{g}$ in $\cP$ with $\alpha=P\widehat{\alpha}$. One can then use the fact that $P=QF$ to see that  $(F\widehat{h},F\widehat{\alpha})$ is a solution to the initial lifting problem in $\cQ$. A similar argument applies for the second condition in \cref{Def:cartesian_1cells}, using the surjectivity on objects and fullness on morphisms and 2-cells of $F$ to carry the lifting problem back to $\cP$, and the faithfulness of $F$ on 2-cells to ensure that the required pasting equations among the 2-cells in $\cP$ also hold.
			
			Conversely, we must also show that any $Q$-cartesian morphism $q\colon x\to y$ in $\cQ$ is marked. As we know that there is some $\widehat{q}\colon x'\to y'$ in $\cP$ with $F\widehat{q}=q$, by \cref{prop:char_trivfib} it suffices to show that $\widehat{q}$ is marked, i.e.\ $P$-cartesian. Given a lifting problem as below,
        \[\begin{tikzcd}[ampersand replacement=\&,sep=scriptsize]
        	{z'} \&\&\&\& {Fz'} \&\&\&\& {QFz'} \\
        	\&\& {} \&\& {} \&\& {} \&\& {} \\
        	{x'} \&\& {y'} \&\& x \&\& y \&\& Qx \&\& Qy
        	\arrow["{g'}", from=1-1, to=3-3]
        	\arrow["{Fg'}", from=1-5, to=3-7]
        	\arrow["h"', from=1-9, to=3-9]
        	\arrow[""{name=0, anchor=center, inner sep=0}, "{Pg'}", from=1-9, to=3-11]
        	\arrow["F", between={0.2}{0.8}, maps to, from=2-3, to=2-5]
        	\arrow["Q", between={0.2}{0.8}, maps to, from=2-7, to=2-9]
        	\arrow["{\widehat{q}}"', from=3-1, to=3-3]
        	\arrow["q"', from=3-5, to=3-7]
        	\arrow["{Qq = P\widehat{q}}"', from=3-9, to=3-11]
        	\arrow["\alpha"',"\cong", between={0.2}{0.8}, Rightarrow, from=3-9, to=0]
        \end{tikzcd}\]
        since $q$ is $Q$-cartesian there is a lift $(\widehat{h}\colon Fz'\to x,\widehat{\alpha}\colon q\widehat{h} \xRightarrow{\cong} Fg')$ in $\cQ$ with $\alpha=Q\widehat{\alpha}$. Using the fact that $F$ is a trivial fibration, we can lift $(\widehat{h},\widehat{\alpha})$ to $(h'\colon z'\to x', \alpha'\colon \widehat{q}h'\xRightarrow{\cong} g')$ in $\cP$. Hence $Ph'= QFh' = Q\widehat{h}=h$ and $P\alpha'=\alpha$, producing the desired $P$-lifting along $\widehat{q}$. The second condition in \cref{Def:cartesian_1cells} follows in a similar way.

			Finally, it remains to show that the horizontal composite of $Q$-cartesian $2$-cells is $Q$-cartesian. Given composable $Q$-cartesian $2$-cells $\alpha$ and $\beta$, we can once more use the fact that $F$ is a trivial fibration to lift them via $F$ to $\widehat\alpha$ and $\widehat\beta$ in $\cP$. These are cartesian by \cref{prop:char_trivfib} as we now know that both $\markedscat{\cP}$ and $\markedscat{\cQ}$ are cart-marked, and since $P$ is a $2$-fibration we have that $\widehat\alpha.\widehat\beta$ is cartesian; hence, so is $F(\widehat\alpha.\widehat\beta) = \alpha.\beta$.
		\end{proof}

        We conclude this subsection with a technical result regarding the stability of trivial fibrations under certain pushouts.

		\begin{prop}\label{Prop:pushout_tf_along_cofib} 
			The pushout of a marked trivial fibration in $\twocat^{+}_{/\cC^\sharp}$ along a marked cofibration is a marked trivial fibration.
		\end{prop}

		\begin{proof}
		We use the characterization of marked trivial fibrations from \cref{prop:char_trivfib}. Consider a pushout square in $\twocat^{+}_{/\cC^\sharp}$
		\[\begin{tikzcd}
			\cP & {\widehat{\cP}} \\
			\cQ & \mathcal{Z}
			\arrow["I", from=1-1, to=1-2]
			\arrow["F"', from=1-1, to=2-1]
			\arrow["{\widehat{F}}", from=1-2, to=2-2]
			\arrow["{\widehat{I}}"', from=2-1, to=2-2]
			\arrow["\lrcorner"{anchor=center, pos=0.125, rotate=180}, draw=none, from=2-2, to=1-1]
		\end{tikzcd}\]
		where $F$ is a marked trivial fibration and $I$ is a marked cofibration; we omit the sets of markings in order to lighten the notation. By \cref{Def:cofibrations} and \cref{prop:char_trivfib}, the underlying 2-functors  $UI$ and $UF$ are respectively a cofibration and a trivial fibration in $\twocat$, and so by Step 2 in the proof of \cite[Theorem 6.3]{Lack2002} we obtain that $U\widehat{F}$ is a trivial fibration in $\twocat$. It remains to check the second and third conditions of \cref{prop:char_trivfib}, which follow from the fact that the sets of marked morphisms and $2$-cells of the pushout $\mathcal{Z}$ are given by the corresponding pushouts of the sets of marked morphisms and $2$-cells. 
		\end{proof}

\subsection{Weak equivalences}\label{sec:weakequivs}
		To introduce the class of weak equivalences, we follow the strategy of \cite{guetta2023fibrantlyinduced}: we first define the weak equivalences between naive fibrant objects, and then use them together with the class of marked anodyne extensions of \cref{Def:An} to define weak equivalences between arbitrary objects as in \cite[Definition 2.4]{guetta2023fibrantlyinduced}.

        \begin{defn}
            A marked 2-functor $F\colon\markedscat{\cP}\to \markedscat{\cQ}$ between naive fibrant objects in $\twocat^{+}_{/\cC^\sharp}$ is a \emph{marked weak equivalence} if its underlying 2-functor $UF\colon\cP\to\cQ$ is a biequivalence. 
        \end{defn}

		\begin{defn}\label{Def:marked_weak_equiv}
			A marked 2-functor $F \colon (\cP,E^1_{\cP},E^2_{\cP}) \to (\cQ,E^1_{\cQ},E^2_{\cQ})$ in $\twocat^{+}_{/\cC^\sharp}$ between objects $P$ and $Q$ is a \emph{marked weak
			equivalence} if there exists a marked naive fibrant replacement of $F$ whose underlying 2-functor is a biequivalence.
			More explicitly, if there is a commutative diagram of marked 2-functors over $\cC^\sharp$
			\[
				\begin{tikzcd}
				\markedscat{\cP} \arrow[r,"F"] \arrow[d,"\iota_{\mathcal P}"'] & \markedscat{\cQ} \arrow[d,"\iota_{\cQ}"]
				\\
				\markedscat{\widehat{\cP}}\arrow[r,"\widehat F"'] & \markedscat{\widehat{\cQ}} 
				\end{tikzcd}
			\]
			where $\widehat{P}\colon \markedscat{\widehat{\cP}}\to\cC^\sharp$ and  $\widehat{Q}\colon \markedscat{\widehat{\cQ}}\to\cC^\sharp$ are cart-marked 2-fibrations, $U\widehat F$ is a biequivalence, and $\iota_{\mathcal P}$ and $\iota_{\cQ}$ are marked anodyne extensions.
		\end{defn}

        As we have now introduced classes of marked trivial fibrations and marked weak equivalences in $\twocat^{+}_{/\cC^\sharp}$, we must check that they satisfy the required compatibility; this is the content of the next result.

	\begin{prop}
	\label{prop:trivial_fibration_is_w}
	Every marked trivial fibration is a marked weak equivalence.
	\end{prop}
	\begin{proof}
	Let $F\colon \markedscat{\cP} \to \markedscat{\cQ}$ be a  marked trivial fibration in $\twocat^{+}_{{/\Csharp}}$ from $P$ to $Q$; we need to find a marked naive fibrant replacement of $F$ whose underlying functor is a biequivalence. First, let $\iota_{\cP}\colon\markedscat{\cP}\to \markedscat{\widehat{\cP}}$ be a naive fibrant replacement of  $P\colon\markedscat{\cP}\to \cC^{\sharp}$ in $\twocat^{+}_{/\Csharp}$. If we consider the pushout of $F$ along $\iota_{\cP}$: 

	\[\begin{tikzcd}
		{\markedscat{\cP}} & {\markedscat{\cQ}} \\
		{\markedscat{\widehat{\cP}}} & {\markedscat{\mathcal{Z}}}
		\arrow["F", from=1-1, to=1-2]
		\arrow["\iota_{\cP}"', from=1-1, to=2-1]
		\arrow["{\iota'}", from=1-2, to=2-2]
		\arrow["{\widehat{F}}"', from=2-1, to=2-2]
		\arrow["\lrcorner"{anchor=center, pos=0.125, rotate=180}, draw=none, from=2-2, to=1-1]
	\end{tikzcd}\]
	we know by \cref{Prop:pushout_tf_along_cofib} that the marked 2-functor $\widehat{F}$ is a marked trivial fibration, so that we have a commutative triangle 
	\[\begin{tikzcd}
		{\markedscat{\widehat{\cP}}} && {\markedscat{\mathcal{Z}}} \\
		& {\cC^{\sharp}}
		\arrow["{\widehat{F}}","\sim"', from=1-1, to=1-3, twoheadrightarrow]
		\arrow[from=1-1, to=2-2,"\widehat{P}"']
		\arrow[from=1-3, to=2-2,"Z"]
	\end{tikzcd}\]
	where $\widehat{P}$ is a cart-marked 2-fibration; hence, so is $Z$ by \cref{prop:mtf_sends_cart_to_cart}. Note that the marked 2-functor $\iota'$ is a marked anodyne extension, as it is a pushout of $\iota_\cP$. Therefore $Z\colon\markedscat{\mathcal{Z}}\to\cC^\sharp$ is a marked naive fibrant replacement of $Q\colon\markedscat{\cQ}\to\Csharp$, and so $\widehat{F}$ is a marked naive fibrant replacement of $F$. Lastly, the fact that the 2-functor $U\widehat{F}$ is a biequivalence follows from \cref{cor:trivfibbtwfibobj}. 
	\end{proof}

    \subsection{Path objects}\label{sec:pathobject}

	The next step towards the construction of the desired model structure is to establish the existence of path objects for naive fibrant objects. We start by constructing a candidate, and then show it has all the required properties.

    Given an object $X$ in a model category $\cM$, by \emph{path object} we mean some object $\Path(X)$ in $\cM$, together with a factorization of the diagonal map $X\to X\times X$ as below
    \[\begin{tikzcd}
        X\ar[rr]\ar[dr,"W"']  & & X\times X\\
        & \Path(X)\ar[ur,"R"']
    \end{tikzcd}\]
    where $W$ is a weak equivalence and $R$ is a fibration.

    \begin{rmk}
        Recall that the cartesian product of two objects $P\colon\cP\to\cC$, $Q\colon\cQ\to\cC$ in the category $\twocat_{/\cC}$ is given by the (strict) pullback $\cP\times_{\cC}\cQ$ over $\cC$; i.e.\ the 2-category whose objects (resp.\ morphisms, 2-cells) are the pairs of objects (resp.\ morphisms, 2-cells) in $\cP$ whose images under $P$ are \emph{equal} in $\cC$. If we add markings, we get that the cartesian product of $P\colon\markedscat{\cP}\to\Csharp$, $Q\colon\markedscat{\cQ}\to\Csharp$ in the category $\twocat^{+}_{/\Csharp}$ is given by the pullback 2-category $\cP\times_{\cC}\cQ$, whose set of marked morphisms (resp.\ marked 2-cells) is the pullback in $\Set$ below left (resp.\ below right),
        \[\begin{tikzcd}
		E^1_\cP\times_{\operatorname{mor}\cC} E^1_\cQ\dar\rar \drar["\lrcorner"{anchor=center, pos=0.125}, draw=none] & E^1_\cQ\dar["Q"]\\
        E^1_\cP\rar["P"] & \operatorname{mor}\cC
	\end{tikzcd}\qquad \qquad
    \begin{tikzcd}
		E^2_\cP\times_{\operatorname{2cell}\cC} E^2_\cQ\dar\rar \drar["\lrcorner"{anchor=center, pos=0.125}, draw=none] & E^2_\cQ\dar["Q"]\\
        E^2_\cP\rar["P"] & \operatorname{2cell}\cC
	\end{tikzcd}\]
    which sits over $\Csharp$ in a canonical way.
    \end{rmk}

	\begin{constr}[Path object]\label{construction:path_object}
	Let $P\colon\markedscat{\cP}\to\cC^\sharp$ be a cart-marked 2-fibration.
	We construct a diagram over $\cC^\sharp$
	\[
	\markedscat{\cP}\xrightarrow{W} (\Path(P),\widehat{E^1},\widehat{E^2})\xrightarrow{R} (\cP\times_{\cC}\cP,E^1_{\cP\times_{\cC}\cP}, E^2_{\cP\times_{\cC}\cP}).
	\]
	First, the object $(\Path(P),\widehat{E^1},\widehat{E^2})$ is the marked 2-category whose
\begin{itemize}[leftmargin=*]
	\item
    objects are the equivalences $(\varphi, \varphi^{-1},\eta,\epsilon)$ in $\cP$  with $\varphi\colon x\to x'$, for which $P\varphi=\id_{Px}=P\varphi^{-1}$, and  $P\eta=\id=P\epsilon$. We refer to them as $\varphi$ for simplicity;
	\item
	 morphisms $(\varphi\colon x\to x')\to (\psi\colon y\to y')$ are triples $(f,f',\gamma)$ consisting of morphisms $f\colon x\to y$ and $f'\colon x'\to y'$ in $\cP$ such that $Pf = Pf'$, and an invertible 2-cell $\gamma\colon f'\varphi \Rightarrow \psi f$ with $P\gamma=\id_{Pf}$;

	\item
	2-cells $(f,f',\gamma)\Rightarrow (g,g',\delta)$ between morphisms as above consist of 2-cells $\sigma\colon f\Rightarrow g$ and $\sigma' \colon f'\Rightarrow g'$ in $\cP$ such that $  (\psi \cdot \sigma). \gamma = \delta.(\sigma' \cdot \varphi)$ and $P\sigma=P\sigma'$;

	\item
	marked morphisms are those where both $f$ and $f'$ are marked in $\cP$, and
	\item
	marked 2-cells are those where both $\sigma$ and $\sigma'$ are marked in $\cP$.
	\end{itemize}

	It is a tedious but straightforward verification that this truly forms a marked 2-category, and that the correspondence sending $(\varphi\colon x\to x')\mapsto Px$, $(f,f',\gamma)\mapsto Pf$, and $(\sigma,\sigma')\mapsto P\sigma$ gives a marked 2-functor $\pi\colon\Path(P)\to\Csharp$.

	Next, the marked 2-functor $W\colon \cP\to\Path(P)$ is given by the action $a\mapsto \id_a$ on objects, $\varphi \mapsto (\varphi,\varphi,\id_{\varphi})$ on morphisms, and $\sigma \mapsto (\sigma,\sigma)$ on 2-cells.

	Lastly, the marked 2-functor $R\colon \Path(P)\to\cP\times_{\cC}\cP$ maps an object $\varphi\colon x\xrightarrow{\sim} x'$ to $(x,x')$, a morphism $(f,f',\gamma)$ to $(f,f')$, and a 2-cell $(\sigma,\sigma')$ to itself.
	\end{constr}
	
    \begin{rmk}\label{Path2cells}
        The fact that the 2-cells $(\sigma,\sigma')\colon (f,f',\gamma)\Rightarrow (g,g',\delta)$ in $\Path(P)$ must satisfy $(\psi \cdot \sigma). \gamma = \delta.(\sigma' \cdot \varphi)$, where $\gamma$ and $\delta$ are invertible 2-cells and $(\varphi,\varphi^{-1},\eta,\epsilon)$ and $(\psi, \psi^{-1},\eta',\epsilon')$ are equivalences, implies that $\sigma'$ is determined by the rest of the data; indeed, one can check that
        \[\sigma'=(g'\cdot\epsilon).(\delta^{-1}\cdot \varphi^{-1}).(\psi\cdot\sigma\cdot \varphi^{-1}).(\gamma\cdot \varphi^{-1}).(f\cdot\epsilon^{-1}).\]
        Moreover, the identity requirements on objects and morphisms guarantee that $P\sigma=P\sigma'$.
    \end{rmk}

	Our goal is to show that $\Path(P)$ gives a path object for a given naive fibrant object $P$, for which we must prove that the 2-functors $W$ and $R$ of \cref{construction:path_object} are a weak equivalence and a fibration, respectively. Since our most explicit descriptions of these classes of maps take place when we work between fibrant objects, we start by showing that the path object of a naive fibrant object is itself naive fibrant, for which it will be crucial to identify the cartesian morphisms and 2-cells in $\Path(P)$.

\begin{lem}\label{lemma:markings}
	A morphism $(f,f',\gamma)$ (resp.\ a 2-cell $(\sigma,\sigma')$) in $\Path(P)$ is $\pi$-cartesian if and only if the morphisms $f,f'$ (resp.\ the 2-cells $\sigma,\sigma'$) are $P$-cartesian.
\end{lem}

\begin{proof}
	We prove the claim about morphisms, as the one for 2-cells is simpler and follows very similar ideas. Let $(f,f',\gamma)$ be a morphism in $\Path(P)$ from $\varphi\colon x\xrightarrow{\sim} x'$ to $\psi \colon y\xrightarrow{\sim} y'$, and first suppose it is $\pi$-cartesian. To show that $f$ is $P$-cartesian, we consider the lifting problem below left.
	\[\begin{tikzcd}[row sep=scriptsize, column sep=scriptsize]
				z &&&& Pz \\
				&& {} && {} \\
				x && y && Px && Py
				\arrow["g", from=1-1, to=3-3]
				\arrow["h"', from=1-5, to=3-5]
				\arrow[""{name=0, anchor=center, inner sep=0}, "Pg", from=1-5, to=3-7]
				\arrow["P"{pos=0.3}, between={0}{0.6}, maps to, from=2-3, to=2-5]
				\arrow["f"', from=3-1, to=3-3]
				\arrow["Pf"', from=3-5, to=3-7]
				\arrow["\alpha","\cong"', between={0.3}{0.7}, Rightarrow, from=3-5, to=0]
	\end{tikzcd}\qquad \qquad
    \begin{tikzcd}[row sep=scriptsize, column sep=scriptsize]
				\id_z &&&& Pz \\
				&& {} && {} \\
				\varphi && \psi && Px && Py
				\arrow["{(g,\psi g,\id)}"{pos=0.35}, from=1-1, to=3-3]
				\arrow["h"', from=1-5, to=3-5]
				\arrow[""{name=0, anchor=center, inner sep=0}, "Pg", from=1-5, to=3-7]
				\arrow["\pi"{pos=0.3}, between={0}{0.6}, maps to, from=2-3, to=2-5]
				\arrow["{(f,f',\gamma)}"', from=3-1, to=3-3]
				\arrow["Pf"', from=3-5, to=3-7]
				\arrow["\alpha","\cong"', between={0.3}{0.7}, Rightarrow, from=3-5, to=0]
	\end{tikzcd}\] 
	This data gives rise to a lifting problem from $\Path(P)$ as above right, which since $(f,f',\gamma)$ is $\pi$-cartesian, admits a lift, given by a morphism $(\widehat{h},\widehat{h'},\beta)\colon \id_z\to \varphi$ and an invertible 2-cell $(\widehat{\alpha},\widehat{\alpha'})\colon (f\widehat{h},f'\widehat{h'}, (\gamma\cdot\widehat{h})(f'\cdot\beta))\Rightarrow (g,\psi g,\id)$. In particular, we have that $P\widehat{h}=h$ and $P\widehat{\alpha}=\alpha$, and so $\widehat{\alpha}\colon f\widehat{h}\xRightarrow{\cong} g$ gives the desired $P$-lift.

 	The second condition in \cref{Def:cartesian_1cells} proceeds similarly, since a lifting problem from $\cP$ as in \enumref{Def:cartesian_1cells}{2} involving some 2-cell $\sigma\colon h\Rightarrow h'$ gives rise to a corresponding lifting problem from $\Path(P)$, which admits a unique lift $(\widehat{\sigma},\widehat{\sigma'})$; then $\widehat{\sigma}$ is the $P$-lift we seek. Moreover, this is unique, as a different $P$-lift $\widehat{\sigma_2}$ would give rise to a different $\pi$-lift $(\widehat{\sigma_2},\widehat{\sigma'_2})$ following \cref{Path2cells}. This concludes the proof that $f$ is $P$-cartesian; hence so is $f'$ by \cite[Proposition 3.1.8]{Buckley2014}, as we can construct an invertible 2-cell $f' \xRightarrow{\cong} \varphi'f \varphi^{-1}$ using $\alpha$ and the equivalence data of $f$ and $f^{-1}$. 

	For the converse, we want to show that a morphism $(f,f',\gamma)$ in $\Path(P)$ from $\varphi\colon x\xrightarrow{\sim} x'$ to $\psi\colon y\xrightarrow{\sim} y'$ is $\pi$-cartesian if $f,f'$ are $P$-cartesian. Starting with a lifting problem in $\Path(P)$ as below left, for some objects $\varphi\colon x\to x'$, $\psi\colon y\to y'$ and $\rho\colon z\to z'$
\[\begin{tikzcd}[row sep=scriptsize, column sep=scriptsize]
				\rho &&&& Pz \\
				&& {} && {} \\
				\varphi && \psi && Px && Py
				\arrow["{(g,g',\delta)}"{pos=0.35}, from=1-1, to=3-3]
				\arrow["h"', from=1-5, to=3-5]
				\arrow[""{name=0, anchor=center, inner sep=0}, "Pg", from=1-5, to=3-7]
				\arrow["\pi"{pos=0.3}, between={0}{0.6}, maps to, from=2-3, to=2-5]
				\arrow["{(f,f',\gamma)}"', from=3-1, to=3-3]
				\arrow["Pf"', from=3-5, to=3-7]
				\arrow["\alpha","\cong"', between={0.3}{0.7}, Rightarrow, from=3-5, to=0]
			\end{tikzcd}
            \qquad \qquad
            \begin{tikzcd}[row sep=scriptsize, column sep=scriptsize]
				z &&&& Pz \\
				&& {} && {} \\
				x && y && Px && Py
				\arrow["g", from=1-1, to=3-3]
				\arrow["h"', from=1-5, to=3-5]
				\arrow[""{name=0, anchor=center, inner sep=0}, "Pg", from=1-5, to=3-7]
				\arrow["P"{pos=0.3}, between={0}{0.6}, maps to, from=2-3, to=2-5]
				\arrow["f"', from=3-1, to=3-3]
				\arrow["Pf"', from=3-5, to=3-7]
				\arrow["\alpha","\cong"', between={0.3}{0.7}, Rightarrow, from=3-5, to=0]
			\end{tikzcd}\]     
	we get a lifting problem in $\cP$ as above right, which admits a $P$-lift $\widehat{h}\colon z\to x$, $\widehat{\alpha}\colon f \widehat{h}\xRightarrow{\cong} g$ since $f$ is $P$-cartesian. We claim that this can be used to define a $\pi$-lift of the original problem, given by the morphism 
	\[(\widehat{h}\colon z\to x, \varphi\widehat{h}\rho^{-1}\colon z'\to x',\varphi\widehat{h}\cdot\eta_{\rho}^{-1}\colon \varphi\widehat{h}\rho^{-1}\rho\xRightarrow{\cong} \varphi\widehat{h})\colon \rho\to \varphi\] 
	and the invertible 2-cell
	\[(\widehat{\alpha},\widehat{\alpha'})\colon (f \widehat{h}, f' \varphi \widehat{h} \rho^{-1}, (f' \varphi\widehat{h}\cdot\eta_{\rho}^{-1})(\gamma\cdot \widehat{h}))\Rightarrow (g,g',\delta),\] 
	where $\widehat{\alpha'}$ is determined from $\widehat{\alpha}$ and the rest of the data as in \cref{Path2cells}. To prove this claim, it suffices to show that the above defines a morphism in $\Path(P)$, as $(\widehat{\alpha},\widehat{\alpha'})$ is a 2-cell in $\Path(P)$ by construction, and the fact that they give a $\pi$-lift is also immediate from construction. The morphism requirements can be deduced directly from the fact that $P\varphi,P\rho^{-1}$ and $P\eta_{\rho}^{-1}$ equal the corresponding identities.

	To check that $(f,f',\gamma)$ verifies the second condition in \cref{Def:cartesian_1cells}, suppose we have a lifting problem as in \enumref{Def:cartesian_1cells}{2} which includes a 2-cell $\sigma\colon h\Rightarrow k$ in $\cC$. We wish to construct a $\pi$-lift $(\widehat{\sigma},\widehat{\sigma'})\colon (\widehat{h},\widehat{h'}, \tau)\to (\widehat{k},\widehat{k'}, \nu)$, where we can assume that $P\widehat{h}=h$ by \cref{prop:buckley3.2.1}, since $P$ is a 2-fibration. For this, we first note that this data gives rise to a lifting problem in $\cP$ along the $P$-cartesian morphism $f$, and hence admits a unique solution $\widehat{\sigma}\colon\widehat{h}\Rightarrow\widehat{k}$ satisfying the required conditions. As explained in \cref{Path2cells}, this determines a unique $\widehat{\sigma'}$ giving rise to a morphism $(\widehat{\sigma},\widehat{\sigma'})$ in $\Path(P)$; it is then tedious but straightforward to verify that $\widehat{\sigma'}$ also satisfies the required composites for $(\widehat{\sigma},\widehat{\sigma'})$ to form a lift as desired.
\end{proof}

\begin{prop}\label{pathnaivefib}
     If the marked 2-functor $P\colon \markedscat{\cP} \to \cC^{\sharp}$ is a cart-marked 2-fibration, then so is $\pi\colon (\Path(P),\widehat{E^1},\widehat{E^2})\to\cC^{\sharp}$.
\end{prop}

\begin{proof}
    The claim regarding the markings is implied by \cref{lemma:markings}, since $\markedscat{\cP}$ is cart-marked. To show that $\pi$ is a 2-fibration, first note that the horizontal composite of $\pi$-cartesian 2-cells is cartesian, due to the characterization from \cref{lemma:markings} and the fact that this holds for $P$-cartesian 2-cells as $P$ is a 2-fibration.  
    
    Next, we show that $\pi$ has cartesian lifts of morphisms. For this, consider an object $\varphi \colon x\to x'$ in $\Path(P)$ and a morphism $g\colon b\to \pi \varphi=Px=Px'$ in $\cC$.	Since $P$ is a 2-fibration, we can find cartesian $P$-lifts $\widehat{g_1}\colon\widehat{b_1}\to x$ and $\widehat{g_2}\colon\widehat{b_2}\to x'$ with $P\widehat{g_1}=P\widehat{g_2}=g$. Then $\varphi\widehat{g_1}$ and $\widehat{g_2}$ are both cartesian $P$-lifts of $g$ (since $P\varphi=\id_{Px}$), thus by \cref{prop:propertiescartesian} there is an equivalence $\psi\colon \widehat{b_1}\to \widehat{b_2}$ and an invertible 2-cell $\gamma \colon\widehat{g_2} \psi\Rightarrow \varphi\widehat{g_1}$ in $\cP$. We claim that $(\widehat{g_1},\widehat{g_2},\gamma)\colon \psi \to \varphi$ is the cartesian $\pi$-lift of $g$ that we seek. Indeed, by construction we know that this is cartesian, and mapped to $g$ by $\pi$. Moreover, $\psi$ and $(\widehat{g_1},\widehat{g_2},\gamma)$ satisfy the requirements of objects and morphisms in $\Path(P)$, since $P\psi,P\psi^{-1},P\eta,P\epsilon,$ and $P\gamma$ are all identities by \cref{prop:propertiescartesian}. 

    Lastly, to prove that $\pi$ is locally a fibration, consider a morphism $(f,f',\gamma)\colon (\varphi\colon x'\to x)\to (\psi\colon y'\to y)$ in $\Path(P)$ and a 2-cell $\sigma\colon k\Rightarrow Pf=Pf'$ in $\cC$. Since $P$ is locally a fibration, we can find cartesian $P$-lifts $\widehat\sigma\colon\widehat k\Rightarrow f$ and $\widehat\sigma'\colon\widehat{k'}\Rightarrow f'$ with $P\widehat{\sigma}=P\widehat{\sigma'}=\sigma$. Note that both $\psi\cdot\widehat{\sigma}\colon \psi\widehat{k}\Rightarrow \psi f$ and $\gamma(\widehat{\sigma'}\cdot \varphi)\colon \widehat{k'}\varphi\Rightarrow \psi f$ are  cartesian $P$-lifts of $\sigma$, since $P\varphi, P\psi$ and $P\gamma$ are identities. Hence, by the 1-categorical analogue of \cref{prop:propertiescartesian}, there is a unique invertible 2-cell $\tau\colon\widehat{k'}\varphi \Rightarrow \psi\widehat k$ such that $P\tau=\id$ and $(\psi\cdot\widehat{\sigma})\tau = \gamma(\widehat{\sigma'}\cdot \varphi)$.
	We conclude that $(\widehat{\sigma},\widehat{\sigma'})\colon(\widehat{k},\widehat{k'},\tau)\Rightarrow (f,f',\gamma)$ is a cartesian $\pi$-lift of $\sigma$ as desired.
\end{proof}

    We can now use the above result to prove that $\Path(P)$ really is a path object for $P$.

	\begin{prop}
	\label{prop:path-object}
		Let $P\colon \markedscat{\cP} \to \cC^{\sharp}$ be a cart-marked 2-fibration. Then, the factorization of the diagonal
        	\[
	\markedscat{\cP}\xrightarrow{W} (\Path(P),\widehat{E^1},\widehat{E^2})\xrightarrow{R} (\cP\times_{\cC}\cP,E^1_{\cP\times_{\cC}\cP}, E^2_{\cP\times_{\cC}\cP})
	\] of \cref{construction:path_object} gives a path object for $P$  over $\cC^{\sharp}$.
	\end{prop}
	\begin{proof}
	Since we have that $\Path(P)$ is naive fibrant by \cref{pathnaivefib}, in order to prove that $W$ is a weak equivalence it suffices to show its underlying 2-functor is a biequivalence. Similarly, to show that $R$ is a fibration, by \cref{prop:fibbtwfibobj} it suffices to show $UR$ is an equifibration.

    We first focus on $W$. To show it is essentially surjective, let $\varphi\colon x\to x'$ be an object in $\Path(P)$; we wish to find an object $y\in\cP$ and an equivalence $Wb=\id_y\to \varphi$ in $\Path(P)$. This can be achieved by letting $y=x'$ and considering the morphism $(\varphi^{-1},\id_{x'},\epsilon_\varphi^{-1})\colon\id_y\to \varphi$. This is a morphism in $\Path(P)$ since $P\varphi,P\varphi^{-1}$ and $P\epsilon_\varphi^{-1}$ are all identities, and moreover, it is an equivalence.

    Next, we show that $W$ is locally an equivalence, i.e.\ $W_{x,x'}\colon\cP(x,x')\to\Path(P)(\id_x,\id_{x'})$ is an equivalence for any $x,x'\in\cP$. Given a morphism $(f,f',\gamma)\colon\id_{x}\to\id_{x'}$ in $\Path(P)$, we have an invertible 2-cell $(f,f',\gamma)\to (f,f,\id_f)=Wf$ in $\Path(P)$ given by $(\id_f,\gamma)$; this shows that $W_{x,x'}$ is essentially surjective on objects. Lastly, note that given two morphisms $Wf=(f,f,\id_f)$ and $Wg=(g,g,\id_g)$ from $Wx=\id_x$ to $Wy=\id_y$, any 2-cell $(\sigma,\sigma')\colon Wf\to Wg$ in $\Path(P)$ must satisfy $(\id_y\cdot \sigma)\id = \id(\sigma'\cdot \id_x)$; hence $\sigma=\sigma'$ and $(\sigma,\sigma')=W\sigma$, from which we conclude that $W_{x,x'}$ is fully faithful.

    We now focus on $R$, and show its underlying 2-functor is an equifibration. Given an object $\varphi\colon x\to x'$ in $\Path(P)$ and an equivalence $(f,f')\colon(y,y')\to(x,x')$ in $\cP\times_{\cC}\cP$, we need to find an object $\widehat{\varphi}\colon y\to y'$ and an equivalence $(f,f',\alpha)\colon\widehat{\varphi}\to \varphi$ in $\Path(P)$. Since $f,f'$ are equivalences, they are $P$-cartesian and we can find solutions for the following lifting problems
	\[
	\begin{tikzcd}[ampersand replacement=\&,sep=scriptsize]
		y \&\&\&\& Py \\
		\&\& {} \&\& {} \\
		{y'} \&\& {x'} \&\& Py \&\& Px
		\arrow["{\widehat{\varphi}}"', dashed, from=1-1, to=3-1]
		\arrow[""{name=0, anchor=center, inner sep=0}, "{\varphi f}", from=1-1, to=3-3]
		\arrow["\id"', from=1-5, to=3-5]
		\arrow[""{name=1, anchor=center, inner sep=0}, "{Pf}", from=1-5, to=3-7]
		\arrow["P"{pos=0.3}, between={0}{0.6}, maps to, from=2-3, to=2-5]
		\arrow["{f'}"', from=3-1, to=3-3]
		\arrow["{Pf}"', from=3-5, to=3-7]
		\arrow["\gamma","\cong"', between={0.3}{0.7}, Rightarrow, dashed, from=3-1, to=0]
		\arrow["\id","\cong"', between={0.3}{0.7}, Rightarrow, from=3-5, to=1]
	\end{tikzcd} \qquad \quad
	\begin{tikzcd}[ampersand replacement=\&,sep=scriptsize]
		{y'} \&\&\&\& Py \\
		\&\& {} \&\& {} \\
		y \&\& {x'} \&\& Py \&\& Px
		\arrow["{\widehat{\varphi}^{-1}}"', dashed, from=1-1, to=3-1]
		\arrow[""{name=0, anchor=center, inner sep=0}, "{\varphi^{-1}f'}", from=1-1, to=3-3]
		\arrow["\id"', from=1-5, to=3-5]
		\arrow[""{name=1, anchor=center, inner sep=0}, "{Pf}", from=1-5, to=3-7]
		\arrow["P"{pos=0.3}, between={0}{0.6}, maps to, from=2-3, to=2-5]
		\arrow["f"', from=3-1, to=3-3]
		\arrow["{Pf}"', from=3-5, to=3-7]
		\arrow["\sigma","\cong"', between={0.3}{0.7}, Rightarrow, dashed, from=3-1, to=0]
		\arrow["\id","\cong"', between={0.3}{0.7}, Rightarrow, from=3-5, to=1]
	\end{tikzcd}
	\]
    where we used the fact that $Px=Px', Py=Py'$,  $Pf=Pf'$, and that $P\varphi=\id, P\varphi^{-1}=\id$.
	Since $P$ is locally cartesian, by \cref{prop:buckley3.2.1} we can assume that $P\widehat{\varphi},P\widehat{\varphi}^{-1}, P\gamma$ and $P\sigma$ are all identities, and so all that remains to be proved is that $\widehat{\varphi}$ is an object of $\Path(P)$. We can construct a unit $\eta\colon \id\Rightarrow\widehat{\varphi}^{-1}\widehat{\varphi}$ by considering the invertible 2-cell

	\[\begin{tikzcd}[ampersand replacement=\&,cramped,column sep=22pt]
		\id \& {f^{-1}f} \&\& {f^{-1}\varphi^{-1}\varphi f} \&\& {f^{-1}\varphi^{-1}f'\widehat{\varphi}} \&\& {f^{-1}f\widehat{\varphi}^{-1}\widehat{\varphi}} \&\& {\widehat{\varphi}^{-1}\widehat{\varphi}}
		\arrow["{\eta_{f}}", Rightarrow, from=1-1, to=1-2]
		\arrow["{f^{-1}\cdot \eta_{\varphi} \cdot f}", Rightarrow, from=1-2, to=1-4]
		\arrow["{f^{-1}\varphi^{-1}\cdot \gamma^{-1}}", Rightarrow, from=1-4, to=1-6]
		\arrow["{f^{-1}\cdot \sigma^{-1}\cdot\widehat{\varphi}}", Rightarrow, from=1-6, to=1-8]
		\arrow["{\eta_{f}^{-1}\cdot \widehat{\varphi}^{-1}\widehat{\varphi}}", Rightarrow, from=1-8, to=1-10]
	\end{tikzcd}\]
which is mapped to the identity by $P$. The construction of the counit is analogous.

Finally, to prove the second condition of \cref{defn:lackfib}, consider a morphism $(f,f',\gamma)$ from $\varphi\colon a\to a'$ to $\psi\colon y\to y'$ in $\Path(P)$, and an invertible 2-cell $(\delta, \delta')\colon (g,g')\Rightarrow (f,f')$ in $\cP\times_\cC \cP$. We wish to find a morphism $(g,g',\sigma)$ in $\Path(P)$ (that is, some invertible 2-cell $\sigma\colon g' \varphi\Rightarrow \psi g$ with $P\sigma=\id$) in such a way that $(\delta,\delta')\colon (g,g',\sigma)\to (f,f',\gamma)$ is a 2-cell in $\Path(P)$ (that is, such that $(\psi\cdot \delta)\sigma=\gamma(\delta'\cdot \varphi)$). This can be achieved by defining $\sigma$ as the composite
\[g' \varphi\xRightarrow{\delta' \cdot \varphi} f' \varphi\xRightarrow{\gamma} \psi f\xRightarrow{\psi\cdot \delta^{-1}} \psi g\]
and using the fact that $P\varphi,P\psi$ and $P\gamma$ are identities, and that $P\delta'=P\delta$.
\end{proof}

    \subsection{The model structure on \texorpdfstring{$\mslicetwocat$}{the marked slice category}}\label{sec:theMS}

    Having established all of the necessary ingredients, we can finally construct the desired model structure, whose fibrant objects capture the cart-marked 2-fibrations.

	\begin{thm}\label{Thm:the ms}
	There is a combinatorial model structure on the category $\twocat^{+}_{/ \Csharp}$ of marked $2$-categories over $\cC^{\sharp}$, whose
	\begin{itemize}
    \item 
		fibrant objects are the cart-marked $2$-fibrations of \cref{Def:Cart-marked-2-fibrations},
		\item 
		cofibrations are the marked cofibrations  of \cref{Def:cofibrations}, 
		\item 
		weak equivalences between fibrant objects are the marked 2-functors whose underlying 2-functor is a biequivalence, and
		\item fibrations between fibrant objects are the marked 2-functors whose underlying 2-functor is an equifibration.
	\end{itemize}
	\end{thm}
	\begin{proof}
    	We apply Theorem 2.8 from \cite{guetta2023fibrantlyinduced}.
    	First, we have that our desired class of cofibrations is generated by a set, as we showed in \cref{prop:genmarkedcofibs}, and moreover, that it contains the class of anodyne extensions as outlined in \cref{AnisinCof}. Furthermore, we have that all trivial fibrations are weak equivalences by \cref{prop:trivial_fibration_is_w}; that the weak equivalences between naive fibrant objects (i.e.\ the biequivalences) have 2-out-of-6 as they are the weak equivalences in the canonical model structure in $\twocat$; they are accessible for the same reason; we have established the existence of path objects for fibrant objects in \cref{prop:path-object}; and lastly, the fact that naive fibrations that are weak equivalences between fibrant objects are trivial fibrations follows from \cref{prop:fibbtwfibobj,cor:trivfibbtwfibobj}.
	\end{proof}

\section{The marked Grothendieck construction}\label{sec:mGC}

We now consider the Grothendieck construction of a 2-functor $F\colon \cC^{\coop}\to \ttwocat$ which gives a functor 
        $ {\textstyle \int}_{\cC}\colon 2\Fun[\cC^{\coop},\ttwocat]\to \twocat_{/\cC}$.
        Since we are using marked $2$-categories, we present an adapted construction in \cref{subsec:marked_GC}
        \[{\textstyle \int}_{\cC}^+\colon 2\Fun[\cC^{\coop},\ttwocat]\to \mslicetwocat \]
        which endows each $2$-category $\int_{\cC}F$ with the marking described in \cref{Const:marked-scaled-GC}. 
        
        One might expect that the marked Grothendieck construction  would be a right Quillen equivalence, but unfortunately this cannot be the case: the functor $\int_\cC^+$ is not even a right adjoint, as we show in \cref{rmk:example_marking_equivs_not_radjoint}. Regardless of this, in \cref{section5} we prove that this functor induces an equivalence on homotopy categories. Towards that goal, \cref{subsec:construction_weak_left_adjoint} studies the compatibility between the functor $\int_\cC^+$ and the relevant model structures.

    \subsection{The marked Grothendieck construction}
    \label{subsec:marked_GC}
         We start by upgrading the 2-categorical Grothendieck construction (see \cite{Buckley2014}) to the marked setting.

        \begin{constr}\label{Const:marked-scaled-GC}
            Let $F\colon\cC^{\coop}\to \ttwocat$ be a 2-functor. The \emph{marked Grothendieck construction} of $F$ is the  marked $2$-category $\int_{\cC}^{+}F$ whose 
            \begin{itemize}[leftmargin=*]
                    \item \label{Const:marked-scaled-GC.1}
                    objects are pairs $(c,x)$ consisting of objects $c \in \cC$ and $x\in Fc$,

                    \item \label{Const:marked-scaled-GC.2} 
                    morphisms are pairs $(f,\widehat{f})\colon (c,x)\to (c',x')$ consisting of morphisms $f\colon c\to c'$ in $\cC$ and  $\widehat{f}\colon x \to Ff(x')$ in $Fc$,

                    \item \label{Const:marked-scaled-GC.3}
                    $2$-cells are pairs $(\alpha,\widehat{\alpha})\colon  (f,\widehat{f})\Rightarrow (g, \widehat{g})$ between morphisms as above, consisting of 2-cells $\alpha\colon f \Rightarrow g$ in $\cC$ and $\widehat{\alpha}\colon  \widehat{f} \Rightarrow (F\alpha)_{x'}.\widehat{g}$  in $Fc$,
                    
            \item \label{Const:marked-scaled-GC.4}
            identity morphisms are $\id_{(c,x)}= (\id_c, \id_{x})$, and
        identity $2$-cells are $(\id_f,\id_{\widehat{f}})\colon (f,\widehat{f}) \Rightarrow (f,\widehat{f})$,

            \item \label{Const:marked-scaled-GC.6}
            composite of morphisms 
            $(c,x)\xrightarrow{(f,\widehat f)} (c',x')\xrightarrow{(g,\widehat{g})} (c'',x'')$ has its first component given by $gf$ and its second component given by the composite $Ff(\widehat{g}).\widehat{f}$ in $Fc$:
    
            \[
                \begin{tikzcd}[column sep=1cm]
            x & {Ff(x')} & {FfFg(x'')=F(gf)(x''),}
            \arrow["{\widehat{f}}", from=1-1, to=1-2]
            \arrow["{Ff(\widehat{g})}", from=1-2, to=1-3]
                \end{tikzcd}
            \]
            \item \label{Const:marked-scaled-GC.7}
             horizontal composite of $2$-cells 
            
        \[
            \begin{tikzcd}
            {(c,x)} && {(c',x')} && {(c'',x'')}
            \arrow[""{name=0, anchor=center, inner sep=0}, "{(f,\widehat{f})}", curve={height=-18pt}, from=1-1, to=1-3]
            \arrow[""{name=1, anchor=center, inner sep=0}, "{(g,\widehat{g})}"', curve={height=18pt}, from=1-1, to=1-3]
            \arrow[""{name=2, anchor=center, inner sep=0}, "{(h,\widehat{h})}", curve={height=-18pt}, from=1-3, to=1-5]
            \arrow[""{name=3, anchor=center, inner sep=0}, "{(k,\widehat{k})}"', curve={height=18pt}, from=1-3, to=1-5]
            \arrow["{(\alpha,\widehat{\alpha})}", shift right=2, shorten <=7pt, shorten >=7pt, Rightarrow, from=0, to=1]
            \arrow["{(\beta,\widehat{\beta})}", shift right=2, shorten <=7pt, shorten >=7pt, Rightarrow, from=2, to=3]
            \end{tikzcd}
        \]
            has its first component given by the horizontal composite $\beta  \alpha$ in $\cC$ and its second component given by the pasting diagram:
        
        \[\begin{tikzcd}[ampersand replacement=\&,
        column sep=16pt,
    row sep=3ex]
        	x \&\& {Ff(x')} \&\& {Ff.Fh(x'')}  \\
        	\\
        	\&\& {Fg(x')} \&\& {Ff.Fk(x'')} \\
        	\\
        	\&\&\&\& {Fg.Fk(x'')} 
        	\arrow[""{name=0, anchor=center, inner sep=0}, "{\widehat{f}}", from=1-1, to=1-3]
        	\arrow[""{name=0p, anchor=center, inner sep=0}, phantom, from=1-1, to=1-3, start anchor=center, end anchor=center]
        	\arrow[""{name=1, anchor=center, inner sep=0}, "{\widehat{g}}"', from=1-1, to=3-3]
        	\arrow[""{name=1p, anchor=center, inner sep=0}, phantom, from=1-1, to=3-3, start anchor=center, end anchor=center]
        	\arrow[""{name=2, anchor=center, inner sep=0}, "{Ff(\widehat{h})}", from=1-3, to=1-5]
        	\arrow[""{name=2p, anchor=center, inner sep=0}, phantom, from=1-3, to=1-5, start anchor=center, end anchor=center]
        	\arrow[""{name=3, anchor=center, inner sep=0}, "{Ff(\widehat{k})}"', from=1-3, to=3-5]
        	\arrow[""{name=3p, anchor=center, inner sep=0}, phantom, from=1-3, to=3-5, start anchor=center, end anchor=center]
        	\arrow[""{name=3p, anchor=center, inner sep=0}, phantom, from=1-3, to=3-5, start anchor=center, end anchor=center]
        	\arrow["{(F\alpha)_{x'}}"{description}, from=3-3, to=1-3]
        	\arrow[""{name=4, anchor=center, inner sep=0}, "{Fg(\widehat{k})}"', from=3-3, to=5-5]
        	\arrow[""{name=4p, anchor=center, inner sep=0}, phantom, from=3-3, to=5-5, start anchor=center, end anchor=center]
        	\arrow["{Ff.(F\beta)_{x''}}"{description}, from=3-5, to=1-5]
        	\arrow["{(F\alpha)_{Fk(x'')}}"{description}, from=5-5, to=3-5]
        	\arrow["{\widehat{\alpha}}", between={0.3}{0.7}, Rightarrow, from=0p, to=1p]
        	\arrow["{Ff(\widehat{\beta})}", shift right=3, between={0.3}{0.7}, Rightarrow, from=2p, to=3p]
        	\arrow["{=}", shift right=3, draw=none, from=3p, to=4p]
        \end{tikzcd}\]

            \item \label{Const:marked-scaled-GC.8}
            vertical composite of $2$-cells
            $(f,\widehat{f}) \xRightarrow{(\alpha,\widehat{\alpha})} (g,\widehat{g})\xRightarrow{(\beta,\widehat{\beta})} (h,\widehat{h})$ has first component given by the vertical composite $\beta.\alpha$ and second component given by the pasting diagram: 
            
        \[
            \begin{tikzcd}[column sep=16pt, row sep=3ex]
                x &&&&&& {Ff(x')} \\
                \\
                &&& {Fg(x')} \\
                \\
                &&& {Fh(x')}
                \arrow[""{name=0, anchor=center, inner sep=0}, "{\widehat{f}}", from=1-1, to=1-7]
                \arrow[""{name=1, anchor=center, inner sep=0}, "{\widehat{g}}", from=1-1, to=3-4]
                \arrow[""{name=1p, anchor=center, inner sep=0}, phantom, from=1-1, to=3-4, start anchor=center, end anchor=center]
                \arrow[""{name=2, anchor=center, inner sep=0}, "{\widehat{h}}"', from=1-1, to=5-4]
                \arrow[""{name=2p, anchor=center, inner sep=0}, phantom, from=1-1, to=5-4, start anchor=center, end anchor=center]
                \arrow[""{name=3, anchor=center, inner sep=0}, "{(F\alpha)_{x'}}", from=3-4, to=1-7]
                \arrow[""{name=3p, anchor=center, inner sep=0}, phantom, from=3-4, to=1-7, start anchor=center, end anchor=center]
                \arrow[""{name=4, anchor=center, inner sep=0}, "{(F(\beta\alpha))_{x'}}"', from=5-4, to=1-7]
                \arrow[""{name=4p, anchor=center, inner sep=0}, phantom, from=5-4, to=1-7, start anchor=center, end anchor=center]
                \arrow["{(F\beta)_{x'}}" description, from=5-4, to=3-4]
                \arrow["{\widehat{\alpha}}"'{pos=0.4}, shorten <=11pt, shorten >=15pt, Rightarrow, from=0, to=3-4]
                \arrow[""{name=5p, anchor=center, inner sep=0}, phantom, from=3-4, to=5-4, start anchor=center, end anchor=center]
                \arrow["{\Downarrow \scriptstyle{\widehat{\beta}}}"{pos=0.6}, shorten <=6pt, shorten >=6pt, phantom, from=1-1, to=5p]
                \arrow["{=}"{description}, draw=none, from=5p, to=1-7]
            \end{tikzcd}.
            \]
            \item \label{Const:marked-scaled-GC.9}
            marked morphisms $(f,\widehat{f})$ are the ones where $\widehat{f}$ is an equivalence.
            \item \label{Const:marked-scaled-GC.10}
            marked 2-cells $(\alpha,\widehat{\alpha})$ are the ones where $\widehat{\alpha}$ is invertible.
        \end{itemize}
        \end{constr}

        \begin{rmk}
        The underlying 2-category of \cref{Const:marked-scaled-GC} is precisely Buckley's Grothendieck construction (\cite[Construction 2.2.1]{Buckley2014}); hence the fact that $\int^+_{\cC}F$ is a $2$-category is justified by \cite[Proposition 2.2.2]{Buckley2014}. Furthermore it is not hard to see that this construction also corresponds to \cite[Construction 3.3.3]{Buckley2014} in the case where the functor $F$ is a 2-functor valued in $\twocat$ instead of a trihomomorphism valued in $\bicat$.
        \end{rmk}

        \begin{rmk}\label{Def:marked_GC}
        The marked Grothendieck construction gives a functor
        \[{\textstyle \int}_{\cC}^+\colon 2\Fun[\cC^{\coop},\ttwocat]\to \mslicetwocat. \]
        Indeed, the marked $2$-category $\int_{\cC}^+ F$ defined above comes with a  canonical marked projection $\pi_F\colon \int_{\cC}^+ F \to \cC^{\sharp}$ onto the first coordinate. Moreover, this construction is functorial in the sense that if $F,G\colon \cC^{\coop} \to \ttwocat$ are 2-functors, then a $2$-natural transformation $\theta\colon F \Rightarrow G$ induces a marked 2-functor $\int_{\cC}^+ \theta\colon \int_{\cC}^+ F \to \int_{\cC}^+ G$ over $\cC^{\sharp}$ which sends objects, morphisms and $2$-cells as in the following diagram:
    
        \[\begin{tikzcd}[cramped]
        	{(c,x)} && {(c',x')} &&& {(c,\theta_c x)} && {(c',\theta_{c'}x')}
        	\arrow[""{name=0, anchor=center, inner sep=0}, "{(f,\widehat{f})}", curve={height=-18pt}, from=1-1, to=1-3]
        	\arrow[""{name=1, anchor=center, inner sep=0}, "{(g,\widehat{g})}"', curve={height=18pt}, from=1-1, to=1-3]
        	\arrow["{\int_{\cC}^+ \theta}", between={0.2}{0.8}, maps to, from=1-3, to=1-6]
        	\arrow[""{name=2, anchor=center, inner sep=0}, "{(f,\theta_c\widehat{f})}", curve={height=-18pt}, from=1-6, to=1-8]
        	\arrow[""{name=3, anchor=center, inner sep=0}, "{(g, \theta_c \widehat{g})}"', curve={height=18pt}, from=1-6, to=1-8]
        	\arrow["{(\alpha,\widehat{\alpha})}", shift right=3, between={0.3}{0.7}, Rightarrow, from=0, to=1]
        	\arrow["{(\alpha,\theta_c\widehat{\alpha})}", shift right=3, between={0.3}{0.7}, Rightarrow, from=2, to=3]
        \end{tikzcd}\]
        \end{rmk}

        Our choice of markings in the marked Grothendieck construction is so that the projection to $\Csharp$ is cart-marked. This is justified by the next proposition, which is claimed without proof in \cite[Proposition 4.3]{vazquez2021three}; we include a proof here to keep the paper self-contained.

        \begin{prop}\label{prop:markingsinGC}
        Every 2-functor $F\colon \cC^{\coop} \to \ttwocat$ satisfies the following conditions:
        \begin{itemize}
        \item A morphism $(f,\widehat{f})$ in $\int_{\cC}^{+} F$ is cartesian if and only if $\widehat{f}$ is an equivalence.
        \item A 2-cell $(\alpha,\widehat{\alpha})$ in $\int_{\cC}^{+} F$ is cartesian if and only if $\widehat{\alpha}$ is an isomorphism.
        \end{itemize}
        \end{prop}
        \begin{proof}
        We begin by showing that if $(f,\widehat{f})\colon (c,x)\to (c',x')$ is cartesian then $\widehat{f}\colon x \to Ff(x')$ is an equivalence. The morphism $(f,\id_{Ff(x')})\colon (c,Ff(x')) \to (c',x')$ is also cartesian as shown in the proof of \cite[Proposition 3.3.4]{Buckley2014}. That is, $(f,\widehat{f})$ and $(f,\id_{Ff(x')})$ are both cartesian $\pi_F$-lifts of the same morphism. Hence, by \enumref{prop:propertiescartesian}{3} there exists an equivalence $(h,\widehat{h})\colon (c,Ff(x')) \to (c,x)$ together with an invertible $2$-cell $(\alpha,\widehat{\alpha})\colon (f,\widehat{f})(h,\widehat{h}) \Rightarrow (f,\id_{Ff(x')})$. The equivalence data 
		\begin{equation*}
            \left\lbrace\,
            \begin{gathered}
            (h^{-1},\widehat{h}^{-1})\colon (c,x) \to (c,Ff(x')) \\ 
			(\epsilon,\widehat{\epsilon})\colon (h^{-1},\widehat{h}^{-1}).(h,\widehat{h}) \Rightarrow \id_{(c,Ff(x'))} \\
			(\eta,\widehat{\eta})\colon \id_{(c,x)} \Rightarrow (h,\widehat{h}).(h^{-1},\widehat{h}^{-1})
            \end{gathered}
            \right.
        \end{equation*}
		and $(\alpha,\widehat{\alpha})$ can be chosen so that $h=\id_c$, $h^{-1}=\id_c$, $\epsilon=\id_{\id_c}$, $\eta=\id_{\id_c}$, and $\alpha = \id_f$. With these choices, $\widehat{\epsilon}$, $\widehat{\eta}$, and $\widehat{\alpha}$ are invertible $2$-cells. Moreover, $Fh(x)= x$ and $(F\alpha)_{x'} = \id_{Ff(x')}$, so $\widehat{h}\colon Ff(x')\to x$ and we have an invertible $2$-cell $\widehat{\alpha}\colon \widehat{f}.\widehat{h}\Rightarrow \id_{Ff(x')}$ together with the composite of invertible $2$-cells
       
        \[\begin{tikzcd}[cramped,column sep=18pt]
        	{\widehat{h}.\widehat{f}} && {\widehat{h}.\widehat{f}.\widehat{h}.\widehat{h}^{-1}} && {\widehat{h}\widehat{h}^{-1}} && {\id_x}
        	\arrow["{(\widehat{h}.\widehat{f})\cdot \widehat{\eta}}", Rightarrow, from=1-1, to=1-3]
        	\arrow["{\widehat{h}\cdot \widehat{\alpha}\cdot \widehat{h}^{-1}}", Rightarrow, from=1-3, to=1-5]
        	\arrow["{\widehat{\eta}^{-1}}", Rightarrow, from=1-5, to=1-7]
        \end{tikzcd}\]
        We conclude that $\widehat{h}$ and $\widehat{f}$ are inverse equivalences. 
	
        Conversely, let $(f,\widehat{f})\colon (c,x)\to(c',x')$ be a morphism such that $\widehat{f}$ is an equivalence. Note that we can always write $(f,\widehat{f})$ as the composite $(f,\id_{Ff(x')}).(\id_c, \widehat{f})$, where $(f,\id_{Ff(x')})$ is cartesian as established in the proof of \cite[Proposition 3.3.4]{Buckley2014}, and $(\id_c, \widehat{f})$ is an equivalence and hence cartesian. We conclude that $(f,\widehat{f})$ must be cartesian as well.
        
        Lastly, for the claim regarding 2-cells, note that for each pair of objects $(c,x),(c',x')$, the projection $\pi_F\colon\int_{\cC}^{+}F\to\cC^\sharp$ is locally given by the 1-categorical  Grothendieck construction of the functor $\cC(c,c')^{\mathrm{op}}\to\cat$ that sends $f\colon c\to c'$ to the comma category $x\downarrow UFf$. Here $UFf\colon UFc'\to UFc$ denotes the underlying functor between underlying 1-categories, and $x\colon [0]\to UFc$ denotes the constant functor at $x$.
            The statement regarding 2-cells then follows from the 1-dimensional analogue, observed in \cite[Remark 5.2]{mosersarazola2023model}.
        \end{proof}

        As in the classical case, the projection $\int^{+}_{\cC} F\to\cC^{\sharp}$ is expected to be a ``prototypical fibration''. The following result shows that it has the expected behaviour. 
        
        \begin{cor}\label{cor:projectionisfibrant}
            For every 2-functor $F\colon\cC^{\coop}\to\ttwocat$, the projection $\pi_F\colon \int^{+}_{\cC} F\to\cC^{\sharp}$ is a cart-marked 2-fibration.
        \end{cor}
        \begin{proof}
            The map $\pi_F$ is a 2-fibration by \cite[Proposition 3.3.4]{Buckley2014}. Moreover, it is cart-marked  due to the definition of marked morphisms and 2-cells in $\int^{+}_{\cC} F$ from \cref{Const:marked-scaled-GC}, together with the characterization of cartesian morphisms and 2-cells from \cref{prop:markingsinGC}. 
        \end{proof}

    Given this result, we will often make an abuse of notation and employ $\int^{+}_{\cC} F$ when referring to the cart-marked 2-fibration $\pi_F$. This is simply because the former conveys more information and context for some results.

    \subsection{Compatibility with the model structure}
    \label{subsec:construction_weak_left_adjoint} 
    We would like to be able to say that the marked Grothendieck construction 
    \[{\textstyle \int}_{\cC}^+\colon 2\Fun[\cC^{\coop},\ttwocat] \to \mslicetwocat \] 
    is a right Quillen functor, but unfortunately, that cannot be: as the following example shows, this functor does not preserve limits, and hence cannot be a right adjoint.   

    \begin{ex}  \label{rmk:example_marking_equivs_not_radjoint}
    When $\cC=[0]$ is the terminal $2$-category, the marked Grothendieck construction can more simply be described as the functor $\int_{[0]}^+ \colon \twocat \to  \twocat^+ $ which sends a $2$-category $\cD$ to the marked $2$-category $\cD^{\natural}$ where we have marked all equivalence morphisms and all invertible $2$-cells. Note that it is not the same as the trivial marking functor $(-)^\flat$ which marks only the identity arrows. 

    We will show that $\int_{[0]}^+$ does not preserve limits. For this, let $E$ be the walking equivalence (see \cref{defn:walkingequiv}) and consider the 2-functor $R\colon E \to E$ which is the identity on objects, and with $R(i)=i$, $R(i^{-1})=i^{-1}  i  i^{-1}$. The equalizer in $\twocat$ of $R$ and the identity is then:
            \[\begin{tikzcd}[ampersand replacement=\&]
                {(0\xrightarrow{i} 1)} \& E \& E
                \arrow[from=1-1, to=1-2]
                \arrow["R", shift left=2, from=1-2, to=1-3]
                \arrow["{\id}"', shift right=2, from=1-2, to=1-3]
            \end{tikzcd}\]
            Note that $i$ is an equivalence in $E$ but not in the equalizer. Then, applying the functor $\int_{[0]}^+$ to this equalizer, we obtain the  diagram in $\twocat^+$ below left
            
             \[\begin{tikzcd}
                 {(0\xrightarrow{i} 1)^\flat} & {E^{\natural}} & {E^{\natural}}
                 \arrow[from=1-1, to=1-2]
                 \arrow["{R}", shift left=2, from=1-2, to=1-3]
                 \arrow["{\id}"', shift right=2, from=1-2, to=1-3]
             \end{tikzcd}\qquad \qquad \qquad
            \begin{tikzcd}
                {(0\xrightarrow{i} 1)^\sharp} \\
                {(0\xrightarrow{i} 1)^\flat} & {E^{\natural}} & {E^{\natural}}
                \arrow["{\textrm{not marked}}"', dashed, from=1-1, to=2-1]
                \arrow[from=1-1, to=2-2]
                \arrow[from=2-1, to=2-2]
                \arrow["{R}", shift left=2, from=2-2, to=2-3]
                \arrow["{\id}"', shift right=2, from=2-2, to=2-3]
            \end{tikzcd}\]
             which does not satisfy the universal property of the equalizer in $\twocat^+$, as exhibited by the diagram drawn above to the right. Indeed, the only 2-functor $(0\xrightarrow{i} 1)^\sharp\to (0\xrightarrow{i} 1)^\flat$ that makes the required triangle commute is the identity, but this does not preserve markings, as $i$ is marked in the source but not in the target.
        \end{ex}

In spite of this failure, we will eventually show that the marked Grothendieck construction gives an equivalence on homotopy categories. Towards this end, we now collect several results that show how this functor still exhibits the expected behaviour with respect to the model categorical data.

    \begin{prop}\label{lemma:Grothendieck_constr_preserves_fibrations}
        The functor $\int^+_{\cC}$ preserves fibrations and trivial fibrations. 
    \end{prop}

    \begin{proof}
        Let $\alpha\colon F\Rightarrow G$ be a (trivial) fibration in $2\Fun[\cC^{\coop},\ttwocat]$; i.e., for every $c\in \cC$, the 2-functor $\alpha_c\colon Fc\to Gc$ is a (trivial) fibration in $\twocat$. Since the marked 2-functor $\int_{\cC}^+ \alpha\colon \int_{\cC}^+ F\to \int_{\cC}^+ G$ acts as $\alpha_c$ on objects, morphisms, and 2-cells, we get that its underlying 2-functor is a (trivial) fibration in $\twocat$. Now, by \cref{cor:projectionisfibrant}, we know that $\int_{\cC}^+ F$ and $\int_{\cC}^+ G$ are  fibrant objects; hence, we deduce from  \cref{prop:fibbtwfibobj,cor:trivfibbtwfibobj} that $\int_{\cC}^+ \alpha$ is a (trivial) fibration in $\twocat^+_{/\cC^\sharp}$.
    \end{proof}

The following result will be helpful in understanding the behaviour of $\int_\cC^+$ with respect to weak equivalences. 

    \begin{lem}\label{lem:equiv_then_coordinatewise_equiv} If the pair $(f,\widehat{f})\colon(c,x)\to (c',x')$ is an equivalence in $\int^+_{\cC} F$, then $f$ is an equivalence in $\cC$ and $\widehat{f}$ is an equivalence in $Fc$.
    \end{lem}
    \begin{proof}
Suppose that $(f,\widehat{f})$ is an equivalence in $\int^+_{\cC} F$; then, as the 2-functor $\pi_F\colon \int^+_{\cC} F\to \cC$ preserves equivalences, we get that $\pi_F(f,\widehat{f})=f$ is an equivalence in $\cC$. Moreover, as all equivalences in a 2-category are cartesian, we see that $(f,\widehat{f})$ is cartesian, hence $\widehat{f}$ is an equivalence by \cref{prop:markingsinGC}.
    \end{proof}

    \begin{thm}\label{GCpreservesandreflectsweakequivs}
    $\int^+_{\cC}$ preserves and reflects weak equivalences.
    \end{thm}
    \begin{proof}  
    \cref{lemma:Grothendieck_constr_preserves_fibrations} shows that $\int_\cC^+$ preserves trivial fibrations. Then, by Ken Brown's lemma \cite[Lemma 1.1.12]{hovey} we get that $\int_\cC^+$ preserves all weak equivalences between fibrant objects, hence all weak equivalences as all objects in $2\Fun[\cC^\coop,\ttwocat]$ are fibrant.

    The proof that it reflects weak equivalences is more involved. We want to show that if a natural transformation $\alpha\colon F\Rightarrow G$ in $2\Fun[\cC^{\mathrm{coop}},\ttwocat]$ is such that $\int_{\cC}^+\alpha$ is a weak equivalence, then $\alpha$ is a weak equivalence itself \textemdash that is, $\alpha_c\colon Fc\to Gc$ is a biequivalence for all $c$ in $\cC$. Let $\alpha$ be such a natural transformation. Since both $\int_{\cC}^+ F$ and $\int_{\cC}^+ G$ are cart-marked $2$-fibrations, we know that the underlying 2-functor of $\int_{\cC}^+\alpha\colon \int_{\cC}^+ F\to \int_{\cC}^+ G$ is a biequivalence.

    Consider $c\in\cC$ and the component $\alpha_c\colon Fc\to Gc$. To show that it is essentially surjective on objects, let $y\in Gc$; we want to show that there exists $x\in Fc$ and an equivalence $\alpha_c(x)\to y$. Since $\int_{\cC}^+ \alpha$ is a biequivalence, for $(c,y)\in \int_{\cC}^+ G$ there exists an element $(c',x)\in\int_{\cC}^+ F$ and an equivalence in $\int_{\cC}^+ G$ as below 
    $$(c',\alpha_{c'}(x))\xrightarrow{(f,\widehat{f})}(c,y),$$
	with $f\colon c'\to c$ in $\cC$, and $\widehat{f}\colon\alpha_{c'}(x)\to Gf(y)$ in $Gc'$. In particular, there is an  inverse equivalence
    $$(c,y)\xrightarrow{(g,\widehat{g})}(c',\alpha_{c'}(x))$$
	where $g\colon c\to c'$ is in $\cC$, and $\widehat{g}\colon y \to Gg(\alpha_{c'}(x))$ is in $Gc$. Thus by \cref{lem:equiv_then_coordinatewise_equiv}, we know that both $\widehat{f}$ and $\widehat{g}$ are equivalences. The required equivalence $\alpha_c(x) \to y$ is then defined to be the inverse of $\widehat{g}\colon y\to Gg(\alpha_{c'}(x))=\alpha_c(Fg(x))$. 
    
    It remains to show that $\alpha_c$ induces an equivalence on hom-categories
    $$Fc (x,x')\to Gc(\alpha_c(x),\alpha_c(x'))$$
    for all $x,x'$ in $Fc$ and all $c\in\cC$. Since $\int_{\cC}^+\alpha$ is a weak equivalence, we know that for any two objects $(c,x),(c,x')$ in $\textstyle\int_{\cC}^+F$, we have an equivalence of hom-categories
    $$\textstyle\int_{\cC}^+ F ((c,x),(c,x'))\to \int_{\cC} ^+ G ((c,\alpha_c(x)), (c,\alpha_c(x'))).$$ 
    Explicitly, on objects it sends 
    $$\left(f\colon c\to c,\widehat{f}\colon x\to Ffx'\right)\longmapsto  \left(f\colon c\to c,\alpha_c(\widehat{f})\colon \alpha_c(x)\to \alpha_c(Ffx')\right) $$
    and on morphisms
    $$(\mu,\widehat{\mu})\longmapsto (\mu,\alpha_c(\widehat{\mu}))$$
    where the involved $2$-cells are given as displayed below
	\[\begin{tikzcd}[ampersand replacement=\&,cramped,column sep=18pt]
c \&\& c
	\arrow[""{name=0, anchor=center, inner sep=0}, "f", curve={height=-12pt}, from=1-1, to=1-3]
		\arrow[""{name=1, anchor=center, inner sep=0}, "g"', curve={height=12pt}, from=1-1, to=1-3]
        		\arrow["\mu", between={0.2}{0.8}, Rightarrow, from=0, to=1]
    \end{tikzcd} \qquad
    \begin{tikzcd}[ampersand replacement=\&,cramped,column sep=18pt]
		  x \& {} \& {Ffx'} \& {\alpha_c(x)} \& {} \& {\alpha_c(Ffx')} \\
		\& {Fgx'} \&\&\& {\alpha_c(Fgx')}
		\arrow["{\widehat{f}}", from=1-1, to=1-3]
		\arrow["{\widehat{g}}"', from=1-1, to=2-2]
		\arrow["{\widehat{\mu}}"{pos=0.4}, between={0}{0.8}, Rightarrow, from=1-2, to=2-2]
		\arrow["{\alpha_c(\widehat{f})}", from=1-4, to=1-6]
		\arrow["{\alpha_c(\widehat{g})}"', from=1-4, to=2-5]
		\arrow["{\alpha_c(\widehat{\mu})}"{pos=0.4}, between={0}{0.8}, Rightarrow, from=1-5, to=2-5]
		\arrow["{\alpha_c((F\mu)_{x'})}"', from=2-5, to=1-6]
		\arrow["{(F\mu)_{x'}}"', from=2-2, to=1-3]
	\end{tikzcd}\]

    From the description above, one can observe that when we restrict this equivalence to the subcategory of $\int^+_{\cC} F((c,x),(c,x'))$ whose objects are of the form $(\id, \widehat{f})$ and whose morphisms are of the form $(\id,\widehat{\mu})$ we obtain a functor 
    $$Fc (x,x')\xrightarrow{\alpha_c} Gc(\alpha_c(x),\alpha_c(x')).$$
    Moreover, via the definition of $\int_{\cC}^+\alpha$, it is easy to check that this restriction is also an equivalence.
    \end{proof}

 \section{Equivalence on homotopy categories}\label{section5}
In this final section, we show that the marked Grothendieck construction 
   \[{\textstyle \int}_{\cC}^+ \colon 2\Fun[\cC^{\coop},\ttwocat]\to \mslicetwocat \]
   induces an equivalence between homotopy categories. We already know by \cref{GCpreservesandreflectsweakequivs} that this functor preserves all weak equivalences; hence, it descends to the homotopy categories, or in other words, it admits a total  right derived functor 
   \[{\textstyle \int}_{\cC}^+ \colon \Ho 2\Fun[\cC^{\coop},\ttwocat] \to \Ho\mslicetwocat. \]
Our strategy will be to construct  a functor in the opposite direction
\[L \colon \mslicetwocat\to  2\Fun[\cC^{\coop},\ttwocat] \]  that will give rise to an inverse equivalence when we pass to the homotopy categories.

Since an object $P\colon (\cP,E^1_\cP, E^2_\cP)\to\cC^\sharp$ in the domain carries the data of markings, and we wish to produce a 2-functor $\cC^\coop\to\ttwocat$ where no markings are involved, part of the process of defining $L(P)$ consists of finding a way to remove the markings while still taking that data into account. This will be achieved by localizing the marked 2-category $\cP$, sending marked morphisms to equivalences and marked 2-cells to invertible 2-cells.

\subsection{Localization of 2-categories}\label{subsec.localization}
A treatment of bicategorical localizations at a collection of morphisms was given in \cite{pronk1996etendues}, which requires  the localizing class to satisfy a bicategorical calculus of fractions. Unfortunately, for our goals, such an assumption would be too restrictive. In addition, we need to be able to  localize not only at a collection of morphisms, but also at 2-cells.  We now construct such a localization. 

\begin{constr}\label{constr:pushout_localization}
    Let $\cP$ be a $2$-category, $\cW_1$ a set of morphisms, and $\cW_2$ a set of $2$-cells in $\cP$. We start by considering the localization of $\cP$ at $\cW_1$ as the following pushout in $\twocat$:
	\[\begin{tikzcd}[ampersand replacement=\&,cramped]
		{\displaystyle\coprod_{f \in \cW_1} [1]} \& \cP \\
		{\displaystyle\coprod_{f \in \cW_1} E_{\mathrm{adj}}} \& {\cP[\cW_1^{-1}]}
		\arrow[from=1-1, to=1-2]
		\arrow[hook, from=1-1, to=2-1]
		\arrow["{L_1}", dashed, hook, from=1-2, to=2-2]
		\arrow[dashed, from=2-1, to=2-2]
		\arrow["\lrcorner"{anchor=center, pos=0.125, rotate=180}, draw=none, from=2-2, to=1-1]
	\end{tikzcd}\]
	The left vertical map is induced by the inclusions $[1] \hookrightarrow E_{\mathrm{adj}}$ for each $f \in \cW_1$ and the top horizontal map sends the copy of $[1]$ indexed by $f$ to the morphism $f$ in $\cP$. This pushout exists since $\twocat$ is cocomplete and all the data involved is small. 

	The $2$-category $\cP[\cW_1^{-1}]$ has the same objects as $\cP$. Each morphism $f \in \cW_1$ is now part of a formally added adjoint equivalence data $(f, f^{-1}, \eta_f, \epsilon_f)$ in $\cP[\cW_1^{-1}]$, where the data of $f^{-1},\eta_f$ and $\epsilon_f$ is added to $\cP[\cW_1^{-1}]$ regardless of whether $f$ already was an equivalence in $\cP$. The functor $L_1$ in the pushout square is the identity on objects and the inclusion on morphisms and $2$-cells. 
    
    Next, we localize $\cP[\cW_1^{-1}]$ at $\cW_2$ by considering the following pushout in $\twocat$:
	\[\begin{tikzcd}[ampersand replacement=\&,cramped, sep=small]
		\& \cP \& \\
		{\displaystyle\coprod_{\alpha \in \cW_2} \Sigma[1]} \&\& {\cP[\cW_1^{-1}]} \\
		\\
		{\displaystyle\coprod_{\alpha \in \cW_2} \Sigma\interval} \&\& {{\cP}[{\cW}^{-1}]}
		\arrow["{L_1}", hook, from=1-2, to=2-3]
		\arrow[hook, from=2-1, to=1-2]
		\arrow[""{name=0, anchor=center, inner sep=0}, from=2-1, to=2-3]
		\arrow[hook,from=2-1, to=4-1]
		\arrow["{L_2}", dashed, hook, from=2-3, to=4-3]
		\arrow[dashed, from=4-1, to=4-3]
		\arrow["\lrcorner"{anchor=center, pos=0.125, rotate=180}, draw=none, from=4-3, to=2-1]
		\arrow["{=}"{description}, draw=none, from=0, to=1-2]
	\end{tikzcd}\]
	The left vertical map is induced by the inclusions $\Sigma[1] \hookrightarrow \Sigma\interval$, and the top horizontal map sends the copy of $\Sigma[1]$ indexed by some $\alpha \in \cW_2$ to  $L_1(\alpha)$ in $\cP[\cW_1^{-1}]$. Again, this pushout exists since $\twocat$ is cocomplete and all the data involved is small. We denote the resulting $2$-category by $\cP[\cW^{-1}]$ with $\cW = \{ \cW_1, \cW_2 \}$, and use $L_2\colon \cP[\cW_1^{-1}] \to \cP[\cW^{-1}]$ to denote the right vertical 2-functor in the pushout.
    The composite
	\[ \Loc \colon \cP \xrightarrow{L_1} \cP[\cW_1^{-1}] \xrightarrow{L_2} \cP[\cW^{-1}] \]
	is the identity on objects and the inclusion on morphisms and $2$-cells.
\end{constr}

\begin{rmk}\label{rmk:localization_properties}
	By construction, the 2-functor $\Loc\colon \cP \to \cP[\cW^{-1}]$ above satisfies the following properties:
	\begin{itemize}
		\item For every morphism $f \in \cW_1$, $\Loc(f)$ is an adjoint equivalence in $\cP[\cW^{-1}]$. Indeed, $L_1(f)$ is an adjoint equivalence in $\cP[\cW_1^{-1}]$ and $L_2$ preserves adjoint equivalences since it is a 2-functor.
		\item Similarly, for every 2-cell $\alpha \in \cW_2$, $\Loc(\alpha)$ is an isomorphism in $\cP[\cW^{-1}]$ since $L_1$ is the inclusion on $2$-cells and $L_2$ makes $L_1(\alpha)$ into an isomorphism, by construction.
	\end{itemize}
	Given any $2$-category $\cD$, the 2-functor $\Loc$ induces, by precomposition, a functor
	\begin{equation}\label{eq:universal_property_localization}
		\Loc^*\colon \twocat(\cP[\cW^{-1}], \cD) \to \twocat^{\cW}(\cP, \cD)
	\end{equation}
	where $\twocat^{\cW}(\cP, \cD)$ denotes the $2$-category of 2-functors $F\colon \cP \to \cD$ sending $\cW_1$ to equivalences and $\cW_2$ to isomorphisms. 
	The functor $\Loc^*$ is strictly surjective on objects and faithful, but in general, it is \emph{not} full.  
\end{rmk}

\begin{notation}
    Given a marked 2-category $(\cP,E^1_\cP,E^2_\cP)$, we will denote by $\cP[E_\cP^{-1}]$ the localization of the 2-category $\cP$ at the set $E^1_\cP$ of morphisms and $E^2_\cP$ of 2-cells in the sense of \cref{constr:pushout_localization}.
\end{notation}

\begin{rmk}\label{remark:Phat_preserves_normal_forms}
	Any object $P\colon (\cP, E^1_\cP, E^2_\cP) \to \cC^\sharp$ in $\mslicetwocat$ induces a 2-functor $\widehat{P}\colon \cP[E_{\cP}^{-1}] \to \cC[\cC^{-1}]$, where the target denotes the localization of $\cC$ at all its morphisms and 2-cells. Indeed, the functor $P$ induces a morphism of spans
    
	\[\begin{tikzcd}[ampersand replacement=\&,cramped]
		{\displaystyle \coprod_{u \in {E^1_{\cP}}} E_{\mathrm{adj}}} \& {\displaystyle\coprod_{u \in {E^1_{\cP}}}[1]} \& \cP \\
		{\displaystyle\coprod_{f \in \operatorname{mor}\cC}E_{\mathrm{adj}}} \& {\displaystyle\coprod_{f \in \operatorname{mor}\cC}[1]} \& \cC
		\arrow[from=1-1, to=2-1]
		\arrow[hook', from=1-2, to=1-1]
		\arrow[hook, from=1-2, to=1-3]
		\arrow[from=1-2, to=2-2]
		\arrow["P", from=1-3, to=2-3]
		\arrow[hook', from=2-2, to=2-1]
		\arrow[hook, from=2-2, to=2-3]
	\end{tikzcd}\]
	in which the left and middle vertical functors are the unique 2-functors reindexing
	the coproducts along $u \mapsto Pu$ (identically on each component). This gives a 2-functor between pushouts $$P_1 \colon \cP[\smash{{E^1_{\cP}}}^{-1}] \to \cC[\operatorname{mor}\cC^{-1}].$$ 
	Similarly, $P$ also induces a canonical morphism of spans
	\[\begin{tikzcd}[ampersand replacement=\&,cramped]
		{\displaystyle \coprod_{\alpha \in {E^2_{\cP}}} \Sigma \interval} \& {\displaystyle\coprod_{\alpha \in {E^2_{\cP}}}\Sigma[1]} \& {\cP[{E^1_{\cP}}^{-1}]} \\
		{\displaystyle\coprod_{\beta \in \operatorname{2cell}\cC} \Sigma \interval} \& {\displaystyle\coprod_{\beta \in \operatorname{2cell}\cC}\Sigma[1]} \& {\cC[\operatorname{mor}\cC^{-1}]}
		\arrow[from=1-1, to=2-1]
		\arrow[hook', from=1-2, to=1-1]
		\arrow[hook, from=1-2, to=1-3]
		\arrow[from=1-2, to=2-2]
		\arrow["{P_1}", from=1-3, to=2-3]
		\arrow[hook', from=2-2, to=2-1]
		\arrow[hook, from=2-2, to=2-3]
	\end{tikzcd}\]
    and the universal property of pushouts yields the desired 2-functor $$\widehat{P} \colon \cP[E_{\cP}^{-1}] \to \cC[\cC^{-1}].$$

    Moreover, one can similarly check that any marked 2-functor $H\colon P\to Q$ over $\cC^\sharp$ induces a 2-functor $\widehat{H}\colon\widehat{P}\to \widehat{Q}$ over $\cC[\cC^{-1}]$, induced by the universal property of the pushout.
\end{rmk}

\subsection{Slice 2-categories}

A key second ingredient in the construction of the desired functor $L\colon\mslicetwocat\to 2\Fun[\cC^\coop,\ttwocat]$ is the notion of an oplax slice 2-category, and its behaviour with respect to 2-functors and 2-natural transformations. For the reader's convenience, we devote this brief subsection to recalling these definitions.

\begin{defn}\label{defn:marked-slices}
    Let $P\colon\cP\to\cC$ be a 2-functor, and let $c$ be an object in $\cC$. The oplax slice 2-category of $\cP$ under $c$, denoted $c\downarrow P$, is defined as follows:
    \begin{itemize}[leftmargin=*]
        \item \label{defn:marked-slices.1} Objects are pairs $(x,f)$ of an object $x\in\cP$ and a morphism $f\colon c\to Px$ in $\cC$.

        \item \label{defn:marked-slices.2} Morphisms $(x,f)\to(x',f')$ are pairs $(u,\varphi)$ of a morphism $u\colon x\to x'$ in $\cP$ and a 2-cell $\varphi\colon f'\Rightarrow Pu.f$ in $\cC$. 
        
        \item \label{defn:marked-slices.3} 2-cells $(u,\varphi)\Rightarrow(v,\psi)$ between morphisms as above consist of a 2-cell $\sigma\colon u\Rightarrow v$ in $\cP$ satisfying $(P\sigma \cdot f).\varphi=\psi$.

        \item \label{defn:marked-slices.5} The identity morphism at an object $(x,f)$ is given by $(\id_x,\id_f)$, and the identity 2-cell at a morphism $(u,\varphi)$ is given by $\id_u$.
        
        \item \label{defn:marked-slices.4} The composite of morphisms \[(x,f)\xrightarrow{(u,\varphi)}(x',f')\xrightarrow{(u',\varphi')}(x'',f'')\] is given by $(u'u,(Pu'\cdot \varphi).\varphi')$.

        \item \label{defn:marked-slices.6} The vertical and horizontal composition of 2-cells are given by the respective compositions of 2-cells in $\cP$. 
    \end{itemize}
    It is tedious but straightforward to verify that the composition of 2-cells is well-defined, and that this gives a 2-category (see for instance \cite[Proposition 7.1.2]{Niles2021}, where a lax version is established for the more general setting of lax functors and bicategories).
    \end{defn}

We have so far described the action of a 2-functor $-\downarrow P\colon\cC^\coop\to\ttwocat$ at the level of objects. We now detail how this 2-functor acts on morphisms. 

\begin{defn}\label{defn:induced-2functor-comma}
     Let $P\colon\cP\to\cC$ be a 2-functor and $r\colon c\to d$ be a morphism in $\cC$. Then whiskering with $r$ induces a 2-functor $r\downarrow P\colon d\downarrow P\to c\downarrow P$ which sends
\begin{itemize}[leftmargin=*]
    \item an object $(x,f)$ in $d\downarrow P$ to the object $(x,f r)$ in $c\downarrow P$,
    \item a morphism $(u,\varphi)$ in $d\downarrow P$ to the morphism $(u,\varphi\cdot r)$ in $c\downarrow P$,
    \item a 2-cell $\sigma\colon(u,\varphi)\Rightarrow(v,\psi)$ in $d\downarrow P$ to itself regarded as a 2-cell $\sigma\colon (u,\varphi\cdot r)\Rightarrow(v,\psi\cdot r)$ in $c\downarrow P$.
\end{itemize}
The fact that this is a 2-functor is established in \cite[Proposition 7.1.9]{Niles2021}. 
\end{defn}

Lastly, we describe the action of the 2-functor $-\downarrow P\colon\cC^\coop\to\ttwocat$ on 2-cells.

\begin{defn} \label{defn:induced-2-natural-comma}
    Let $P\colon\cP\to\cC$ be a 2-functor and $\mu\colon r\Rightarrow s$ be a 2-cell in $\cC$ between morphisms $r,s\colon c\to d$. This induces a 2-natural transformation $\mu\downarrow P\colon s\downarrow P\Rightarrow r\downarrow P$ whose component at an object $(x,f)$ in $d\downarrow P$ is given by the morphism $(\mu\downarrow P)_{(x,f)}=(\id_x,f\cdot\mu)\colon (x,f s)\to(x,f r)$ in $c\downarrow P$.
\end{defn}

\subsection{The homotopy inverse}\label{subsec.defn.L}

We now construct the functor $L\colon\mslicetwocat\to 2\Fun[\cC^\coop,\ttwocat]$. The action of this functor on objects is quite involved, so it is worth providing some motivation ahead of its description. 

Ideally, we would like to take a marked 2-functor $P\colon\markedscat{\cP}\to\Csharp$ to a map $\cC^\coop\to\ttwocat$ that takes an object $c\in\cC$ to its fiber $\cP_c$ of $P$ over $c$. This is a well-understood construction, and is in fact what Buckley uses as the pseudo-inverse to the Grothendieck construction in \cite[Construction 3.3.5]{Buckley2014}. However, this map is not a strict 2-functor, but merely a pseudo-functor. A priori, we could attempt to overcome this obstruction by considering some cofibrant replacements of $\cP$ and $\cC$, which would allow us to work with free underlying 1-categories. Unfortunately, even if we solve this issue of strictness on composites of morphisms, we find a further obstruction: if $\alpha$ is a 2-cell in $\cC$ as below left, then the corresponding 2-cell below right
\[
    \begin{tikzcd}
    c & d && \cP_c & \cP_d 
    \arrow[""{name=0, anchor=center, inner sep=0}, "f", curve={height=-12pt}, from=1-1, to=1-2]
    \arrow[""{name=1, anchor=center, inner sep=0}, "g"', curve={height=12pt}, from=1-1, to=1-2]
    \arrow[""{name=2, anchor=center, inner sep=0}, "", curve={height=-12pt}, from=1-5, to=1-4]
    \arrow[""{name=3, anchor=center, inner sep=0}, "", curve={height=12pt}, from=1-5, to=1-4]
    \arrow["\alpha"'{xshift=-2pt}, shorten <=3pt, shorten >=3pt, Rightarrow, from=0, to=1]
    \arrow[""'{xshift=-2pt}, shorten <={3pt}, shorten >=3pt, Rightarrow, from=2, to=3]
    \end{tikzcd}
\]
is a pseudo-2-natural transformation, instead of a strict one. 

Our solution is to use an oplax slice 2-category construction, which will produce strict 2-functors and 2-natural transformations as needed. To deal with the markings appropriately, we consider the oplax slice $c\downarrow \widehat{P}$ of the localization of $P$ at its markings, as opposed to simply $c\downarrow P$. Finally, the full oplax comma would be too large, and so in order to capture the correct ``homotopy type'', we must restrict to a smaller sub-2-category which will be biequivalent to the fiber $\cP_c$. This is the construction that we now present.

\begin{constr}\label{constr:leftadjoint}
	Let $P\colon (\cP, E^1_\cP, E^2_\cP) \to \cC^\sharp$ be an object in $\mslicetwocat$, and denote by $\widehat{P} \colon \cP[E_{\cP}^{-1}] \to \cC[\cC^{-1}]$ the corresponding 2-functor on localizations from \cref{remark:Phat_preserves_normal_forms}. The 2-functor $L(P)\colon \cC^{\coop}\to\ttwocat$ is defined as follows. 
    
    For an object $c\in\cC$, the $2$-category $L(P)(c)$ is the smallest sub-$2$-category of the oplax comma  $c\downarrow\widehat{P}$ containing the following data:
	\begin{itemize}[leftmargin=*]
		\item Objects $(x, f)$ consisting of an object $x\in\cP$ and a morphism $f\colon c \to Px$ in $\cC$. 
		\item  Morphisms of the form
        \begin{itemize}[leftmargin=*]
            \item $(u, \varphi)\colon (x, f) \to (x', f')$ given by a morphism $u\colon x \to x'$ in $\cP$ and a $2$-cell $\varphi\colon f' \Rightarrow Pu.f$ in $\cC$;

		\[
        \begin{tikzcd}[ampersand replacement=\&,cramped,sep=small]
				\&\& c \&\& \\
				\& {} \&  \& {} \\
				Px \&\&\&\& Px'
				\arrow["{f'}", from=1-3, to=3-5]
				\arrow["{f}"', from=1-3, to=3-1]
				\arrow["{\varphi}"', shift left=4, between={0.3}{0.7}, Rightarrow, from=2-4, to=2-2]
				\arrow["{Pu}"', from=3-1, to=3-5]
		\end{tikzcd}
        \]
            \item $(u^{-1},\widehat{P}\eta_u.f)\colon (x',Pu.f)\to (x,f)$ where $u^{-1}$ and $\eta_u$ are part of the formal equivalence data added to some marked morphism $u\colon x\to x'$ in $\cP$.
        \[
        \begin{tikzcd}[ampersand replacement=\&,cramped,sep=scriptsize]
        	\&\& c \&\& \\
        	 {} \& Px \& {} \& {} \\
        	 Px' \& {} \& {} \&\& Px
        	\arrow["{f}"', from=1-3, to=2-2]
        	\arrow["{f}", from=1-3, to=3-5]
        	\arrow["{Pu}"', from=2-2, to=3-1]
        	\arrow["{=}"{description}, draw=none, from=2-2, to=2-4]
        	\arrow["{\widehat{P}u^{-1}}", tail reversed, no head, from=3-5, to=3-1]
        	\arrow["{\widehat{P}\eta_u}"', shift right=2, Rightarrow, from=3-3, to=3-2]
        	\arrow[equals, from=3-5, to=2-2]
        \end{tikzcd}
        \]
                \end{itemize}

		\item   $2$-cells of the form
        \begin{itemize}[leftmargin=*]
            \item $\sigma\colon (u,\varphi)\Rightarrow (v,\psi)$ for any $u,v\colon x\to x'$ in $\cP$ and $\sigma\colon u\Rightarrow v$ in $\cP$,
            \[\begin{tikzcd}[cramped,sep=small]
            	&& c &&&&&& c && \\
            	& {} & {} & {} & {} && {} & {} && {} \\
            	Px &&&& {Px'} && Px &&&& {Px'} \\
            	&& {}
            	\arrow["f"', from=1-3, to=3-1]
            	\arrow["{f'}", from=1-3, to=3-5]
            	\arrow["f"', from=1-9, to=3-7]
            	\arrow["{f'}", from=1-9, to=3-11]
            	\arrow["{P\sigma}", between={0.4}{0.6}, Rightarrow, from=2-3, to=4-3]
            	\arrow["\varphi"', between={0.3}{0.7}, Rightarrow, from=2-4, to=2-2]
            	\arrow["{=}"{description}, draw=none, from=2-5, to=2-7]
            	\arrow["\psi"', between={0.3}{0.7}, Rightarrow, from=2-10, to=2-8,, shift left=2]
            	\arrow["Pu", curve={height=-12pt}, from=3-1, to=3-5]
            	\arrow["Pv"', curve={height=12pt}, from=3-1, to=3-5]
            	\arrow["Pv"', from=3-7, to=3-11]
            \end{tikzcd}\]
			where $\varphi\colon f'\Rightarrow Pu.f$ and $\psi\colon f'\Rightarrow Pv.f$ in $\cC$; 
            \item  $\beta^{-1}\colon (v, P\beta\cdot f)\Rightarrow (u, \id)$ where $\beta^{-1}$ is the inverse of a marked 2-cell $\beta\colon u\Rightarrow v$ in $\cP$
            \[\begin{tikzcd}[cramped, row sep = 10pt, column sep=8pt]
            	&&&& c &&&&&&&&&& c &&&& \\
            	\\
            	&&& {} && {} & Px && {} && {} &&& {} && {} & Px \\
            	&&&& {} & {} && {} \\
            	Px &&&&&&&& {Px'} && Px &&&&&&&& {Px'} \\
            	&&&& {}
            	\arrow["f", from=1-5, to=3-7]
            	\arrow["f"', from=1-5, to=5-1]
            	\arrow["f", from=1-15, to=3-17]
            	\arrow["f"', from=1-15, to=5-11]
            	\arrow["{=}"{description}, draw=none, from=3-4, to=3-6]
            	\arrow["Pu", from=3-7, to=5-9]
            	\arrow["{=}"{description}, draw=none, from=3-9, to=3-11]
            	\arrow["{=}"{description}, shift left=5, draw=none, from=3-16, to=3-14]
            	\arrow["Pu", from=3-17, to=5-19]
            	\arrow["{\widehat{P}\beta^{-1}}",shift left=3, between={0.3}{0.7}, Rightarrow, from=4-5, to=6-5]
            	\arrow["\beta"', shift right = -1, between={0.3}{0.7}, Rightarrow, from=4-8, to=4-6]
            	\arrow[equals, from=5-1, to=3-7]
            	\arrow["Pv"{description}, curve={height=-18pt}, from=5-1, to=5-9]
            	\arrow["Pu"', curve={height=18pt}, from=5-1, to=5-9]
            	\arrow["Pu"'{description}, from=5-11, to=5-19]
            \end{tikzcd}\]

			\item $\eta_u\colon (\id_{Px},\id)\Rightarrow (u^{-1}, \widehat{P}\eta_u\cdot f).(u,\id)$ depicted below, and its inverse $\eta_u^{-1}$
            \[\begin{tikzcd}[cramped, sep=small]
            	&&& c &&&&&&&& c &&& \\
            	\\
            	&& {} && {} && {} && {} &&& Pa && {} \\
            	&&&&&&&&& {} & {} & {} && {} \\
            	Px &&&&&& Px && Px &&& {Px'} &&& {Px'} \\
            	&&&&&&&&&&& {}
            	\arrow["f"', from=1-4, to=5-1]
            	\arrow["f", from=1-4, to=5-7]
            	\arrow["f"', from=1-12, to=3-12]
            	\arrow["f"', from=1-12, to=5-9]
            	\arrow["{f'}", from=1-12, to=5-15]
            	\arrow["{=}"{description}, shift left=5, draw=none, from=3-5, to=3-3]
            	\arrow["{=}"{description}, draw=none, from=3-9, to=3-7]
            	\arrow["Pu"', from=3-12, to=5-12]
            	\arrow[equals, from=3-12, to=5-15]
            	\arrow["{=}"{description, pos=0.6}, draw=none, from=3-14, to=3-12]
            	\arrow["{=}"{description}, draw=none, from=4-12, to=4-10]
            	\arrow["{\widehat{P}\eta_{u}^{-1}}"'{pos=0.8},"\cong"{pos=0.8}, between={0.7}{1}, Rightarrow, from=4-12, to=6-12]
            	\arrow["{\widehat{P}\eta_u}"', shift left=5, between={0.3}{0.7}, Rightarrow, from=4-14, to=4-12]
            	\arrow[Rightarrow, from=5-1, to=5-7]
            	\arrow["Pu"{description}, from=5-9, to=5-12]
            	\arrow["{\id_{Pa}}"', curve={height=30pt}, equals, from=5-9, to=5-15]
            	\arrow["{\widehat{P}u^{-1}}"{description}, from=5-12, to=5-15]
            \end{tikzcd}\]
            as well as $\epsilon_u\colon (u,\id). (u^{-1},\widehat{P}\eta_u\cdot f)\Rightarrow (\id_{Px'},\id)$ depicted below, and its inverse $\epsilon_u^{-1}$ 
            \[\begin{tikzcd}[cramped,row sep=10pt, column sep=8pt]
            	&&&& c &&&&&&&&&& c &&&& \\
            	\\
            	&& Px && {} && Px && {} && {} && Px & {} && {} & Px \\
            	& {} && {} \\
            	{Px'} &&&& Px &&&& {Px'} && {Px'} &&&&&&&& {Px'} \\
            	\\
            	&&&& {}
            	\arrow["f"', from=1-5, to=3-3]
            	\arrow["f", from=1-5, to=3-7]
            	\arrow["f"{description}, from=1-5, to=5-5]
            	\arrow["f"', from=1-15, to=3-13]
            	\arrow["f", from=1-15, to=3-17]
            	\arrow["{=}"{marking, allow upside down}, draw=none, from=3-3, to=3-5]
            	\arrow["Pu"', from=3-3, to=5-1]
            	\arrow[equals, from=3-3, to=5-5]
            	\arrow["{=}"{marking, allow upside down}, draw=none, from=3-5, to=3-7]
            	\arrow["Pu", from=3-7, to=5-9]
            	\arrow["{=}"{description}, draw=none, from=3-9, to=3-11]
            	\arrow["Pu"', from=3-13, to=5-11]
            	\arrow["{=}"{description}, shift right=5, draw=none, from=3-14, to=3-16]
            	\arrow["Pu", from=3-17, to=5-19]
            	\arrow["{\widehat{P}\eta_u}"', shift left=3, between={0.3}{0.7}, Rightarrow, from=4-4, to=4-2]
            	\arrow["{\widehat{P}u^{-1}}"{description}, from=5-1, to=5-5]
            	\arrow[curve={height=30pt}, Rightarrow, from=5-1, to=5-9]
            	\arrow["Pu"{description}, from=5-5, to=5-9]
            	\arrow["{\widehat{P}\epsilon_u}"'{pos=0.2},"\cong"{pos=0.2}, between={0.1}{0.6}, Rightarrow, from=5-5, to=7-5]
            	\arrow[equals, from=5-11, to=5-19]
            \end{tikzcd}\]
            for any marked morphism $u\colon x\to x'$ in $\cP$, 
			where $\eta_u$ and $\epsilon_u$ denote the formal unit and counit of the adjoint equivalence $u \dashv u^{-1}$. 
			\end{itemize}
		\end{itemize}

        Given a morphism $r\colon c\to c'$ in $\cC$, we get a 2-functor $r\downarrow \widehat{P}\colon c'\downarrow \widehat{P}\to c\downarrow\widehat{P}$ by \cref{defn:induced-2functor-comma}. This 2-functor clearly sends the generating morphisms and 2-cells in $L(P)(c')$ to generating morphisms and 2-cells in $L(P)(c)$. Hence, it restricts to the sub-2-categories, providing a 2-functor $L(P)(r)$. 

        Similarly, given a 2-cell $\beta\colon r\Rightarrow s$ in $\cC$, we get a 2-natural transformation $\beta\downarrow\widehat{P}\colon s\downarrow\widehat{P}\Rightarrow r\downarrow\widehat{P}$ by  \cref{defn:induced-2-natural-comma}, whose component on an object $(x,f)\in L(P)(c')$ is the morphism $(\id_x, f\cdot \beta)$. This is a morphism in $L(P)(c)$, and thus, $\beta\downarrow\widehat{P}$ restricts to a 2-natural transformation $L(P)(\beta)$.  
	\end{constr}

    \begin{rmk}\label{marked.become.invertible}
        The 2-category $L(P)(c)$ is designed so that every marked morphism $u\colon x\to x'$ in $\cP$ gives rise to an adjoint equivalence $(u,\id)\colon (x, f)\to (x', Pu.f)$ in $L(P)(c)$. Indeed, an inverse equivalence is given by $(u^{-1},\widehat{P}\eta_u\cdot f)$ where $u^{-1}$ denotes the formal inverse adjoint equivalence in $\cP[E_\cP^{-1}]$, with unit and counit given by the 2-cells $\eta_u$ and $\epsilon_u$. This is the reason why we include this morphism and these 2-cells (and their inverses) in the definition of $L(P)(c)$.

        Moreover, every marked 2-cell $\beta\colon u\Rightarrow v$ in $\cP$ gives rise to an invertible 2-cell $\beta\colon (u,\id)\Rightarrow (v,P\beta\cdot f)$ in $L(P)(c)$, whose inverse is given by $\beta^{-1}$. This explains the inclusion of the 2-cell $\beta^{-1}$ in the definition of $L(P)(c)$. 
      \end{rmk}

      \begin{rmk}\label{gamma.u.exists}
      Let $u\colon x\to x'$ be a marked morphism in $\cP$ which is also an equivalence, and let $(u,u^{-1},\eta_u,\epsilon_u)$ be the formally added adjoint equivalence data in $\cP[E_\cP^{-1}]$, and   $(u,\underline{u}^{-1},\eta_{\underline{u}},\epsilon_{\underline{u}})$ be any adjoint equivalence data present in $\cP$. By \cref{marked.become.invertible}, we have that both $ (u^{-1}, \widehat{P}\eta_u\cdot f)$ and $(\underline{u}^{-1}, P\eta_{\underline{u}}\cdot f)$ are inverse equivalences to $(u,\id)$ in $L(P)(c)$, and so there exists a unique invertible 2-cell  $\gamma_u\colon (u^{-1}, \widehat{P}\eta_u\cdot f)\Rightarrow (\underline{u}^{-1}, P\eta_{\underline{u}}\cdot f)$ in $L(P)(c)$ such that $\eta_{\underline{u}}=(\gamma_u\cdot u)\eta_u$. Indeed, this arises from the unique invertible 2-cell  $\gamma_u\colon u^{-1}\Rightarrow \underline{u}^{-1}$  in $\cP[E_\cP^{-1}]$ with the same property.
      \end{rmk}

\begin{constr}
    The functor $L\colon\mslicetwocat\to 2\Fun[\cC^\coop,\ttwocat]$ is defined as follows. On objects, it sends a marked 2-functor $P\colon\markedscat{\cP}\to\cC^\sharp$ to the 2-functor $L(P)$ of \cref{constr:leftadjoint}. 
    
    On morphisms, given a marked 2-functor $H\colon P\to Q$ over $\cC^\sharp$, we get an induced 2-functor $\widehat{H}\colon\widehat{P}\to\widehat{Q}$ by \cref{remark:Phat_preserves_normal_forms}, which then induces a 2-natural transformation $-\downarrow\widehat{H}\colon -\downarrow\widehat{P}\Rightarrow -\downarrow\widehat{Q}$. For each object $c\in\cC$, the component of this 2-natural transformation is the 2-functor $c\downarrow\widehat{P}\to c\downarrow\widehat{Q}$ given by $(x,f)\mapsto (Hx,f)$, $(u,\varphi)\mapsto (\widehat{H}u,\varphi)$, and $\sigma\mapsto \widehat{H}\sigma$. This restricts to a 2-functor $L(P)(c)\to L(Q)(c)$ since it respects all the generators: by construction of $\widehat{H}$, we have that $(Hu,\widehat{H}u^{-1}, \widehat{H}\eta_u,\widehat{H}\epsilon_u)=(Hu,(Hu)^{-1},\eta_{Hu},\epsilon_{Hu})$ for any marked $u$ in $\cP$, and that $\widehat{H}\beta^{-1}=(H\beta)^{-1}$ for any marked 2-cell $\beta$ in $\cP$. We can thus define $L(H)\colon L(P)\Rightarrow L(Q)$ as the restriction of the 2-natural transformation $-\downarrow\widehat{H}$, and we see that this construction is functorial.
\end{constr}

\subsection{The unit}

Our next goal is to define a natural transformation $\eta\colon \id\Rightarrow \int_\cC^+ L$ which is a pointwise weak equivalence; this will descend to a natural isomorphism once we pass to the homotopy categories. 

\begin{constr}\label{constr.unit}
    Given a marked 2-functor $P\colon\markedscat{\cP}\to\Csharp$, we define a marked 2-functor $\eta_P\colon P \to\int_\cC^+ L(P)$ over $\Csharp$ by: 
    \[
		\begin{aligned}
		x &\mapsto (Px, (x,\id_{Px})) && \text{on objects,}\\
		u &\mapsto (Pu, (u,\id_{Pu})) && \text{on morphisms,}\\
		\beta &\mapsto (P\beta, \beta) && \text{on $2$-cells.}
		\end{aligned}
	\] The fact that this is a 2-functor over $\Csharp$ is immediate. To see that $\eta_P$ preserves the markings,  let $u\colon x\to x'$ be in $E^1_\cP$. Then $(u,\id_{Pu})$ is an equivalence in $L(P)(Px)$ as explained in \cref{marked.become.invertible}, and since the marked morphisms in $\int_\cC^+ L(P)$ are precisely those morphisms whose second component is an equivalence (cf. \cref{prop:markingsinGC}), it follows that $\eta_P(u)$ is a marked morphism in $\int_\cC^+ L(P)$. The argument for $2$-cells is analogous.
\end{constr}

\begin{lem}\label{eta.is.natural}
    The 2-functors $\eta_P$ of \cref{constr.unit} assemble into a 2-natural transformation $\eta\colon \id\Rightarrow\int_\cC^+ L$.
\end{lem}
\begin{proof}
    Given a marked 2-functor $H\colon P\to Q$ over $\Csharp$, we must show that $\eta_Q H=\int_\cC^+ H \eta_P$. This is immediate from the definitions, as both composites have actions given by
      \[
		\begin{aligned}
		x &\mapsto (Px=QHx, (Hx,\id_{Px})) && \text{on objects,}\\
		u &\mapsto (Pu=QHu, (Hu,\id_{Pu})) && \text{on morphisms,}\\
		\beta &\mapsto (P\beta=QH\beta, H\beta) && \text{on $2$-cells.}
		\end{aligned}
	\]
\end{proof}

We now show the desired result regarding the behaviour of the unit on fibrant objects. In order to show that each component $\eta_P\colon P\to \int_\cC^+ L(P)$ is a biequivalence, thanks to \cref{bieq.on.fibers}, it will be enough to inspect each fiber at a time.

    \begin{prop}\label{Lem:unitiswequiv}
        Let $P\colon \markedscat{\cP} \to \cC^{\sharp}$ be a fibrant object in $\mslicetwocat$. Then the unit $\eta_P\colon P \to \int^+_{\cC} L(P)$ is a weak equivalence in $\mslicetwocat$.
    \end{prop}
    \begin{proof}
By \cref{bieq.on.fibers}, it suffices to show that the 2-functor induced on fibers $$(\eta_P)_c\colon P_c\to (\textstyle{\int}^+_{\cC}L(P))_c=L(P)(c)$$  is a biequivalence for each $c\in\cC$. Since $P$ is fibrant, we may start by choosing a cleavage consisting of $P$-cartesian lifts $\chi_{f,x}\colon f^*x\to x$ for each morphism $f\colon c\to Px$ in $\cC$, and $\chi_{\varphi}\colon v\Rightarrow u$ for each 2-cell $\varphi\colon f\Rightarrow Pu$ in $\cC$. We ask only that $\chi_{\id,x}=\id_x$ and $\chi_{\id_{Pu}}=\id_u$ but no other compatibilities (e.g.\ between composites or whiskerings), and so this is guaranteed to exist by the axiom of choice. 
Now, for every object $(x,f)$ in $L(P)(c)$, we consider the map $(\chi_{f,x},\id_f)\colon (f^*x,\id_c)\to (x,f)$. Note that $(f^*x,\id_c)=(\eta_P)_c (x)$ and that $(\chi_{f,x},\id)$ is an equivalence in $L(P)(c)$ since the map $\chi_{f,x}$ is cartesian and hence marked in $\cP$; see \cref{marked.become.invertible}. This shows that  $(\eta_P)_c$ is essentially surjective on objects.

To show that $(\eta_P)_c$ is full on morphisms up to an invertible 2-cell, it suffices to prove that if $(u,\varphi)\colon (x,f)\to (y,g)$ is either of the two types of generating morphisms in $L(P)(c)$, then there exists an invertible 2-cell $\tau$ in $L(P)(c)$ of the form
\[\begin{tikzcd}[cramped, sep=12pt]
	{(x,f)} && {(y,g)} \\
	\\
	{(f^{\ast}x,\id_c)} && {(g^{\ast}y, \id_c)}
	\arrow["{(u,\varphi)}", from=1-1, to=1-3]
	\arrow["{(\chi_{f,x},\id)}", "\mathrel{\rotatebox[origin=c]{-90}{$\simeq$}}"', from=3-1, to=1-1]
	\arrow[from=3-1, to=3-3]
	\arrow["\tau"', between={0.3}{0.7},"\cong", Rightarrow, from=3-3, to=1-1]
	\arrow["{(\chi_{g,y}, \id)}"',"\mathrel{\rotatebox[origin=c]{90}{$\simeq$}}", from=3-3, to=1-3]
\end{tikzcd}\]
where the bottom horizontal morphism is in the image of $(\eta_P)_c$. 

The first type of morphism is of the form $(u,\varphi)$ for $u\colon x\to y$ in $\cP$ and $\varphi\colon g\Rightarrow Pu.f$ in $\cC$. Recall that $f=P(\chi_{f,x})$, and hence, since $P$ is a 2-fibration, this lifts to a cartesian 2-cell $\chi_{\varphi}\colon \widetilde{g}\Rightarrow u.\chi_{f,x}$ in $\cP$. Moreover, we can also consider the lifting problem below,
\[\begin{tikzcd}[row sep = small]
				f^*x &&&& c \\
				&& {} && {} \\
				g^*y && y && c && Py
				\arrow["\widetilde{g}", from=1-1, to=3-3]
				\arrow["\id"', from=1-5, to=3-5]
				\arrow[""{name=0, anchor=center, inner sep=0}, "g", from=1-5, to=3-7]
				\arrow["P"{pos=0.3}, shorten >=19pt, maps to, from=2-3, to=2-5]
				\arrow["\chi_{g,y}"', from=3-1, to=3-3]
				\arrow["g"', from=3-5, to=3-7]
				\arrow["\id", shorten <=8pt, shorten >=8pt, Rightarrow, from=3-5, to=0]
			\end{tikzcd}\]
which, as $\chi_{g,y}$ is cartesian, admits a lift given by a morphism $u^*\colon f^*x\to g^* y$ and an invertible 2-cell $\lambda\colon \chi_{g,y}u^*\Rightarrow \widetilde{g}$. Note that this gives a morphism $(\eta_P)_c(u^*)=(u^*,\id)\colon (f^*x,\id_c)\to (g^*y,\id_c)$. The desired 2-cell $\tau$ is then the composite $\chi_{\varphi}\lambda$, which is invertible in $L(P)(c)$ as both $\chi_{\varphi}$ and $\lambda$ are marked in $\cP$; see \cref{marked.become.invertible}. 

The second type of morphism is of the form $(u^{-1},\widehat{P}\eta_u\cdot f)\colon (y,Pu\cdot f)\to (x,f)$ which is the formal inverse adjoint equivalence to the morphism $(u,\id)$ for a marked morphism $u\colon x\to y$ in $\cP$. If we run the argument above for the morphism $(u,\id)$ to get $\tau\colon (\chi_{Pu.f,y},\id)(u^*,\id)\Rightarrow (u,\id)(\chi_{f,x},\id)$, the fact that $\varphi=\id$ implies that $\chi_\varphi=\id$, and in particular, we have $\widetilde{Pu.f}=u.\chi_{f,x}$, which is marked. Since we have an invertible 2-cell $\lambda\colon \chi_{Pu.f,y}u^*\Rightarrow \widetilde{Pu.f}$, we see that $\chi_{Pu.f,y}u^*$ is also marked (by \enumref{prop:propertiescartesian}{0}), so $u^*$ must be marked as well (by \enumref{prop:propertiescartesian}{1.5}). But we also know that $Pu^*=\id$ by construction, so \enumref{prop:propertiescartesian}{2} tells us that $u^*$ is an equivalence. If $(u^*, (u^*)^{-1},\eta_{u^*},\epsilon_{u^*})$ denotes some adjoint equivalence data for $u^*$ in $\cP_c$, one can then check that the following provides the desired invertible 2-cell in $L(P)(c)$:

\begin{equation}\label{eqn.tau.inverse}
\begin{tikzcd}[cramped, column sep=18pt, row sep=10pt]
	{(y,Pu.f)} && {(x,f)} &&&&& \bullet && \bullet && \\
	&&&& {\tau' } & {} & {} &&& \bullet \\
	{((Pu.f)^{\ast}y,\id_c)} && {(f^{\ast}x,\id_c)} &&& \bullet && \bullet && \bullet && \bullet
	\arrow["{(u^{-1},\widehat{P}\eta_u\cdot f)}", from=1-1, to=1-3]
	\arrow["{\chi_{Pu.f,y}}", from=1-8, to=1-10]
	\arrow["{u^{-1}}", from=1-10, to=3-12]
	\arrow["{=}"{description, pos=0.9}, draw=none, from=2-5, to=2-6]
	\arrow["{(\chi_{Pu.f,y},\id)}","\mathrel{\rotatebox[origin=c]{-90}{$\simeq$}}"', from=3-1, to=1-1]
	\arrow[from=3-1, to=3-3]
	\arrow["{\tau'}"',"\cong", between={0.3}{0.7}, Rightarrow, from=3-3, to=1-1]
	\arrow["{(\chi_{f,x}, \id)}"',"\mathrel{\rotatebox[origin=c]{90}{$\simeq$}}", from=3-3, to=1-3]
	\arrow[""{name=0, anchor=center, inner sep=0}, equals, from=3-6, to=1-8]
	\arrow["{(u^{\ast})^{-1}}"', from=3-6, to=3-8]
	\arrow["{u^{\ast}}"{description}, from=3-8, to=1-8]
	\arrow["{\chi_{f,x}}"', from=3-8, to=3-10]
	\arrow["\tau"', between={0.3}{0.7}, Rightarrow, from=3-10, to=1-8]
	\arrow["u"{description}, from=3-10, to=1-10]
	\arrow[""{name=1, anchor=center, inner sep=0}, equals, from=3-10, to=3-12]
	\arrow["{\eta_{u^*}}"{pos=0.3}, between={0.1}{0.5}, Leftarrow, from=2-10, to=1]
	\arrow["{\epsilon_{u^*}}", between={0.3}{0.7}, Rightarrow, from=3-8, to=0]
\end{tikzcd}
\end{equation}

Finally, we focus on 2-cells. Note that applying $(\eta_P)_c$ cannot introduce any new relations that were not already present in $\cP$, so we only need to prove that $(\eta_P)_c$ is full on 2-cells. For this, it suffices to show that if $\sigma\colon (u,\varphi)\Rightarrow (v,\psi)$ is any of the generating 2-cells in $L(P)(c)$, then there exists some 2-cell $\sigma^*$ between the previously constructed data as below
\[\begin{tikzcd}[cramped,sep=13]
	&& {} && \\
	{(x,f)} &&&& {(y,g)} \\
	& {} & {} & {} \\
	&& {} \\
	{(f^{\ast}x,\id_c)} & {} && {} & {(g^{\ast}y,\id_c)} \\
	&& {}
	\arrow["{(v,\psi)}", curve={height=-18pt}, dashed, from=2-1, to=2-5]
	\arrow["{(u,\varphi)}"', curve={height=18pt}, from=2-1, to=2-5]
	\arrow["\sigma"', between={0.3}{0.7}, Rightarrow, from=3-3, to=1-3]
	\arrow["{(\chi_{f,x},\id)}","\mathrel{\rotatebox[origin=c]{-90}{$\simeq$}}"', from=5-1, to=2-1]
	\arrow["{(v^{\ast},\id)}", curve={height=-18pt}, dashed, from=5-1, to=5-5]
	\arrow["{(u^{\ast},\id)}"', curve={height=18pt}, from=5-1, to=5-5]
	\arrow["{\tau_u}"{pos=0.8},"\cong"'{pos=0.8}, shift left=1, between={0.5}{1}, Rightarrow, from=5-2, to=3-2]
	\arrow["{\tau_v}"'{pos=0.8},"\cong"{pos=0.8}, shift right =1, between={0.5}{1}, Rightarrow, dashed, from=5-4, to=3-4]
	\arrow["{(\chi_{g,y},\id)}"',"\mathrel{\rotatebox[origin=c]{90}{$\simeq$}}", from=5-5, to=2-5]
	\arrow["{\sigma^{\ast}}"', between={0.3}{0.7}, Rightarrow, from=6-3, to=4-3]
\end{tikzcd}\]
such that $\tau_v (\chi_{g,y}\cdot \sigma^*)= (\sigma\cdot \chi_{f,x})\tau_u$. 

Suppose that $\sigma\colon (u,\varphi)\Rightarrow (v,\psi)$ is a 2-cell of the first type, with $u,v\colon x\to y$, $\sigma\colon u\Rightarrow v$ in $\cP$, $\varphi\colon g\Rightarrow Pu.f$, $\psi\colon g\Rightarrow Pv.f$ in $\cC$, and $\psi=(P\sigma \cdot f)\varphi$. If we consider the lifts $\chi_\varphi\colon \widetilde{g}_u\Rightarrow u.\chi_{f,x}$ of $\varphi$ and $\chi_\psi\colon \widetilde{g}_v\Rightarrow v.\chi_{f,x}$ of $\psi$, we can construct the lifting problem below,
\begin{equation}\label{lift.unit}\begin{tikzcd}[cramped,sep=scriptsize]
	{\widetilde{g}_u} &&&&&& g && \\
	&&& {} && {} \\
	{\widetilde{g}_v} && v.\chi_{f,x} &&&& g && {Pv.f}
	\arrow["{(\sigma\cdot \chi_{f,x}).\chi_{\varphi}}", Rightarrow, from=1-1, to=3-3]
	\arrow[equals, from=1-7, to=3-7]
	\arrow[""{name=0, anchor=center, inner sep=0}, "{(P\sigma\cdot f).\varphi}", Rightarrow, from=1-7, to=3-9]
	\arrow["P", maps to, from=2-4, to=2-6]
	\arrow["{\chi_{\psi}}"', Rightarrow, from=3-1, to=3-3]
	\arrow["\psi"', Rightarrow, from=3-7, to=3-9]
	\arrow["{=}"{description}, draw=none, from=3-7, to=0]
\end{tikzcd}
\end{equation}
which yields some 2-cell $\alpha\colon \widetilde{g}_u\Rightarrow\widetilde{g}_v$ such that $\chi_\psi \alpha=(\sigma\cdot\chi_{f,x})\chi_\varphi$. The 2-cell $\alpha$ then gives rise to the lifting problem 
            \begin{equation}\label{lift.unit.2}
			\begin{tikzcd}
				f^*x &&&&& c \\
				{} && {} &&& {} && {} \\
				g^*y &&& y && c &&& Py
				\arrow["{u^*}"', curve={height=12pt}, from=1-1, to=3-1]
				\arrow["{v^*}", curve={height=-12pt}, dashed, from=1-1, to=3-1]
				\arrow[""{name=0, anchor=center, inner sep=0}, "{\widetilde{g}_v}"{pos=0.6}, curve={height=-12pt}, from=1-1, to=3-4]
				\arrow[""{name=1, anchor=center, inner sep=0}, "\widetilde{g}_u"'{pos=0.6}, curve={height=12pt}, from=1-1, to=3-4]
				\arrow[""{name=2, anchor=center, inner sep=0}, "\id"', curve={height=12pt}, from=1-6, to=3-6]
				\arrow[""{name=3, anchor=center, inner sep=0}, "{\id}", curve={height=-12pt}, dashed, from=1-6, to=3-6]
				\arrow[""{name=4, anchor=center, inner sep=0}, "g"'{pos=0.6}, curve={height=12pt}, from=1-6, to=3-9]
				\arrow[""{name=5, anchor=center, inner sep=0}, "{g'}"{pos=0.6}, curve={height=-12pt}, from=1-6, to=3-9]
				\arrow["{\lambda_v}"', shift left=5, shorten <=14pt, shorten >=27pt, Rightarrow, dashed, from=2-1, to=2-3]
				\arrow["{\id}"', shift left=5, shorten <=14pt, shorten >=27pt, Rightarrow, dashed, from=2-6, to=2-8]
				\arrow[""{name=6, anchor=center, inner sep=0}, "{\chi_{g,y}}"', from=3-1, to=3-4]
				\arrow[""{name=7, anchor=center, inner sep=0}, "g"', from=3-6, to=3-9]
				\arrow["\alpha", shift right=5, shorten <=3pt, shorten >=3pt, Rightarrow, from=1, to=0]
				\arrow["\id", shorten <=5pt, shorten >=5pt, Rightarrow,  from=2, to=3]
				\arrow["\id", shift right=5, shorten <=3pt, shorten >=3pt, Rightarrow, from=4, to=5]
				\arrow["{\lambda_u}"{pos=0.3}, shift left=4, shorten <=4pt, shorten >=20pt, Rightarrow, from=3-1, to=6]
				\arrow["\id"{pos=0.3}, shift left=4, shorten <=4pt, shorten >=20pt, Rightarrow, from=3-6, to=7]
			\end{tikzcd}
			\end{equation} from which we get a 2-cell $\sigma^*\colon u^*\Rightarrow v^*$ in $\cP_c$ satisfying $\lambda_v(\chi_{g,y}\cdot \sigma^*)=\alpha\lambda_u$. Recalling that $\tau_u=\chi_\varphi \lambda_u$ and $\tau_v=\chi_\psi \lambda_v$, we see that  $\sigma^*\colon (u^*,\id)\Rightarrow (v^*,\id)$ is the desired 2-cell.

            Next, let $\beta^{-1}\colon (v,P\beta\cdot f)\Rightarrow (u,\id)$ be the inverse of a marked 2-cell $\beta\colon u\Rightarrow v$. If we apply the above argument to $\beta\colon (u,\id)\Rightarrow (v,P\beta\cdot f)$, the fact that $\beta$ is marked implies that $\chi_\psi$ and $(\beta\cdot \chi_{f,x})\chi_\varphi$ are two cartesian lifts of $\psi$. Therefore, the 2-cell $\alpha\colon\widetilde{g}_u\Rightarrow\widetilde{g}_v$ from (\ref{lift.unit}) is invertible, and thus, using  uniqueness of the lift in (\ref{lift.unit.2}) applied to $\alpha$, its inverse, and their composites, shows that $\beta^*\colon u^*\Rightarrow v^*$ must be invertible as well. The required 2-cell corresponding to the formal inverse $\beta^{-1}$ is then given by $(\beta^*)^{-1}$.

            Lastly, a 2-cell of the form $\eta_u\colon (\id_{Px},\id)\Rightarrow (u^{-1}, \widehat{P}\eta_u\cdot f). (u,\id)$ for some marked morphism $u\colon x\to y$ in $\cP$ has, by construction, a corresponding 2-cell $$\eta_{u^*}\colon (\id_{f^*x},\id)\Rightarrow ((u^*)^{-1}, P\eta_{u^*}). (u^*,\id).$$ This satisfies the required pasting equality $\tau' (\chi_{f,x}\cdot \eta_{u^*})= (\eta_u\cdot \chi_{f,x})\tau$, by construction of $\tau'$ (see \ref{eqn.tau.inverse}) and by the triangle identities. The inverse 2-cell $\eta_u^{-1}$ corresponds to $\eta_{u^*}^{-1}$, and an analogous argument works for the 2-cells $\epsilon_u$ and $\epsilon_u^{-1}$.
    \end{proof}

\subsection{The counit}\label{subsec.counit}

We would now like to define a natural transformation between $\id$ and  $L\int_\cC^+ $ which is pointwise a weak equivalence, but this is not quite possible without further assumptions. Instead, we will define a zig-zag of natural transformations \[\id \Leftarrow L({\textstyle \int}_\cC^+ ,\mathrm{spl})\Rightarrow L{\textstyle \int}_\cC^+\]  which are pointwise weak equivalences. This is enough for our purposes as these will descend to natural isomorphisms when we pass to the homotopy categories.

We start by introducing the functor $L({\textstyle \int}_\cC^+ ,\mathrm{spl})\colon 2\Fun[\cC^\coop,\ttwocat]\to 2\Fun[\cC^\coop,\ttwocat]$.

\begin{constr}
    Given a 2-functor $F\colon\cC^\coop\to\ttwocat$, let $\pi^\mathrm{spl}\colon (\int_\cC^+F,\mathrm{spl})\to\Csharp$ denote the marked 2-functor  whose underlying 2-functor is the projection $\pi\colon\int_\cC^+F\to\cC$, and where the marking on the 2-category $\int_\cC^+ F$ consists of all the morphisms of the form $(f,\id)$, and of all the 2-cells $(\alpha,\widehat{\alpha})$ such that $\widehat{\alpha}$ is invertible.

   This induces a functor $(\int_\cC^+,\mathrm{spl})\colon 2\Fun[\cC^\coop,\ttwocat]\to \mslicetwocat$, since the action of the Grothendieck construction on a natural transformation $\theta$ will map $(f,\id)$ to $(f,\theta_c(\id))=(f,\id)$, and hence, it preserves this new marking, as well.  
\end{constr}

\begin{rmk}\label{rmk.iota}
    Note that this marking constitutes a subset of the usual cartesian marking obtained by the marked Grothendieck construction, and so the identity 2-functor induces a map in $\mslicetwocat$
    \[\begin{tikzcd}
        (\int_\cC^+F,\mathrm{spl})\ar[rr,"\iota_F"]\ar[dr,"\pi^\mathrm{spl}"'] && \int_\cC^+ F\ar[ld,"\pi"]\\
        & \Csharp
    \end{tikzcd}\]  
\end{rmk}

\begin{defn}
We define a 2-natural transformation $\mu_F\colon L(\int_\cC^+ F,\mathrm{spl})\Rightarrow L\int_\cC^+ F$ by $\mu_F=L(\iota_F)$; that is, by applying the functor $L\colon \mslicetwocat\to 2\Fun[\cC^\coop,\ttwocat]$ to the map $\iota_F$ of \cref{rmk.iota}.
\end{defn}

\begin{lem}
    The 2-natural transformations $\mu_F\colon L(\int_\cC^+ F,\mathrm{spl})\Rightarrow L\int_\cC^+ F$ assemble into a natural transformation $\mu\colon L(\int_\cC^+ ,\mathrm{spl})\Rightarrow L\int_\cC^+$.
\end{lem}
\begin{proof}
   Given a 2-natural transformation $\theta\colon F\Rightarrow G$, we must check that the square below left commutes
   \[\begin{tikzcd}
       L(\int_\cC^+ F,\mathrm{spl})\rar["\mu_F", Rightarrow]\dar[Rightarrow,"L\int_\cC^+\theta"'] & L\int_\cC^+ F\dar["L\int_\cC^+\theta",Rightarrow]\\
       L(\int_\cC^+ G,\mathrm{spl})\rar["\mu_G", Rightarrow] & L\int_\cC^+ G
   \end{tikzcd}\qquad\qquad
   \begin{tikzcd}
       (\int_\cC^+ F,\mathrm{spl})\rar\dar["\int_\cC^+\theta"'] & \int_\cC^+ F\dar["\int_\cC^+\theta"]\\
       (\int_\cC^+ G,\mathrm{spl})\rar & \int_\cC^+ G
   \end{tikzcd}\] but this is ensured by the functoriality of $L$, as the square above right commutes, by definition.
\end{proof}

The fact that  $\mu\colon L(\int_\cC^+ ,\mathrm{spl})\Rightarrow L\int_\cC^+$ is a pointwise weak equivalence is not an artifact of the Grothendieck construction, but rather of the close relation between the two markings considered, as we now show. 

\begin{lem}\label{lem.split.localization}
    Let $P\colon\markedscat{\cP}\to\Csharp$ be a marked 2-functor, and $\cP^\mathrm{spl}=(\cP,X,E^2_\cP)$ denote a marked 2-category with the same underlying 2-category and same marked 2-cells, and with a subset of marked morphisms $X\subseteq E^1_\cP$ with the property that every marked morphism $u\colon x\to x'$ in $E^1_\cP$ can be factored as
    \[\begin{tikzcd}[cramped,sep=scriptsize]
    	x && {} && {x'} \\
    	\\
    	&& z
    	\arrow["u", from=1-1, to=1-5]
    	\arrow["{u_1}"',"\mathrel{\rotatebox[origin=c]{-30}{$\simeq$}}", from=1-1, to=3-3]
    	\arrow["{u_2}"', from=3-3, to=1-5]
    \end{tikzcd}\]
    where $u_1$ is an equivalence in $\cP$ and $u_2$ is a morphism in $X$. Then the marked 2-functor $\cP^\mathrm{spl}\to \markedscat{\cP}$ over $\Csharp$ induces pointwise biequivalences $L(P^\mathrm{spl})(c)\to L(P)(c)$ for every $c\in\cC$.
\end{lem}
\begin{proof}
    Note that $L(P^\mathrm{spl})(c)$ and $L(P)(c)$ have the same objects, and the same generating morphisms of the first type. The difference lies in the generating morphisms of the second type $(u^{-1},\widehat{P}\eta_u.f)\colon(x',Pu.f)\to (x,f)$ for a marked $u\colon x\to x'$ since there are fewer such morphisms in $L(P^\mathrm{spl})(c)$ than in $L(P)(c)$. It is enough to show that every such morphism in $L(P)(c)$ is related by an invertible 2-cell to some morphism in $L(P^\mathrm{spl})(c)$, as a quick inspection of the generating 2-cells reveals that any of these 2-cells in $L(P)(c)$ whose boundary is already in $L(P^\mathrm{spl})(c)$ must already be in $L(P^\mathrm{spl})(c)$.

    Consider then a morphism  $(u^{-1},\widehat{P}\eta_u.f)\colon(x',Pu.f)\to (x,f)$ for some $u\colon x\to x'$ in $E^1_\cP$. By assumption, we can factor $u=u_2u_1$ where $u_1$ is an equivalence in $\cP$ and $u_2$ is in $X$. If $(u_1,\underline{u_1}^{-1},\eta_{\underline{u_1}},\epsilon_{\underline{u_1}})$ denotes some adjoint equivalence data in $\cP$ for $u_1$, and $(u,u^{-1},\eta_u,\epsilon_u)$, $(u_2,u_2^{-1},\eta_{u_2},\epsilon_{u_2})$ denote the formal adjoint equivalence data for $u$ and $u_2$ in $\cP[E_\cP^{-1}]$,  respectively, we see that both $(u,u^{-1},\eta_u,\epsilon_u)$ and $(u,\underline{u_1}^{-1}u_2^{-1},(\underline{u_1}^{-1}\cdot \eta_{u_2}\cdot u_1)\eta_{\underline{u_1}},\epsilon_{u_2}(u_2\cdot \epsilon_{\underline{u_1}}\cdot u_2^{-1}))$ are adjoint equivalences data for $u$ in $\cP[E_\cP^{-1}]$. Hence there exists some invertible 2-cell $\gamma_u\colon u^{-1}\Rightarrow \underline{u_1}^{-1}u_2^{-1}$ in $\cP[E_\cP^{-1}]$ such that $(\gamma_u\cdot u)\eta_u=(\underline{u_1}^{-1}\cdot \eta_{u_2}\cdot u_1)\eta_{\underline{u_1}}$.

    The 2-cell $\gamma_u$ then gives an invertible 2-cell in $L(P)(c)$ as explained in \cref{gamma.u.exists}
    \[\begin{tikzcd}[cramped,sep=scriptsize]
    	{(x',Pu.f)} && {} && {(x,f)} \\
    	\\
    	&& {(z,Pu_1.f)}
    	\arrow["{(u^{-1},\widehat{P}\eta_u.f)}", from=1-1, to=1-5]
    	\arrow["{(u_2^{-1}, \widehat{P^{\mathrm{spl}}}\eta_{u_2}\cdot(Pu_1.f))}"',"\mathrel{\rotatebox[origin=c]{-30}{$\simeq$}}", from=1-1, to=3-3]
    	\arrow["{\gamma_u}"',"\cong", between={0.3}{0.7}, Rightarrow, from=1-3, to=3-3]
    	\arrow["{(\underline{u_1}^{-1},P\eta_{\underline{u_1}}\cdot f)}"', from=3-3, to=1-5]
    \end{tikzcd}\]
    where both of the maps along the bottom are morphisms in $L(P^\mathrm{spl})(c)$; this concludes our proof.
\end{proof}

As a corollary, we obtain the corresponding result for $\mu$.

\begin{cor}\label{split.counit.weak.equiv}
The natural transformation $\mu\colon L(\int_\cC^+ -,\mathrm{spl})\Rightarrow L\int_\cC^+$ is pointwise a weak equivalence.   
\end{cor}
\begin{proof}
    This is a consequence of \cref{lem.split.localization} as any marked morphism $(f,\widehat{f})$ in the marked Grothendieck construction $\int_\cC^+ F$ is such that $\widehat{f}$ is an equivalence and can be factored as
    \[\begin{tikzcd}[cramped,sep=scriptsize]
    	{(c,x)} && {} && {(c',x')} \\
    	\\
    	&& {(c,Ff(x'))}
    	\arrow["{(f,\widehat{f})}", from=1-1, to=1-5]
    	\arrow["{(\id_c,f)}"',"\mathrel{\rotatebox[origin=c]{-30}{$\simeq$}}", from=1-1, to=3-3]
    	\arrow["{(f,\id)}"', from=3-3, to=1-5]
    \end{tikzcd}\]
\end{proof}

Next, we define a natural transformation $\epsilon\colon L(\int_\cC^+  -,\mathrm{spl})\Rightarrow\id$.

\begin{constr}\label{constr.counit.split}
    Given a 2-functor $F\colon\cC^\coop\to\ttwocat$ and an object $d\in\cC$, we define a 2-functor $(\epsilon_F)_d\colon L(\int_\cC^+ F,\mathrm{spl})(d)\to Fd$ as follows:
    \begin{itemize}[leftmargin=*]
        \item On objects, it is given by $((c,x),f\colon d\to c)\mapsto Ff(x)$.
        \item It acts on each of the generating morphisms as follows:
        \begin{itemize}
            \item $((k,\widehat{k}),\varphi)\colon ((c,x),f)\to ((c',x'),f')\mapsto Ff(x)\xrightarrow{Ff(\widehat{k})} FfFk(x')\xrightarrow{(F\varphi)_{x'}} Ff'(x')$ for some morphism $(k,\widehat{k})$ in $\int_\cC^+ F$ and some 2-cell $\varphi\colon f'\Rightarrow kf$. 
            
            In particular, it maps a morphism $((k,\id),\id)\colon ((c,Fk(x')), f)\to ((c',x'),kf)$ for $(k,\id)$ marked to $FfFk(x')\xrightarrow{\id} F(kf)(x')$.  
            \item A generator $((k,\id),\id)^{-1}$ which is an inverse equivalence to $((k,\id),\id)$ is mapped to $F(kf)(x')\xrightarrow{\id} FfFk(x')$.
        \end{itemize}
        \item It acts on each of the generating 2-cells as follows:
        \begin{itemize}
            \item A 2-cell $(\sigma, \widehat{\sigma})\colon ((k,\widehat{k}),\varphi) \Rightarrow ((j,\widehat{j}),\psi)$ for $(\sigma,\widehat{\sigma})\colon (k,\widehat{k})\Rightarrow (j,\widehat{j})$ in $\int_\cC^+ F$ and $\varphi,\psi$ in $\cC$ is mapped to the 2-cell $(F\varphi)_{x'}\cdot Ff(\widehat{\sigma})$ depicted in
	\[\begin{tikzcd}[ampersand replacement=\&,
    column sep=6pt,
    row sep=1.8ex]
	\&\&\& {Ff.Fk(x')} \&\&\& \\
	\&\& {} \\
	{Ff(x)} \& {} \&\& {} \&\& {} \& {Ff'(x')} \\
	\&\& {} \\
	\&\&\& {Ff.Fj(x')}
	\arrow["{(F\varphi)_{x'}}", from=1-4, to=3-7]
	\arrow["{Ff(\widehat{\sigma})}"', between={0.3}{0.7}, shift left= 3,Rightarrow, from=2-3, to=4-3]
	\arrow["{Ff(\widehat{k})}", from=3-1, to=1-4]
	\arrow["{Ff(\widehat{j})}"', from=3-1, to=5-4]
	\arrow["{\scriptstyle{=}}"{pos=0.8}, draw=none, from=3-4, to=3-6]
	\arrow["{(Ff\cdot F\sigma)_{x'}}"{description}, from=5-4, to=1-4]
	\arrow["{(F\psi)_{x'}}"', from=5-4, to=3-7]
	\end{tikzcd}\]
	where the right triangle commutes since $(\sigma \cdot f).\varphi = \psi$. 
    \item A 2-cell $(\sigma,\widehat{\sigma})^{-1}$ which is the inverse of some marked $(\sigma,\widehat{\sigma})$ (i.e.\ such that $\widehat{\sigma}$ is invertible) is mapped to $(F\psi)_{x'}\cdot Ff(\widehat{\sigma}^{-1})$, the inverse of $(\epsilon_F)_d(\sigma,\widehat{\sigma})$.
\item The 2-cells corresponding to formal units and counits are mapped to identities.
        \end{itemize}
    \end{itemize}

    Note that no choices are involved in the construction, and so the 2-functoriality of $(\epsilon_F)_d$ is a consequence of that of $F$. 
\end{constr}

\begin{lem}
    The 2-functors $(\epsilon_F)_d$ of \cref{constr.counit.split} assemble into a 2-natural transformation $\epsilon_F\colon L(\int_\cC^+ F,\mathrm{spl})\Rightarrow F$.
\end{lem}
\begin{proof}
    This is a tedious but straightforward verification, which follows from the 2-functoriality of $F$.
\end{proof}

\begin{lem}
    The 2-natural transformations $\epsilon_F$ assemble into a natural transformation $\epsilon\colon L(\int_\cC^+ -,\mathrm{spl})\Rightarrow \id$.
\end{lem}
\begin{proof}
    Given a 2-natural transformation $\theta\colon F\Rightarrow G$, we must check that the following square of 2-functors commutes for each $d\in\cC$:
    \[\begin{tikzcd}
        L(\int_\cC^+ F,\mathrm{spl})(d)\rar["{(\epsilon_F)_d}"] \dar["{(L\int_\cC^+\theta)_d}"'] & Fd\dar["\theta_d"]\\
        L(\int_\cC^+ G,\mathrm{spl})(d)\rar["{(\epsilon_G)_d}"] & Gd
    \end{tikzcd}\] where we recall that the left vertical arrow is obtained by applying $L$ to the 2-functor $\int_\cC^+\theta$ given by $(c,x)\mapsto (c,\theta_d(x))$, $(k,\widehat{k})\mapsto (k,\theta_d(\widehat{k}))$ and $(\alpha,\widehat{\alpha})\mapsto (\alpha,\theta_d(\widehat{\alpha}))$. Unpacking the definitions, one can see that the fact that this commutes is a direct consequence of the 2-functoriality of $\theta_d$ (by which identities are mapped to identities) and the 2-naturality of $\theta$. 
\end{proof}

\begin{prop}\label{counit.is.weak.equiv}
    The natural transformation $\epsilon\colon L(\int_\cC^+ -,\mathrm{spl})\Rightarrow \id$ is pointwise a weak equivalence.
\end{prop}
\begin{proof}
    Let $F\colon\cC^\coop\to\ttwocat$ be a 2-functor; we must show that the 2-natural transformation $\epsilon_F$ is a weak equivalence. Since $\int_\cC^+$ reflects weak equivalences, by \cref{GCpreservesandreflectsweakequivs}, it suffices to show that $\int_\cC^+\epsilon_F$ is a weak equivalence. For this, consider the following diagram:
    \[\begin{tikzcd}[cramped,sep=scriptsize]
    	{\int_{\cC}^{+}F} && {} && {\int_{\cC}^{+}F} \\
    	&& {\int_{\cC}^{+}L\left(\int_{\cC}^{+}F,\mathrm{spl}\right)} \\
    	\\
    	&& {\int_{\cC}^{+}L\int_{\cC}^{+}F}
    	\arrow[equals, from=1-1, to=1-5]
    	\arrow["{\overline{\eta_{\pi}}}"', from=1-1, to=2-3]
    	\arrow["{\eta_{\pi}}"',"\mathrel{\rotatebox[origin=c]{-40}{$\simeq$}}", curve={height=15pt}, from=1-1, to=4-3]
    	\arrow["\int_{\cC}^{+}\epsilon_F"', from=2-3, to=1-5]
    	\arrow["{\int_{\cC}^{+}\mu_F}","\mathrel{\rotatebox[origin=c]{90}{$\simeq$}}"', from=2-3, to=4-3]
    \end{tikzcd}\]
    The map $\overline{\eta_\pi}\colon\int_\cC^+ F\to \int_\cC^+L(\int_\cC^+F,\mathrm{spl})$ is defined identically to $\eta_\pi$, which we can do as the image of the map $\eta_\pi\colon\int_\cC^+ F\to \int_\cC^+L\int_\cC^+F$ of \cref{constr.unit} lies in $\int_\cC^+L(\int_\cC^+F,\mathrm{spl})$. Moreover, one can see that the left triangle commutes, as the map $\int_\cC^+\mu_F$ acts as an inclusion on the data which is already present in $\int_\cC^+ F$. The top triangle commutes, as the composite acts by 
\[
		\begin{aligned}
		(c,x) &\mapsto (c, ((c,x),\id_{c}))\mapsto (c,F\id_c(x))=(c,x) && \text{on objects,}\\
		(f,\widehat{f}) &\mapsto (f,((f,\widehat{f}),\id))\mapsto (f,F\id\cdot F\id(\widehat{f}))=(f,\widehat{f}) && \text{on morphisms,}\\
		(\alpha,\widehat{\alpha}) &\mapsto (\alpha, (\alpha,\widehat{\alpha}))\mapsto (\alpha,\widehat{\alpha}) && \text{on $2$-cells.}
		\end{aligned}
	\] 
    
    Now, the map $\int_\cC^+ \mu_F$ is a weak equivalence since $\mu_F$ is a weak equivalence, by \cref{split.counit.weak.equiv}, and $\int_\cC^+$ preserves weak equivalences, by \cref{GCpreservesandreflectsweakequivs}. Furthermore, the map $\eta_\pi$ is also a weak equivalence as proved in \cref{Lem:unitiswequiv}. Hence, $\overline{\eta_\pi}$ is a weak equivalence, by 2-out-of-3, and thus, by the same reason, so is $\int_\cC^+ \epsilon_F$. 
\end{proof}

\subsection{The equivalence on homotopy categories}

We now  combine our results to show that $\int_\cC^+$ gives an equivalence on homotopy categories. As established in \cref{GCpreservesandreflectsweakequivs}, the functor $\int_\cC^+$ preserves all weak equivalences, and so it is its own derived functor. In order to define a derived functor for $L$, we start by studying its behaviour with respect to weak equivalences.

\begin{prop}\label{L.pres.weak.equiv}
    The functor $L\colon \mslicetwocat\to 2\Fun[\cC^\coop,\ttwocat]$ preserves weak equivalences between fibrant objects.
\end{prop}
\begin{proof}
    Let $H\colon P\to Q$ be a weak equivalence between fibrant objects in $\mslicetwocat$, and consider the diagram
    \[\begin{tikzcd}[cramped,sep=scriptsize]
    	\cP && \cQ \\
    	\\
    	{\int_{\cC}^{+}L(P)} && {\int_{\cC}^{+}L(Q)}
    	\arrow["H","\simeq"', from=1-1, to=1-3]
    	\arrow["{\eta_P}"',"\mathrel{\rotatebox[origin=c]{-90}{$\simeq$}}", from=1-1, to=3-1]
    	\arrow["{\eta_Q}","\mathrel{\rotatebox[origin=c]{90}{$\simeq$}}"', from=1-3, to=3-3]
    	\arrow["{\int_{\cC}^{+}L(H)}"', from=3-1, to=3-3]
    \end{tikzcd}\]
     which commutes, by naturality of $\eta$; see \cref{eta.is.natural}. Since $P$ and $Q$ are fibrant, we know, by \cref{Lem:unitiswequiv}, that the vertical maps are weak equivalences. Hence, by 2-out-of-3, we see that $\int_\cC^+ L(H)$ is a weak equivalence, and as the functor $\int_\cC^+$ reflects weak equivalences, by \cref{GCpreservesandreflectsweakequivs}, we conclude that $L(H)$ must be a weak equivalence, as desired.
\end{proof}

    We can finally prove the main result of this section.

    \begin{thm}\label{mainthm}
         The marked Grothendieck construction 
   ${\textstyle \int}_{\cC}^+ \colon 2\Fun[\cC^{\coop},\ttwocat] \to \mslicetwocat$
   induces an equivalence between homotopy categories.
    \end{thm}
    \begin{proof}
  Define the inverse functor $\mathbb{L}\colon \Ho\mslicetwocat\to \Ho 2\Fun[\cC^{\coop},\ttwocat]$ by $\mathbb{L}(P)=L(P^\mathrm{fib})$, where $P^\mathrm{fib}$ is any fibrant replacement of $P$ in $\mslicetwocat$. The fact that $L$ preserves weak equivalences between fibrant objects as established in \cref{L.pres.weak.equiv} ensures that $\mathbb{L}$ is a well-defined functor on homotopy categories.

  The natural transformation $\eta\colon \id\Rightarrow \int_\cC^+ L$ gives a natural transformation $\id\Rightarrow{\textstyle \int}_\cC^+\mathbb{L}$, whose components are given by \[P\xrightarrow{\sim} P^\mathrm{fib}\xrightarrow{\eta_{P^\mathrm{fib}}} {\textstyle \int}_\cC^+ L(P^\mathrm{fib})={\textstyle \int}_\cC^+\mathbb{L}(P).\] Since each $\eta_{P^\mathrm{fib}}$ is a weak equivalence, by \cref{Lem:unitiswequiv}, we see that this is a natural isomorphism between functors on homotopy categories. Similarly, the zig-zag of natural transformations \[\id \Leftarrow L({\textstyle \int}_\cC^+ -,\mathrm{spl})\Rightarrow L{\textstyle \int}_\cC^+\] which are pointwise weak equivalences, by \cref{split.counit.weak.equiv,counit.is.weak.equiv}, determines a natural isomorphism $\id\Rightarrow L{\textstyle \int}_\cC^+=\mathbb{L}{\textstyle \int}_\cC^+$ between functors on homotopy categories. 

  This shows that there are natural isomorphisms $\id\cong\int_\cC^+\mathbb{L}$ and $\id\cong \mathbb{L}\int_\cC^+$, which proves that the functors on homotopy categories $\int_\cC^+$ and $\mathbb{L}$ are inverse equivalences.
    \end{proof}

In fact, the marked Grothendieck construction establishes not only an equivalence of homotopy categories, but an equivalence on the underlying $(\infty, 1)$-categories of these model categories.
This result can be interpreted as recovering an equivalence of weak $(3, 1)$-categories akin to that underlying Buckley's triequivalence of tricategories.

\begin{cor}
\label{eq.infty}
The functor 
   $${\textstyle \int}_{\cC}^+ \colon 2\Fun[\cC^{\coop},\ttwocat] \to  \mslicetwocat$$
   induces an equivalence on underlying $(\infty,1)$-categories.
    % The composite functor     
    %  \[         
    %  \mathbb{L}\colon\left(\mslicetwocat\right)_{\textrm{fibrant}}
    %         \hookrightarrow
    %     \mslicetwocat
    %         \xrightarrow{L}
    %     2\Fun[\cC^{\coop},\ttwocat]
    % \]
    % is an equivalence on underlying $(\infty,1)$-categories (of the respective relative categories).
\end{cor}
\begin{proof}
First note that the image of $\int_\cC^+$ lies in the subcategory $\left(\mslicetwocat\right)_{\textrm{fib}}\subset \mslicetwocat$ consisting of the fibrant objects, by \cref{cor:projectionisfibrant}. We claim that the restriction of the functor $L$ to this subcategory, which we denote by $L_\textrm{fib}\colon \left(\mslicetwocat\right)_{\textrm{fib}}\to 2\Fun[\cC^{\coop},\ttwocat]$, gives an inverse equivalence on underlying $(\infty,1)$-categories. This would conclude our proof, as the underlying $(\infty, 1)$-categories of any model structure $\mathcal{M}$ and of its full relative subcategory of fibrant objects $\mathcal{M}_\textrm{fib}$ are equivalent; see for instance \cite[Proposition 5.2]{dwyerkan}. 

    By \cref{GCpreservesandreflectsweakequivs,L.pres.weak.equiv} we have that the functors $\int_\cC^+$ and $L_\textrm{fib}$ preserve weak equivalences and thus give rise to functors between the underlying $(\infty,1)$-categories $(\int_\cC^+)_\infty$ and $(L_\textrm{fib})_\infty$.  Moreover, the natural transformation $\eta\colon \id\Rightarrow \int_\cC^+ L_\textrm{fib}$ is a pointwise weak equivalence  by \cref{Lem:unitiswequiv}, and so it induces a natural isomorphism $\id\cong (\int_\cC^+)_\infty (L_\textrm{fib})_\infty$. Similarly, the zig-zag of natural transformations $\id \Leftarrow L({\textstyle \int}_\cC^+ -,\mathrm{spl})\Rightarrow L{\textstyle \int}_\cC^+=L_\textrm{fib}{\textstyle \int}_\cC^+$ consists of pointwise weak equivalences by \cref{split.counit.weak.equiv,counit.is.weak.equiv}. %These transformations restrict to transformations on $\mathbb L$, as $\int_\cC^+$ lands on fibrant objects (\cref{cor:projectionisfibrant}). {\color{red}By the lemma,} the functor $L({\textstyle \int}_\cC^+ -,\mathrm{spl})$ is also relative when restricted to fibrant objects, and hence, by \cite[Corollary 2.9]{gurski17}, the composite functor in the equation above is an equivalence on the underlying $(\infty,1)$-categories (there modeled as the complete Segal spaces of their respective relative categories).
    To show that this zig-zag induces a natural isomorphism $\id\cong (L_\textrm{fib})_\infty (\int_\cC^+)_\infty$, it suffices to prove that the functor $L({\textstyle \int}_\cC^+ -,\mathrm{spl})$ also preserves all weak equivalences, and hence gives a functor on underlying $(\infty,1)$-categories. For this, let $\theta\colon F\Rightarrow G$ be a weak equivalence; we can then consider the naturality square
     \[\begin{tikzcd}
        L(\int_\cC^+ F,\mathrm{spl})\rar["{\epsilon_F}", Rightarrow] \dar["{L\int_\cC^+\theta}"', Rightarrow] & F\dar["\theta", Rightarrow]\\
        L(\int_\cC^+ G,\mathrm{spl})\rar["{\epsilon_G}", Rightarrow] & G
    \end{tikzcd}\] in which $\theta,\epsilon_F$ and $\epsilon_G$ are weak equivalences, and so by 2-out-of-3 we have that $L\int_\cC^+\theta$ is a weak equivalence as desired.
\end{proof}

\printbibliography

\end{document}

%% file: NEWPREAMBLE.tex
\usepackage{sidecap}
\usepackage{quiver}

\usepackage{xcolor}
\definecolor{burgundy}{rgb}{0.5, 0.0, 0.13}

\usepackage{verbatim}
\usepackage{lmodern}
\usepackage{amsmath}
\usepackage{amssymb} 
\usepackage{amsthm} 
\usepackage{thmtools}
\usepackage{enumitem}
\usepackage{graphicx}
\usepackage[normalem]{ulem}
\usepackage[hidelinks]{hyperref}%hidelinks avoid the red boxes around the hyperlinks
\hypersetup{
  colorlinks = true,
  linktoc = all,
  linkcolor = burgundy, 
  urlcolor = Gray,
  citecolor = burgundy,
  pdfencoding = unicode,
  hypertexnames = false,
}

\usepackage{xurl}                   % let URLs break at any character
\usepackage[ 
style           	= alphabetic, 
sorting 			= nyt,
url					= false,
citetracker     	= true, 
natbib          	= true, 
hyperref      		= true,
backend      		= biber,
maxbibnames         = 99,
maxcitenames        = 2,
giveninits          = true,
uniquename    		= init,
parentracker  	  	= true,
backref        	 	= true,
backrefstyle    	=three]{biblatex}
\AtBeginBibliography{\Urlmuskip=0mu\relax} % no stretched gaps inside broken URLs
\renewcommand*{\bibnamedash}{\underline{\hspace{3em}}\kern 0.1em} 
\DeclareFieldFormat{postnote}{#1}
\DeclareFieldFormat[article,inbook,incollection,inproceedings,patent,thesis,unpublished]{title}{\textit{#1\isdot}} 
\DeclareFieldFormat{journaltitle}{#1}  % Makes Journal title in normal text
\DeclareFieldFormat[book,inbook,incollection,inproceedings]{series}{#1} 

\renewbibmacro*{in:}{\setunit{\addcomma\space}}

\usepackage{mathtools}
\usepackage{indentfirst}
\usepackage{tabularx}
\usepackage{bbm}
\usepackage[capitalise]{cleveref}
\usepackage{stmaryrd}
\usepackage{tabularx}
\usepackage{mathrsfs}

\usepackage[left=1.1in,top=2cm,right=1.1in,bottom=2cm]{geometry}

\makeatletter
\newcommand{\mylabel}[2]{#2\def\@currentlabel{#2}\label{#1}}
\makeatother

\usepackage{array}
\newcolumntype{C}[1]{>{\centering\let\newline\\\arraybackslash\hspace{0pt}}m{#1}}

\usepackage{tikz} %pictures
\usepackage{tikz-cd} %diagrams
\newcommand{\btk}{\begin{tikzcd}}
\newcommand{\etk}{\end{tikzcd}}
\usetikzlibrary{arrows,calc,decorations.markings,bending}
\usepackage{xparse}

\usepackage{blindtext}

\makeatletter
\newcommand{\@bbify}[1]{
  \ifcsname b#1\endcsname
  \message{WARNING: Overwriting b#1 with blackboard letter!}
  \fi
  \expandafter\edef\csname b#1\endcsname
  {\noexpand\ensuremath{\noexpand\mathbb #1}\noexpand\xspace}}
\newcommand{\@calify}[1]{
  \ifcsname c#1\endcsname
  \message{WARNING: Overwriting c#1 with calligraphic letter!}
  \fi 
  \expandafter\edef\csname c#1\endcsname
  {\noexpand\ensuremath{\noexpand\mathcal #1}\noexpand\xspace}}
\newcommand{\@bfify}[1]{
  \ifcsname bf#1\endcsname
  \message{WARNING: Overwriting c#1 with bold letter!}
  \fi
  \expandafter\edef\csname bf#1\endcsname
  {\noexpand\ensuremath{\noexpand\mathbf #1}\noexpand\xspace}}
\newcounter{@letter}\stepcounter{@letter}
\loop\@bbify{\Alph{@letter}}\@calify{\Alph{@letter}}\@bfify{\Alph{@letter}}
\ifnum\the@letter<26\stepcounter{@letter}\repeat
\makeatother

\tikzstyle{d}=[double distance=.3ex]
\tikzstyle{w}=[preaction={draw=white,-,line width=5pt}]

\newcounter{diagram}
\renewcommand{\thediagram}{\thetheorem}

\tikzset{%
node distance=1.5cm, la/.style={scale=0.8}, lasmall/.style={scale=0.75}, over/.style={auto=false,fill=white,inner sep=1.5pt, minimum size=0, outer sep=0},
    symbol/.style={%
        draw=none,
        every to/.append style={%
            edge node={node [sloped, allow upside down, auto=false]{$#1$}}},
            
    }, pro/.style={postaction={decorate,decoration={
        markings,
        mark=at position .5 with {\node at (0,0) {$\bullet$};}
      }},
      inner sep=.9ex,
      },
      prosmall/.style={postaction={decorate,decoration={
        markings,
        mark=at position .5 with {\node at (0,0) {$\scriptstyle \bullet$};}
      }},
      inner sep=.9ex,
      },
  n/.style={double equal sign distance, -implies}, t/.style={double distance=2.5pt, -implies, postaction={draw,-}},
}

\newcommand{\Set}{\mathrm{Set}}

\newcommand{\an}{\ensuremath{\mathrm{An}}}

\newcommand{\nfib}{\ensuremath{\mathrm{NFib}}}

\newcommand{\Path}{\mathrm{Path}}
\newcommand{\id}{\mathrm{id}}

\newcommand{\op}{\mathrm{op}}

\newcommand{\cat}{\mathrm{Cat}}

\newcommand{\twocat}{2\mathrm{Cat}}
\newcommand{\ttwocat}{\mathbf{2}\cat}

\newlist{rome}{enumerate}{7}
\setlist[rome]{label=(\roman*)}

\newtheorem{thm}{Theorem}[section]
\newtheorem{cor}[thm]{Corollary}
\newtheorem{prop}[thm]{Proposition}
\newtheorem{lem}[thm]{Lemma}

\declaretheorem[name=Theorem,numbered=yes]{theoremA}

\numberwithin{equation}{thm}
\theoremstyle{definition}
\newtheorem{defn}[thm]{Definition}

\newtheorem{ex}[thm]{Example}
\newtheorem{notation}[thm]{Notation}

\newtheorem{constr}[thm]{Construction}

\theoremstyle{remark}
\newtheorem{rmk}[thm]{Remark}

\crefname{section}{Section}{Sections}
\crefname{thm}{Theorem}{Theorems}
\crefname{cor}{Corollary}{Corollaries}
\crefname{prop}{Proposition}{Propositions}
\crefname{lem}{Lemma}{Lemmata}
\crefname{claim}{Claim}{Claims}
\crefname{defn}{Definition}{Definitions}
\crefname{terminology}{Terminology}{Terminologies}
\crefname{ex}{Example}{Examples}
\crefname{notation}{Notation}{Notations}
\crefname{descr}{Description}{Descriptions}
\crefname{constr}{Construction}{Constructions}
\crefname{rmk}{Remark}{Remarks}
\crefname{obs}{Observation}{Observations}
\crefname{paragr}{Paragraph}{Paragraphs}

\usepackage[
  textwidth = 2.5cm,
  textsize = small, 
  colorinlistoftodos
]{todonotes}

\DeclareMathOperator{\Ho}{Ho}

\newcommand{\bicat}{\mathrm{Bicat}}

\newcommand{\Cat}{\mathrm{Cat}}

\newcommand{\Csharp}{\cC^{\sharp}}
\newcommand{\mslicetwocat}{\twocat^{+}_{/ \Csharp}}

\newcommand{\markedscat}[1]{\left({#1}, E^1_{#1},E^2_{#1}\right)}

\DeclareMathOperator{\Fun}{Fun}

\DeclareMathOperator{\interval}{\mathbb{I}}

\newcommand{\coop}{\mathrm{coop}}

\newcommand{\set}[1]{\left\lbrace #1 \right\rbrace }

\newcommand{\stlabel}[1]{{\upshape{(\ref*{#1}.\arabic*)}}}
\newcommand{\enumref}[2]{(\hyperref[#1.#2]{\ref*{#1}.\ref*{#1.#2}})}
\newcommand{\enumrefalt}[2]{\hyperref[#1.#2]{\ref*{#1}.\ref*{#1.#2}}}

\newcommand{\walkingadjoint}{E_{\mathrm{adj}}}

\NewDocumentCommand{\celli}{ O{} O{n} O{\cellslide} O{\celllength} m m m }{
  \coordinate (mid) at ($({#5})!{#3}!({#6})$);
  \coordinate (start) at ($(mid)!{#4}!({#5})$);
  \coordinate (end) at ($(mid)!{#4}!({#6})$);
  \draw[#2] (start) to node(label)[inner sep=4pt,outer sep=0,minimum size=0,#1]{{#7}} (end);
   \coordinate (far) at ($(end)+(mid)-(label)$);
   \node[] at ($(end)!7pt!(far)$) {$\scriptscriptstyle\cong$} ;
}

\def\cellslide{0.5}
\def\celllength{.2cm}

\NewDocumentCommand{\cell}{ O{} O{n} O{\cellslide} O{\celllength} m m m }{
  \coordinate (mid) at ($({#5})!{#3}!({#6})$);
  \coordinate (start) at ($(mid)!{#4}!({#5})$);
  \coordinate (end) at ($(mid)!{#4}!({#6})$);
  \draw[#2] (start) to node
  [inner sep=6pt,outer sep=0,minimum size=0,#1]{{#7}} (end);
}

\DeclareMathOperator{\Loc}{Loc}

%% file: references.bib
@string{C = {Carbon}}

@string{F = {Fuel}}

@string{I = {Ind. Eng. Chem. Res.}}

@phdthesis{maillard:tel-03475256,
  TITLE = {{Finite group representations through 2-sheaves}},
  AUTHOR = {Maillard, Jun},
  URL = {https://theses.hal.science/tel-03475256},
  NUMBER = {2021LILUI038},
  SCHOOL = {{Universit{\'e} de Lille}},
  YEAR = {2021},
  MONTH = Jun,
  TYPE = {Theses},
  HAL_ID = {tel-03475256},
  HAL_VERSION = {v1},
}

@incollection {vistoli2007notesgrothendiecktopologiesfibered,
    AUTHOR = {Vistoli, Angelo},
     TITLE = {Grothendieck topologies, fibered categories and descent
              theory},
 BOOKTITLE = {Fundamental algebraic geometry},
    SERIES = {Math. Surveys Monogr.},
    VOLUME = {123},
     PAGES = {1--104},
 PUBLISHER = {Amer. Math. Soc., Providence, RI},
      YEAR = {2005},
      ISBN = {0-8218-3541-6},
   MRCLASS = {14F20 (14A20)},
  MRNUMBER = {2223406},
}

@article {Lack2004,
	AUTHOR = {Lack, Stephen},
	TITLE = {A {Q}uillen model structure for bicategories},
	JOURNAL = {$K$-Theory},
	FJOURNAL = {$K$-Theory. An Interdisciplinary Journal for the Development,
	Application, and Influence of $K$-Theory in the Mathematical
	Sciences},
	VOLUME = {33},
	YEAR = {2004},
	NUMBER = {3},
	PAGES = {185--197},
	ISSN = {0920-3036,1573-0514},
	MRCLASS = {55U35 (18D05 18D15 18G55)},
	MRNUMBER = {2138540},
	MRREVIEWER = {J\v{i}r\'{\i}\ Rosick\'{y}},
	DOI = {10.1007/s10977-004-6757-9},
	URL = {https://doi.org/10.1007/s10977-004-6757-9},
}

@article {Lack2002,
	AUTHOR = {Lack, Stephen},
	TITLE = {A {Q}uillen model structure for 2-categories},
	JOURNAL = {$K$-Theory},
	FJOURNAL = {$K$-Theory. An Interdisciplinary Journal for the Development,
	Application, and Influence of $K$-Theory in the Mathematical
	Sciences},
	VOLUME = {26},
	YEAR = {2002},
	NUMBER = {2},
	PAGES = {171--205},
	ISSN = {0920-3036,1573-0514},
	MRCLASS = {55U35 (18D05 18G55)},
	MRNUMBER = {1931220},
	MRREVIEWER = {J.\ Daniel\ Christensen},
	DOI = {10.1023/A:1020305604826},
	URL = {https://doi.org/10.1023/A:1020305604826},
}

@article {Buckley2014,
	AUTHOR = {Buckley, Mitchell},
	TITLE = {Fibred 2-categories and bicategories},
	JOURNAL = {J. Pure Appl. Algebra},
	FJOURNAL = {Journal of Pure and Applied Algebra},
	VOLUME = {218},
	YEAR = {2014},
	NUMBER = {6},
	PAGES = {1034--1074},
	ISSN = {0022-4049,1873-1376},
	MRCLASS = {18D30 (18D05)},
	MRNUMBER = {3153612},
	MRREVIEWER = {R.\ H.\ Street},
	DOI = {10.1016/j.jpaa.2013.11.002},
	URL = {https://doi.org/10.1016/j.jpaa.2013.11.002},
}

@article {mosersarazola2023model,
    AUTHOR = {Moser, Lyne and Sarazola, Maru},
     TITLE = {A model structure for {G}rothendieck fibrations},
   JOURNAL = {J. Pure Appl. Algebra},
  FJOURNAL = {Journal of Pure and Applied Algebra},
    VOLUME = {228},
      YEAR = {2024},
    NUMBER = {10},
     PAGES = {Paper No. 107692, 29},
      ISSN = {0022-4049,1873-1376},
   MRCLASS = {18D30 (18A25 18N40 18N60)},
  MRNUMBER = {4741462},
MRREVIEWER = {Philippe\ Gaucher},
       DOI = {10.1016/j.jpaa.2024.107692},
       URL = {https://doi-org.proxybmath.univ-lille.fr/10.1016/j.jpaa.2024.107692},
}

@book {Niles2021,
	AUTHOR = {Johnson, Niles and Yau, Donald},
	TITLE = {2-dimensional categories},
	PUBLISHER = {Oxford University Press, Oxford},
	YEAR = {2021},
	PAGES = {xix+615},
	ISBN = {978-0-19-887138-5; 978-0-19-887137-8},
	MRCLASS = {18-02 (18N10 18N15 18N20)},
	MRNUMBER = {4261588},
	MRREVIEWER = {Robert\ Laugwitz},
	DOI = {10.1093/oso/9780198871378.001.0001},
	URL = {https://doi.org/10.1093/oso/9780198871378.001.0001},
}

@misc{guetta2023fibrantlyinduced,
	title={Fibrantly-transferred model structures}, 
	author={Léonard Guetta and Lyne Moser and Maru Sarazola and Paula Verdugo},
	year={2023},
	eprint={2301.07801},
	archivePrefix={arXiv},
	primaryClass={math.AT},
    note={To appear in Transactions of the AMS.}
}

@book {hovey,
    AUTHOR = {Hovey, Mark},
     TITLE = {Model categories},
    SERIES = {Mathematical Surveys and Monographs},
    VOLUME = {63},
 PUBLISHER = {American Mathematical Society, Providence, RI},
      YEAR = {1999},
     PAGES = {xii+209},
      ISBN = {0-8218-1359-5},
   MRCLASS = {55U35 (18D15 18G30 18G55)},
  MRNUMBER = {1650134},
MRREVIEWER = {Teimuraz\ Pirashvili},
}

@article{KellyLack_loc_presentable,
  title={{V}-Cat is locally presentable or locally bounded if {V} is so},
  author={G. M. Kelly and Stephen Lack},
  journal={Theory and Applications of Categories},
  year={2001},
  volume={8},
  pages={555-575},
  url={https://api.semanticscholar.org/CorpusID:18917884}
}

@article {Moser_2019,
    AUTHOR = {Moser, Lyne},
     TITLE = {Injective and projective model structures on enriched diagram
              categories},
   JOURNAL = {Homology Homotopy Appl.},
  FJOURNAL = {Homology, Homotopy and Applications},
    VOLUME = {21},
      YEAR = {2019},
    NUMBER = {2},
     PAGES = {279--300},
      ISSN = {1532-0073,1532-0081},
   MRCLASS = {18G55 (18D20 55U35)},
  MRNUMBER = {3923784},
MRREVIEWER = {Philippe\ Gaucher},
       DOI = {10.4310/HHA.2019.v21.n2.a15},
       URL = {https://doi-org.proxybmath.univ-lille.fr/10.4310/HHA.2019.v21.n2.a15},
}

@article {MR1624638,
    AUTHOR = {Borceux, F. and Quinteiro, C. and Rosick\'{y}, J.},
     TITLE = {A theory of enriched sketches},
   JOURNAL = {Theory Appl. Categ.},
  FJOURNAL = {Theory and Applications of Categories},
    VOLUME = {4},
      YEAR = {1998},
     PAGES = {No. 3, 47--72},
      ISSN = {1201-561X},
   MRCLASS = {18C10 (18D20)},
  MRNUMBER = {1624638},
MRREVIEWER = {R.\ H.\ Street},
}

@book {Cisinski-Deglise2019,
    AUTHOR = {Cisinski, Denis-Charles and D\'eglise, Fr\'ed\'eric},
     TITLE = {Triangulated categories of mixed motives},
    SERIES = {Springer Monographs in Mathematics},
 PUBLISHER = {Springer, Cham},
      YEAR = {2019},
     PAGES = {xlii+406},
      ISBN = {978-3-030-33241-9; 978-3-030-33242-6},
   MRCLASS = {14F42 (14C15 14C35 18G80 19D55)},
  MRNUMBER = {3971240},
MRREVIEWER = {Igor\ A.\ Rapinchuk},
       DOI = {10.1007/978-3-030-33242-6},
       URL = {https://doi-org.proxybmath.univ-lille.fr/10.1007/978-3-030-33242-6},
}

@article {pronk1996etendues,
    AUTHOR = {Pronk, Dorette A.},
     TITLE = {Etendues and stacks as bicategories of fractions},
   JOURNAL = {Compositio Math.},
  FJOURNAL = {Compositio Mathematica},
    VOLUME = {102},
      YEAR = {1996},
    NUMBER = {3},
     PAGES = {243--303},
      ISSN = {0010-437X,1570-5846},
   MRCLASS = {18E35 (18D05)},
  MRNUMBER = {1401424},
MRREVIEWER = {A.\ Verschoren},
       URL = {http://www.numdam.org.proxybmath.univ-lille.fr/item?id=CM_1996__102_3_243_0},
}

@article{muro2014dwyer,
  title={Dwyer--Kan homotopy theory of enriched categories},
  author={Muro, Fernando},
  journal={preprint arXiv:\allowbreak 1201.1575},
  year={2014}
}

@article{vazquez2021three,
  title={The Three F's for Bicategories I: Localization by Fractions is Exact},
  author={Bustillo Vazquez, P  and Pronk, D and Szyld, M},
  journal={arXiv preprint arXiv:\allowbreak 2112.00205},
  year={2021}
}

@article {garcia2CartesianFibrationsModel2021,
    AUTHOR = {Abell\'an Garc\'ia, Fernando and Stern, Walker H.},
     TITLE = {2-{C}artesian fibrations {I}: {A} model for {$\infty
              $}-bicategories fibred in {$\infty $}-bicategories},
   JOURNAL = {Appl. Categ. Structures},
  FJOURNAL = {Applied Categorical Structures. A Journal Devoted to
              Applications of Categorical Methods in Algebra, Analysis,
              Computer Science, Logic, Order and Topology},
    VOLUME = {30},
      YEAR = {2022},
    NUMBER = {6},
     PAGES = {1341--1392},
      ISSN = {0927-2852,1572-9095},
   MRCLASS = {18N45 (18N10 18N40)},
  MRNUMBER = {4519636},
       DOI = {10.1007/s10485-022-09693-x},
       URL = {https://doi-org.proxybmath.univ-lille.fr/10.1007/s10485-022-09693-x},
}

@article {abellan2CartesianFibrationsII2023,
    AUTHOR = {Abell\'an, Fernando and Stern, Walker H.},
     TITLE = {2-{C}artesian fibrations {II}: {A} {G}rothendieck construction
              for {$\infty$}-bicategories},
   JOURNAL = {J. Inst. Math. Jussieu},
  FJOURNAL = {Journal of the Institute of Mathematics of Jussieu. JIMJ.
              Journal de l'Institut de Math\'ematiques de Jussieu},
    VOLUME = {25},
      YEAR = {2026},
    NUMBER = {2},
     PAGES = {663--747},
      ISSN = {1474-7480,1475-3030},
   MRCLASS = {18N65 (18N40 18N45)},
  MRNUMBER = {5038573},
       DOI = {10.1017/S1474748025101436},
       URL = {https://doi-org.proxybmath.univ-lille.fr/10.1017/S1474748025101436},
}

@article {gagnaEquivalenceAllModels2022,
    AUTHOR = {Gagna, Andrea and Harpaz, Yonatan and Lanari, Edoardo},
     TITLE = {On the equivalence of all models for
              {$(\infty,2)$}-categories},
   JOURNAL = {J. Lond. Math. Soc. (2)},
  FJOURNAL = {Journal of the London Mathematical Society. Second Series},
    VOLUME = {106},
      YEAR = {2022},
    NUMBER = {3},
     PAGES = {1920--1982},
      ISSN = {0024-6107,1469-7750},
   MRCLASS = {18N65 (18C35 18N40 18N50 55U10 55U35)},
  MRNUMBER = {4498545},
MRREVIEWER = {Viktoriya\ Ozornova},
       DOI = {10.1112/jlms.12614},
       URL = {https://doi-org.proxybmath.univ-lille.fr/10.1112/jlms.12614},
}

@incollection {grothendieckdescent,
    AUTHOR = {Grothendieck, Alexander},
     TITLE = {Technique de descente et th\'eor\`emes d'existence en
              g\'eom\'etrie alg\'ebrique. {I}. {G}\'en\'eralit\'es.
              {D}escente par morphismes fid\`element plats},
 BOOKTITLE = {S\'eminaire {B}ourbaki, {V}ol. 5},
     PAGES = {Exp. No. 190, 299--327},
 PUBLISHER = {Soc. Math. France, Paris},
      YEAR = {1995},
      ISBN = {2-85629-038-8},
   MRCLASS = {14A15},
  MRNUMBER = {1603475},
}

@incollection {street1,
    AUTHOR = {Street, Ross},
     TITLE = {Fibrations and {Y}oneda's lemma in a {$2$}-category},
 BOOKTITLE = {Category {S}eminar ({P}roc. {S}em., {S}ydney, 1972/1973)},
    SERIES = {Lecture Notes in Math.},
    VOLUME = {Vol. 420},
     PAGES = {104--133},
 PUBLISHER = {Springer, Berlin-New York},
      YEAR = {1974},
   MRCLASS = {18D05},
  MRNUMBER = {396723},
MRREVIEWER = {A.\ Pultr},
}

@article {street2,
    AUTHOR = {Street, Ross},
     TITLE = {Fibrations in bicategories},
   JOURNAL = {Cahiers Topologie G\'eom. Diff\'erentielle},
  FJOURNAL = {Cahiers de Topologie et G\'eom\'etrie Diff\'erentielle},
    VOLUME = {21},
      YEAR = {1980},
    NUMBER = {2},
     PAGES = {111--160},
      ISSN = {0008-0004},
   MRCLASS = {18D05 (18D30)},
  MRNUMBER = {574662},
MRREVIEWER = {Timothy\ Porter},
}

@article {benabou,
    AUTHOR = {B\'enabou, Jean},
     TITLE = {Fibered categories and the foundations of naive category
              theory},
   JOURNAL = {J. Symbolic Logic},
  FJOURNAL = {The Journal of Symbolic Logic},
    VOLUME = {50},
      YEAR = {1985},
    NUMBER = {1},
     PAGES = {10--37},
      ISSN = {0022-4812,1943-5886},
   MRCLASS = {18A15 (03B30 03G30 18D30)},
  MRNUMBER = {780520},
MRREVIEWER = {Andreas\ Blass},
       DOI = {10.2307/2273784},
       URL = {https://doi.org/10.2307/2273784},
}

@incollection {graydescent,
    AUTHOR = {Gray, John W.},
     TITLE = {Fibred and cofibred categories},
 BOOKTITLE = {Proc. {C}onf. {C}ategorical {A}lgebra ({L}a {J}olla, {C}alif.,
              1965)},
     PAGES = {21--83},
 PUBLISHER = {Springer-Verlag New York, Inc., New York},
      YEAR = {1966},
   MRCLASS = {18.10 (14.00)},
  MRNUMBER = {213413},
MRREVIEWER = {A.\ Heller},
}

@article {giraud1,
    AUTHOR = {Giraud, Jean},
     TITLE = {M\'ethode de la descente},
   JOURNAL = {Bull. Soc. Math. France M\'em.},
  FJOURNAL = {Soci\'et\'e{} Math\'ematique de France. Bulletin. M\'emoire},
    VOLUME = {2},
      YEAR = {1964},
     PAGES = {viii+150},
      ISSN = {0583-8665},
   MRCLASS = {14.55},
  MRNUMBER = {190142},
MRREVIEWER = {R.\ Deheuvels},
}

@article {giraud2,
    AUTHOR = {Giraud, Jean},
     TITLE = {Cohomologie non ab\'elienne},
   JOURNAL = {C. R. Acad. Sci. Paris},
  FJOURNAL = {Comptes Rendus Hebdomadaires des S\'eances de l'Acad\'emie des
              Sciences},
    VOLUME = {260},
      YEAR = {1965},
     PAGES = {2666--2668},
      ISSN = {0001-4036},
   MRCLASS = {18.20 (14.55)},
  MRNUMBER = {201490},
MRREVIEWER = {R.\ Deheuvels},
}

@incollection {catlogic1,
    AUTHOR = {Gray, John W.},
     TITLE = {The categorical comprehension scheme},
 BOOKTITLE = {Category {T}heory, {H}omology {T}heory and their
              {A}pplications, {III} ({B}attelle {I}nstitute {C}onference,
              {S}eattle, {W}ash., 1968, {V}ol. {T}hree)},
    SERIES = {Lecture Notes in Math.},
    VOLUME = {No. 99},
     PAGES = {242--312},
 PUBLISHER = {Springer, Berlin-New York},
      YEAR = {1969},
   MRCLASS = {18.10},
  MRNUMBER = {249483},
MRREVIEWER = {D.\ Pumpl\"un},
}

@incollection {lawverelogic,
    AUTHOR = {Lawvere, F. William},
     TITLE = {Equality in hyperdoctrines and comprehension schema as an
              adjoint functor},
 BOOKTITLE = {Applications of {C}ategorical {A}lgebra ({P}roc. {S}ympos.
              {P}ure {M}ath., {V}ol. {XVII}, {N}ew {Y}ork, 1968)},
    SERIES = {Proc. Sympos. Pure Math.},
    VOLUME = {XVII},
     PAGES = {1--14},
 PUBLISHER = {Amer. Math. Soc., Providence, RI},
      YEAR = {1970},
   MRCLASS = {18.10},
  MRNUMBER = {257175},
MRREVIEWER = {H.\ Gonshor},
}

@incollection {catlogic2,
    AUTHOR = {Par\'e, Robert and Schumacher, Dietmar},
     TITLE = {Abstract families and the adjoint functor theorems},
 BOOKTITLE = {Indexed categories and their applications},
    SERIES = {Lecture Notes in Math.},
    VOLUME = {661},
     PAGES = {1--125},
 PUBLISHER = {Springer, Berlin-New York},
      YEAR = {1978},
      ISBN = {3-540-08914-4},
   MRCLASS = {18A40},
  MRNUMBER = {514193},
MRREVIEWER = {George\ D.\ Reynolds},
}

@book {typetheory,
    AUTHOR = {Jacobs, Bart},
     TITLE = {Categorical logic and type theory},
    SERIES = {Studies in Logic and the Foundations of Mathematics},
    VOLUME = {141},
 PUBLISHER = {North-Holland Publishing Co., Amsterdam},
      YEAR = {1999},
     PAGES = {xviii+760},
      ISBN = {0-444-50170-3},
   MRCLASS = {03G30 (03-02 03B15 03B40)},
  MRNUMBER = {1674451},
MRREVIEWER = {Andreas\ Blass},
}

@book{grothendieck,
  title={Theorie des Intersections et Theoreme de Riemann-Roch: Seminaire de Geometrie Algebrique du Bois Marie 1966/67 (SGA 6)},
  author={Grothendieck, A and Ferrand, D and Jouanolou, Jean-Pierre and Jussila, O and Kleiman, S and Raynaud, M and Serre, Jean-Pierre},
  year={2006},
  publisher={Springer}
}

@unpublished{grothcat1,
      title={Grothendieck construction for bicategories}, 
      author={Igor Bakovic},
      year={2009},
      url  = {http://www.irb.hr/korisnici/ibakovic/sgc.pdf},
      note = {\url{http://www.irb.hr/korisnici/ibakovic/sgc.pdf}},
}

@article {grothcat3,
    AUTHOR = {Hermida, Claudio},
     TITLE = {Some properties of {${\bf Fib}$} as a fibred {$2$}-category},
   JOURNAL = {J. Pure Appl. Algebra},
  FJOURNAL = {Journal of Pure and Applied Algebra},
    VOLUME = {134},
      YEAR = {1999},
    NUMBER = {1},
     PAGES = {83--109},
      ISSN = {0022-4049},
   MRCLASS = {18D20 (03B15 18A15)},
  MRNUMBER = {1665258},
MRREVIEWER = {Martin Hofmann},
       DOI = {10.1016/S0022-4049(97)00129-1},
       URL = {https://doi-org.proxy1.library.jhu.edu/10.1016/S0022-4049(97)00129-1},
}

@article {graycats,
    AUTHOR = {Gordon, R. and Power, A. J. and Street, Ross},
     TITLE = {Coherence for tricategories},
   JOURNAL = {Mem. Amer. Math. Soc.},
  FJOURNAL = {Memoirs of the American Mathematical Society},
    VOLUME = {117},
      YEAR = {1995},
    NUMBER = {558},
     PAGES = {vi+81},
      ISSN = {0065-9266,1947-6221},
   MRCLASS = {18D05},
  MRNUMBER = {1261589},
MRREVIEWER = {Kimmo\ I.\ Rosenthal},
       DOI = {10.1090/memo/0558},
       URL = {https://doi.org/10.1090/memo/0558},
}

@article {local1,
    AUTHOR = {Descotte, M. E. and Dubuc, E. J. and Szyld, M.},
     TITLE = {Model bicategories and their homotopy bicategories},
   JOURNAL = {Adv. Math.},
  FJOURNAL = {Advances in Mathematics},
    VOLUME = {404},
      YEAR = {2022},
     PAGES = {Paper No. 108455, 56},
      ISSN = {0001-8708,1090-2082},
   MRCLASS = {18N10 (18N45 55P60 55U40)},
  MRNUMBER = {4423809},
MRREVIEWER = {Niall\ Taggart},
       DOI = {10.1016/j.aim.2022.108455},
       URL = {https://doi.org/10.1016/j.aim.2022.108455},
}

@article {local3,
    AUTHOR = {Renaudin, Olivier},
     TITLE = {Plongement de certaines th\'eories homotopiques de {Q}uillen
              dans les d\'erivateurs},
   JOURNAL = {J. Pure Appl. Algebra},
  FJOURNAL = {Journal of Pure and Applied Algebra},
    VOLUME = {213},
      YEAR = {2009},
    NUMBER = {10},
     PAGES = {1916--1935},
      ISSN = {0022-4049,1873-1376},
   MRCLASS = {18G55 (55U35)},
  MRNUMBER = {2526867},
MRREVIEWER = {Philippe\ Gaucher},
       DOI = {10.1016/j.jpaa.2009.02.014},
       URL = {https://doi.org/10.1016/j.jpaa.2009.02.014},
}

@article {dwyerkan,
    AUTHOR = {Dwyer, W. G. and Kan, D. M.},
     TITLE = {Function complexes in homotopical algebra},
   JOURNAL = {Topology},
  FJOURNAL = {Topology. An International Journal of Mathematics},
    VOLUME = {19},
      YEAR = {1980},
    NUMBER = {4},
     PAGES = {427--440},
      ISSN = {0040-9383},
   MRCLASS = {55U35 (55U10)},
  MRNUMBER = {584566},
MRREVIEWER = {Edgar\ H.\ Brown, Jr.},
       DOI = {10.1016/0040-9383(80)90025-7},
       URL = {https://doi-org.ezp1.lib.umn.edu/10.1016/0040-9383(80)90025-7},
}
